\documentclass[final,10pt]{elsarticle}%
\usepackage{lineno,hyperref}
\usepackage{epsfig,amssymb,amsmath,version,amsthm}
\usepackage{amssymb,version,graphicx,fancybox,mathrsfs,multirow}
\usepackage[notcite,notref]{showkeys}
\usepackage{url,hyperref}
\usepackage{subfigure,epstopdf,bm}
\usepackage{color}
\usepackage{accents}
\usepackage{appendix}
\usepackage{algorithm}
\usepackage{algorithmicx}
\usepackage{algpseudocode}
\usepackage{amsmath}
\usepackage{amsfonts}
\usepackage{amssymb}
\usepackage{graphicx}%
\modulolinenumbers[5]
\journal{Computer Methods in Applied Mechanics and Engineering}
\catcode`\@=11 \theoremstyle{plain}
\@addtoreset{equation}{section}

\@addtoreset{figure}{section}
\renewcommand\thefigure{\thesection.\@arabic\c@figure}
\@addtoreset{table}{section}
\renewcommand\thetable{\thesection.\@arabic\c@table}
\floatname{algorithm}{Algorithm}
\newtheorem{thm}{\bf Theorem}

\newenvironment{theorem}{\begin{thm}} {\end{thm}}
\newtheorem{cor}{\bf Corollary}

\newenvironment{corollary}{\begin{cor}} {\end{cor}}
\newtheorem{lmm}{\bf Lemma}

\theoremstyle{remark}
\newtheorem{rem}{Remark}[section]
\def \ri {{\rm i}}
\newcommand{\bs}[1]{\boldsymbol{#1}}
\begin{document}
\begin{frontmatter}
	\title{Taylor expansion based fast Multipole Methods for 3-D Helmholtz equations in Layered Media}
	
	\author[add1,add2]{Bo Wang}
	\author[add3]{Duan Chen}
	\author[add4]{Bo Zhang}
	\author[add2]{Wenzhong Zhang}
	\author[add5]{Min Hyung Cho}
	\author[add2]{Wei Cai\corref{mycorrespondingauthor}}
	\cortext[mycorrespondingauthor]{Corresponding author}
	\ead{cai@smu.edu}
	
	\address[add1]{LCSM, Ministry of Education, School of Mathematics and Statistics, Hunan Normal University, Changsha, Hunan 410081, P. R. China}
	\address[add2]{Department of Mathematics, Southern Methodist University, Dallas, TX 75275, USA}
	\address[add3]{Department of Mathematics and Statistics, University of North Carolina at Charlotte, Charlotte, NC 28223, USA}
	\address[add4]{Department of Computer Science, Indiana University, IN 47408, USA}
	\address[add5]{Department of Mathematical Science, University of Massachusetts Lowell, Lowell, MA 01854, USA}
	\begin{abstract}
		In this paper, we develop fast multipole methods for 3D Helmholtz kernel in
		layered media. Two algorithms based on different forms of Taylor expansion of layered
		media Green's function are developed. A key component of the first algorithm is an efficient algorithm based on discrete complex image approximation and recurrence formula for the
		calculation of the layered media Green's function and its derivatives, which
		are given in terms of Sommerfeld integrals. The second algorithm uses symmetric derivatives in the Taylor expansion to reduce the size of precomputed tables for the derivatives of layered media Green's function. Numerical tests in layered media have validated the accuracy and $O(N)$ complexity of the proposed algorithms.
	\end{abstract}
	
	\begin{keyword}
		Fast multipole method, layered media,  Helmholtz equation, Taylor expansion
	\end{keyword}
	
\end{frontmatter}


\section{Introduction}

Wave scattering of objects embedded in layered media can be computed by
integral equation (IE) methods using domain Green's functions which satisfy
the transmission conditions at material layer interfaces and the Sommerfeld
radiation condition at infinity (cf.
\cite{michalski1990electromagnetic,millard2003fast,
chen2018accurate,lai2014fast}). As a result, IE methods based on domain
Green's functions only require solution unknowns to be given on the
scatterer's surface for a surface IE formulation or over the scatterer body
for a volume IE. This is different from formulations based on free space
Green's functions, which will need additional unknowns on the infinite
material layer interfaces (cf.
\cite{cho2018spectrally,cho2015robust,lai2015fast}). For the solution of the
resulting linear system from the discretized IEs, iterative solvers such as
GMRES are usually used, which require the product of a full matrix, from the
discretization of the integral operator, and a solution vector. A direct
product will generate an $O(N^{2})$ cost where $N$ is the size of the matrix.
Therefore, the main computational issue for IE methods is to develop fast
solvers to speed up such a matrix-vector product. The popular fast method is
the fast multipole method (FMM) developed by Greengard and Rokhlin using
multipole expansions for the free space Green's functions
\cite{greengard1987fast,greengard1997new}. However, extending FMMs to layered
media has been a long outstanding challenge for IE methods.

Numerical algorithms for layered-media problems have traditionally been
carried out in the Fourier spectral domain due to the availability of the
closed-form Green's functions (CFGF's) for layered media in the spectral
domain. Since a series of techniques have been developed to obtain
approximated CFGF's for layered media in the spatial domain (cf.
\cite{chow1991closed,aksun1996robust, alparslan2010closed}), extensions of
FMMs to layered media problems were proposed by applying spherical harmonic
expansions to the approximated CFGF's (see
\cite{jandhyala1995multipole,gurel1996electromagnetic, geng2001fast} for
Laplace, Helmholtz and Maxwell's equations, respectively). For the Laplace
equation, only real images were used in the approximating the CFGF's and
traditional FMMs can be then applied to the approximated CFGF's, directly.
However, for Helmholtz and Maxwell's equations in layered media, complex
images are required to obtain approximated CFGF's. Thus, addition theorems
used for the free space FMMs need \ to be modified for wave functions with
complex arguments, so far no rigorous mathematical formulations and numerical
implementation have been obtained. Other efforts to speed up the computation of integral
operator for layered media Green's functions include the inhomogeneous plane
wave method \cite{huchew2000}, windowed Green's function method for
layered-media \cite{bruno2016windowed}, and cylindrical wave decomposition of
the Green's function in 3-D and 2-D FMM \cite{cho2012}.

In this paper, we will develop FMM methods for the 3D Helmholtz equation for
layered media based on Taylor expansion instead of the mutltipole expansion.
In addition to the original FMMs \cite{greengard1987fast, greengard1997new}
using spherical harmonic expansion, FMMs based on Taylor expansion (TE) have
already been developed and investigated for free space Green's functions and
other kernels \cite{tausch2004variable,
ying2004kernel,fong2009black,darve2004efficient} and have been shown to have
similar error estimate as multipole expansion using spherical harmonic
expansion \cite{tausch2004variable}. Since analytical form of the Green's
functions in layered media usually can only be obtained in the spectral domain
using Sommerfeld integrals, it will be more convenient to develop FMMs based
on Taylor expansions in multi-layer media.

We will start with the derivation of analytical form of the Green's function
in spectral domain for Helmholtz equations in multi-layer media. Two different
versions of Taylor expansion based FMMs (TE-FMMs) will be proposed to compute
the interaction between sources and targets located in different layers. The
two versions come from using Taylor expansions with nonsymmetric derivatives
or symmetric derivatives, respectively. For the two algorithms, different
statigies are introduced for efficient computation of the translation operator
from far field expansion centered at a source box to local expansion centered
at a target box. In the case of Taylor expansion with nonsymmetric
derivatives, we propose an efficient and low memory algorithm based on
discrete complex image method (DCIM) approximation of the Green's functions in
the spectral domain together with recurrence formulas for derivatives of free
space Green's function. Meanwhile, for the case of Taylor expansion with
symmetric derivatives, precomputed tables for the translation operators will
be used, instead. With these Taylor expansion based FMMs, fast computation is
achieved for interactions among particles in multi-layer media, as shown in
numerical examples for two layers and three layers cases.

The rest of this paper is organized as follows. In section 2, a general
formulation for Green's function of Helmholtz equation in multi-layer media
are derived. Unlike the derivation presented in \cite{cho2012parallel}, the
derivation here shows source and target information in separate parts of the
formulas for general multi-layered media. In section 3, the first version
TE-FMM using non-symmetric derivatives is proposed for multi-layer media.
Using DCIM approximation and recurrence formulas for derivatives, a fast
algorithm for the computation of the translation operator from a TE in source
box to a TE in target box is given. Then, fast algorithm for interactions
among particles in multi-layer media is presented. The second TE-FMM using
symmetric derivatives is developed in Section 4. There, we first introduce the
TE-FMM using symmetric derivatives for the free space case, and then extend it
to the case of multi-layer media. Numerical results using both versions of the
TE-FMMs are given for two and three layers media in Section 5. Various
efficiency comparison results are given to show the performance of the
proposed TE-FMMs.

\section{Spectral form of Green's function in multi-layer media}

In this section, we briefly summarize the derivation of the Green's function
of Helmholtz equations in multi-layer media \cite{cho2012parallel} with source
and target coordinates separated in the Fourier spectral form.

\subsection{General formula}

Consider a layered medium consisting of $L$-interfaces located at $z=d_{\ell
},\ell=0,1,\cdots,L-1$ in Fig. \ref{layerstructure}. Suppose we have a point
source at $\boldsymbol{r}^{\prime}=(x^{\prime},y^{\prime},z^{\prime})$ in the
$\ell^{\prime}$th layer ($d_{\ell^{\prime}}<z^{\prime}<d_{\ell^{\prime}-1}$).
Then, the layered media Green's function for the Helmholtz equation satisfies
\begin{equation}
\boldsymbol{\Delta}u_{\ell\ell^{\prime}}(\boldsymbol{r},\boldsymbol{r}^{\prime
})+k_{\ell}^{2}u_{\ell\ell^{\prime}}(\boldsymbol{r},\boldsymbol{r}^{\prime
})=-\delta(\boldsymbol{r},\boldsymbol{r}^{\prime}),
\end{equation}
at field point $\boldsymbol{r}=(x,y,z)$ in the $\ell$th layer ($d_{\ell
}<z<d_{\ell}-1$) where $\delta(\boldsymbol{r},\boldsymbol{r}^{\prime})$ is the
Dirac delta function and $k_{\ell}$ is the wave number in the $\ell$th layer.
\begin{figure}[ptbh]
\label{layerstructure} \center
\includegraphics[scale=0.7]{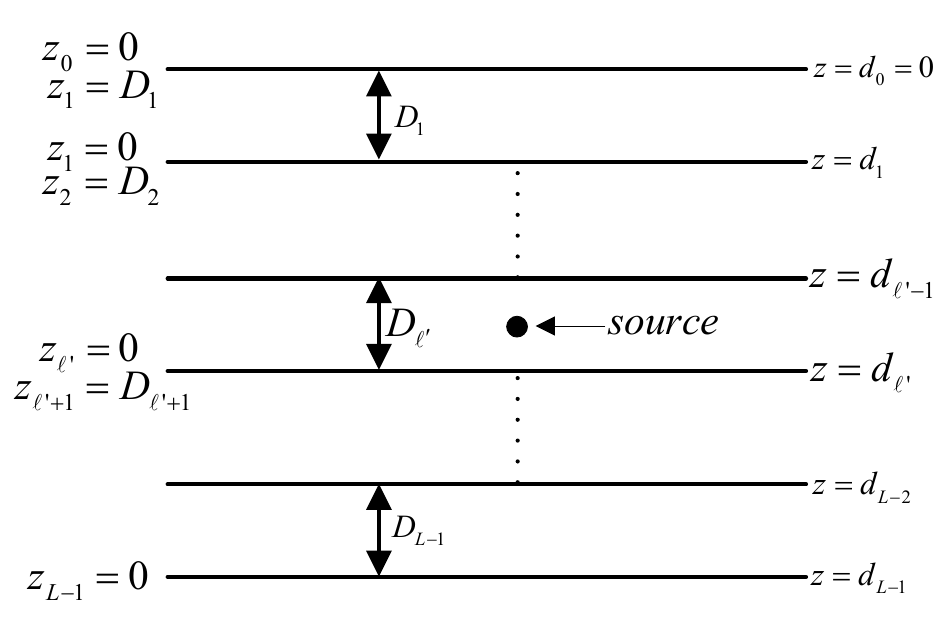}\caption{Sketch of the layer
structure for general multi-layer media.}%
\end{figure}
Define the partial Fourier transform along $x-$ and $y-$directions
for $u_{\ell\ell^{\prime}}(x,y,z)$ as
\[
\widehat{u}_{\ell\ell^{\prime}}(k_{x},k_{y}%
,z)=\mathscr{F}[u_{\ell\ell'}(\bs r, \bs r')](k_{x},k_{y},z):=\int_{-\infty
}^{\infty}\int_{-\infty}^{\infty}u_{\ell\ell^{\prime}}%
(\boldsymbol{r},\boldsymbol{r}^{\prime})e^{-\ri(k_{x}x+k_{y}y)}dxdy.
\]
Then, $\widehat{u}_{\ell\ell^{\prime}}(k_{x},k_{y},z)$ satisfies second order
ordinary differential equations
\[
\frac{d^{2}\widehat{u}_{\ell\ell^{\prime}}(k_{x},k_{y},z)}{dz^{2}}+(k_{\ell
}^{2}-k_{\rho}^{2})\widehat{u}_{\ell\ell^{\prime}}(k_{x},k_{y}%
,z)=-e^{-\ri(k_{x}x^{\prime}+k_{y}y^{\prime})}\delta(z,z^{\prime}),
\]
where $k_{\rho}^{2}=k_{x}^{2}+k_{y}^{2}$. The system of ordinary differential
equations can be solved analytically for each layer in $z$ by imposing
transmission conditions at the interface between $\ell$th and $(\ell-1)$th
layer ($z=d_{\ell-1})$, i.e.,
\[
\widehat{u}_{\ell-1,\ell^{\prime}}(k_{x},k_{y},z)=\widehat{u}_{\ell
\ell^{\prime}}(k_{x},k_{y},z),\quad k_{\ell-1}\frac{d\widehat{u}_{\ell
-1,\ell^{\prime}}(k_{x},k_{y},z)}{dz}=k_{\ell}\frac{d\widehat{u}_{\ell
\ell^{\prime}}(k_{x},k_{y},z)}{dz},
\]
as well as decay conditions in the top and bottom-most layers for
$z\rightarrow\pm\infty$. Generally, an analytic solution has the following
form
\begin{equation}%
\begin{cases}
\displaystyle\widehat{u}_{\ell\ell^{\prime}}(k_{x},k_{y},z)=A_{\ell
\ell^{\prime}}\cosh(\ri k_{\ell z}z_{\ell})+B_{\ell\ell^{\prime}}\sinh(\ri
k_{\ell z}z_{\ell}),\quad\ell\neq\ell^{\prime},L,\\[10pt]%
\displaystyle\widehat{u}_{\ell^{\prime}\ell^{\prime}}(k_{x},k_{y}%
,z)=A_{\ell^{\prime}\ell^{\prime}}\cosh(\ri k_{\ell^{\prime}z}z_{\ell^{\prime
}})+B_{\ell^{\prime}\ell^{\prime}}\sinh(\ri k_{\ell^{\prime}z}z_{\ell^{\prime
}})+\widehat{G}(k_{\ell^{\prime}z},z-z^{\prime}),\\[10pt]%
\displaystyle\widehat{u}_{L\ell^{\prime}}(k_{x},k_{y},z)=A_{L\ell^{\prime}%
}\cosh(\ri k_{\ell z}z)+B_{L\ell^{\prime}}\sinh(\ri k_{\ell z}z),
\end{cases}
\label{solutionformula}%
\end{equation}
where
\[
\widehat{G}(k_{\ell^{\prime}z},z-z^{\prime})=\vartheta\frac{e^{\ri
k_{\ell^{\prime}z}|z-z^{\prime}|}}{k_{\ell^{\prime}z}},\quad\mathrm{with}%
\;\;\vartheta=\frac{\ri e^{-\ri(k_{x}x^{\prime}+k_{y}y^{\prime})}}{2}%
\]
is the Fourier transform of free space Green's function with wave number
$k_{\ell^{\prime}}$. Notations $k_{\ell z}=\sqrt{k_{\ell}^{2}-k_{\rho}^{2}}$,
$k_{\ell^{\prime}z}=\sqrt{k_{\ell^{\prime}}^{2}-k_{\rho}^{2}}$ and local
coordinate $z_{\ell}:=z-d_{\ell}$ are used here (see Fig. \ref{layerstructure}%
). The interface conditions implies that coefficients in the $\ell$th layer
$V_{\ell\ell^{\prime}}=(A_{\ell\ell^{\prime}},B_{\ell\ell^{\prime}%
})^{\mathrm{T}}$ can be recursively determined as follow
\begin{equation}%
\begin{split}
&  V_{\ell\ell^{\prime}}=\prod\limits_{k=\ell+1}^{L}\mathbb{T}_{k-1,k}%
V_{L},\quad\ell^{\prime}<\ell<L,\quad V_{\ell^{\prime}\ell^{\prime}}%
=\prod\limits_{k=\ell^{\prime}+1}^{L}\mathbb{T}_{k-1,k}V_{L\ell^{\prime}%
}+\boldsymbol{S}_{\ell^{\prime},\ell^{\prime}+1},\\
&  V_{\ell^{\prime}-1,\ell^{\prime}}=\prod\limits_{k=\ell^{\prime}}%
^{L}\mathbb{T}_{k-1,k}V_{L}+\boldsymbol{S}_{\ell^{\prime}-1,\ell^{\prime}%
}+\mathbb{T}_{\ell^{\prime}-1,\ell^{\prime}}\boldsymbol{S}_{\ell^{\prime}%
,\ell^{\prime}+1},\\
&  V_{\ell\ell^{\prime}}=\prod\limits_{k=\ell+1}^{\ell^{\prime}-1}%
\mathbb{T}_{k-1,k}\Big(\prod\limits_{k=\ell^{\prime}}^{L}\mathbb{T}%
_{k-1,k}V_{L\ell^{\prime}}+\mathbb{T}_{\ell^{\prime}-1,\ell^{\prime}%
}\boldsymbol{S}_{\ell^{\prime},\ell^{\prime}+1}+\boldsymbol{S}_{\ell^{\prime
}-1,\ell^{\prime}}\Big),\quad0<\ell<\ell^{\prime}-1,
\end{split}
\label{recursion}%
\end{equation}
where the transfer matrix $_{\ell-1,\ell}$ are given by
\[%
\begin{split}
\mathbb T_{\ell-1,\ell}= &
\begin{pmatrix}
\displaystyle\cosh(\ri k_{\ell z}D_{\ell}) & \displaystyle\sinh(\ri k_{\ell
z}D_{\ell})\\[6pt]%
\displaystyle\frac{k_{\ell}k_{\ell z}\sinh(\ri k_{\ell z}D_{\ell})}{k_{\ell
-1}k_{\ell-1,z}} & \displaystyle\frac{k_{\ell}k_{\ell z}\cosh(\ri k_{\ell
z}D_{\ell})}{k_{\ell-1}k_{\ell-1,z}}%
\end{pmatrix}
,\quad\ell=1,2,\cdots,L-1,\\
\mathbb T_{L-1,L}= &
\begin{pmatrix}
\displaystyle\cosh(\ri k_{Lz}d_{L-1}) & \displaystyle\sinh(\ri k_{Lz}%
d_{L-1})\\[6pt]%
\displaystyle\frac{k_{N}k_{Lz}\sinh(\ri k_{Lz}d_{L-1})}{k_{L-1}k_{L-1,z}} &
\displaystyle\frac{k_{L}k_{Lz}\cosh(\ri k_{Lz}d_{L-1})}{k_{L-1}k_{L-1,z}}%
\end{pmatrix}
,
\end{split}
\]
and source vectors are defined as follows:
\begin{equation}
\boldsymbol{S}_{\ell^{\prime}-1,\ell^{\prime}}=%
\begin{pmatrix}
\displaystyle1\\
\displaystyle\frac{k_{\ell^{\prime}}k_{\ell^{\prime}z}}{k_{\ell^{\prime}%
-1}k_{\ell^{\prime}-1,z}}%
\end{pmatrix}
\frac{\vartheta e^{\ri k_{\ell^{\prime}z}(d_{\ell^{\prime}-1}-z^{\prime})}%
}{k_{\ell^{\prime}z}},\quad\boldsymbol{S}_{\ell^{\prime},\ell^{\prime}%
+1}=\left(
\begin{array}
[c]{r}%
\displaystyle-1\\
\displaystyle1
\end{array}
\right)  \frac{\vartheta e^{\ri k_{\ell^{\prime}z}(z^{\prime}-d_{\ell^{\prime
}})}}{k_{\ell^{\prime}z}}.
\end{equation}
The decaying conditions on the top and bottom layers yield initial values for
the recursion \eqref{recursion}
\begin{equation}
A_{0\ell^{\prime}}=B_{0\ell^{\prime}},\quad A_{L\ell^{\prime}}=-B_{L\ell
^{\prime}}.
\end{equation}
Therefore, the system of algebraic equations between $V_{0\ell^{\prime}}$ and
$V_{L\ell^{\prime}}$ can be found from \eqref{recursion} as
\begin{equation}%
\begin{split}%
\begin{pmatrix}
A_{0\ell^{\prime}}\\
A_{0\ell^{\prime}}%
\end{pmatrix}
= &  \prod\limits_{k=1}^{L}\mathbb{T}_{k-1,k}V_{L\ell^{\prime}}+\prod
\limits_{k=1}^{\ell^{\prime}-1}\mathbb{T}_{k-1,k}\boldsymbol{S}_{\ell^{\prime
}-1,\ell^{\prime}}+\prod\limits_{k=1}^{\ell^{\prime}}\mathbb{T}_{k-1,k}%
\boldsymbol{S}_{\ell^{\prime},\ell^{\prime}+1}\\
= &
\begin{pmatrix}
\alpha_{11} & \alpha_{12}\\
\alpha_{21} & \alpha_{22}%
\end{pmatrix}
\left(
\begin{array}
[c]{r}%
A_{L\ell^{\prime}}\\
-A_{L\ell^{\prime}}%
\end{array}
\right)  +%
\begin{pmatrix}
\beta_{11}\\
\beta_{21}%
\end{pmatrix}
\vartheta e^{\ri k_{\ell^{\prime}z}(d_{\ell^{\prime}-1}-z^{\prime})}+%
\begin{pmatrix}
\beta_{12}\\
\beta_{22}%
\end{pmatrix}
\vartheta e^{\ri k_{\ell^{\prime}z}(z^{\prime}-d_{\ell^{\prime}})}.
\end{split}
\end{equation}
It is important to point out that $\alpha_{ij}$ and $\beta_{ij}$ are
independent of the source location $(x^{\prime},y^{\prime},z^{\prime})$, which
only depend on $\{k_{\ell},k_{\ell z}\}_{\ell=0}^{L}$ and $\{D_{\ell}%
\}_{\ell=1}^{L-1}$. Therefore
\begin{equation}%
\begin{split}
A_{0\ell^{\prime}}=B_{0\ell^{\prime}}= &  \frac{[(\alpha_{22}-\alpha
_{21})\beta_{11}+(\alpha_{11}-\alpha_{12})\beta_{21}]}{[(\alpha_{11}%
-\alpha_{12})-(\alpha_{21}-\alpha_{22})]}\frac{\ri e^{-\ri(k_{x}x^{\prime
}+k_{y}y^{\prime})}}{2}\frac{e^{\ri k_{\ell^{\prime}z}(d_{\ell^{\prime}%
-1}-z^{\prime})}}{k_{\ell^{\prime}z}}\\
&  +\frac{[(\alpha_{22}-\alpha_{21})\beta_{12}+(\alpha_{11}-\alpha_{12}%
)\beta_{22}]}{[(\alpha_{11}-\alpha_{12})-(\alpha_{21}-\alpha_{22})]}\frac{\ri
e^{-\ri(k_{x}x^{\prime}+k_{y}y^{\prime})}}{2}\frac{e^{\ri k_{\ell^{\prime}%
z}(z^{\prime}-d_{\ell^{\prime}})}}{k_{\ell^{\prime}z}},\\
A_{L\ell^{\prime}}=-B_{L\ell^{\prime}}= &  \frac{\beta_{21}-\beta_{11}%
}{[(\alpha_{11}-\alpha_{12})-(\alpha_{21}-\alpha_{22})]}\frac{\ri
e^{-\ri(k_{x}x^{\prime}+k_{y}y^{\prime})}}{2}\frac{e^{\ri k_{\ell^{\prime}%
z}(d_{\ell^{\prime}-1}-z^{\prime})}}{k_{\ell^{\prime}z}}\\
&  +\frac{\beta_{22}-\beta_{12}}{[(\alpha_{11}-\alpha_{12})-(\alpha
_{21}-\alpha_{22})]}\frac{\ri e^{-\ri(k_{x}x^{\prime}+k_{y}y^{\prime})}}%
{2}\frac{e^{\ri k_{\ell^{\prime}z}(z^{\prime}-d_{\ell^{\prime}})}}%
{k_{\ell^{\prime}z}}.
\end{split}
\label{coefficientsolver}%
\end{equation}
Together with recursions \eqref{recursion}, any coefficients $\{A_{\ell
\ell^{\prime}},B_{\ell\ell^{\prime}}\}_{\ell=0}^{L}$ can be represented by
\begin{equation}%
\begin{split}
&  A_{\ell\ell^{\prime}}=\big(A_{\ell\ell^{\prime}}^{1}e^{\ri k_{\ell^{\prime
}z}(d_{\ell^{\prime}-1}-z^{\prime})}+A_{\ell\ell^{\prime}}^{2}e^{\ri
k_{\ell^{\prime}z}(z^{\prime}-d_{\ell^{\prime}})}\big)\frac{\ri e^{-\ri(k_{x}%
x^{\prime}+k_{y}y^{\prime})}}{2k_{\ell^{\prime}z}},\\
&  B_{\ell\ell^{\prime}}=\big(B_{\ell\ell^{\prime}}^{1}e^{\ri k_{\ell^{\prime
}z}(d_{\ell^{\prime}-1}-z^{\prime})}+B_{\ell\ell^{\prime}}^{2}e^{\ri
k_{\ell^{\prime}z}(z^{\prime}-d_{\ell^{\prime}})}\big)\frac{\ri e^{-\ri(k_{x}%
x^{\prime}+k_{y}y^{\prime})}}{2k_{\ell^{\prime}z}},
\end{split}
\end{equation}
where coefficients $\{A_{\ell\ell^{\prime}}^{1},A_{\ell\ell^{\prime}}%
^{2};B_{\ell\ell^{\prime}}^{1},B_{\ell\ell^{\prime}}^{2}\}_{\ell,\ell^{\prime
}=0}^{L}$ only depend on $\{k_{\ell},k_{\ell z}\}_{\ell=0}^{L}$ and
$\{D_{\ell}\}_{\ell=1}^{L-1}$.

Expressions given by \eqref{solutionformula} have upgoing and downgoing wave
mixed. It is usually more convenient to rewrite those as upgoing and downgoing
components
\begin{equation}%
\begin{cases}
\displaystyle\widehat{u}_{0\ell^{\prime}}(k_{x},k_{y},z)=b_{0\ell^{\prime}%
}\vartheta\frac{e^{\ri k_{0z}z}}{k_{0z}},\\[10pt]%
\displaystyle\widehat{u}_{\ell\ell^{\prime}}(k_{x},k_{y},z)=a_{\ell
\ell^{\prime}}\vartheta\frac{e^{-\ri k_{\ell z}(z-d_{\ell})}}{k_{\ell z}%
}+b_{\ell\ell^{\prime}}\vartheta\frac{e^{\ri k_{\ell z}(z-d_{\ell})}}{k_{\ell
z}},\quad i\neq0,\ell^{\prime},L,\\[10pt]%
\displaystyle\widehat{u}_{\ell^{\prime}\ell^{\prime}}(k_{x},k_{y}%
,z)=a_{\ell^{\prime}\ell^{\prime}}\vartheta\frac{e^{-\ri k_{\ell^{\prime}%
z}(z-d_{\ell^{\prime}})}}{k_{\ell z}}+b_{\ell^{\prime}\ell^{\prime}}%
\vartheta\frac{e^{\ri k_{\ell^{\prime}z}(z-d_{\ell^{\prime}})}}{k_{\ell z}%
}+\widehat{G}(k_{\ell^{\prime}z},z-z^{\prime}),\\[10pt]%
\displaystyle\widehat{u}_{L\ell^{\prime}}(k_{x},k_{y},z)=a_{L\ell^{\prime}%
}\vartheta\frac{e^{-\ri k_{Lz}z}}{k_{Lz}},
\end{cases}
\label{updownspectraldoamin}%
\end{equation}
where
\begin{equation}%
\begin{split}
a_{\ell\ell^{\prime}}= &  \frac{k_{\ell z}}{2k_{\ell^{\prime}z}}(A_{\ell
\ell^{\prime}}^{1}-B_{\ell\ell^{\prime}}^{1})e^{\ri k_{\ell^{\prime}z}%
(d_{\ell^{\prime}-1}-z^{\prime})}+\frac{k_{\ell z}}{2k_{\ell^{\prime}z}%
}(A_{\ell\ell^{\prime}}^{2}-B_{\ell\ell^{\prime}}^{2})e^{\ri k_{\ell^{\prime
}z}(z^{\prime}-d_{\ell^{\prime}})},\\
b_{\ell\ell^{\prime}}= &  \frac{k_{\ell z}}{2k_{\ell^{\prime}z}}(A_{\ell
\ell^{\prime}}^{1}+B_{\ell\ell^{\prime}}^{1})e^{\ri k_{\ell^{\prime}z}%
(d_{\ell^{\prime}-1}-z^{\prime})}+\frac{k_{\ell z}}{2k_{\ell^{\prime}z}%
}(A_{\ell\ell^{\prime}}^{2}+B_{\ell\ell^{\prime}}^{2})e^{\ri k_{\ell^{\prime
}z}(z^{\prime}-d_{\ell^{\prime}})}.
\end{split}
\end{equation}
It is important to note that $z^{\prime}$ only appears in the exponentials. In
general, these can be written in the form of
\begin{equation}%
\begin{split}
a_{\ell\ell^{\prime}}= &  \sigma_{\ell\ell^{\prime}}^{\downarrow\downarrow
}(k_{\rho})e^{\ri k_{\ell^{\prime}z}(d_{\ell^{\prime}-1}-z^{\prime})}%
+\sigma_{\ell\ell^{\prime}}^{\downarrow\uparrow}(k_{\rho})e^{\ri
k_{\ell^{\prime}z}(z^{\prime}-d_{\ell^{\prime}})},\\
b_{\ell\ell^{\prime}}= &  \sigma_{\ell\ell^{\prime}}^{\uparrow\downarrow
}(k_{\rho})e^{\ri k_{\ell^{\prime}z}(d_{\ell^{\prime}-1}-z^{\prime})}%
+\sigma_{\ell\ell^{\prime}}^{\uparrow\uparrow}(k_{\rho})e^{\ri k_{\ell
^{\prime}z}(z^{\prime}-d_{\ell^{\prime}})},
\end{split}
\end{equation}
where
\begin{equation}%
\begin{split}
\sigma_{\ell\ell^{\prime}}^{\downarrow\downarrow}(k_{\rho}) &  =\frac{k_{\ell
z}}{2k_{\ell^{\prime}z}}(A_{\ell\ell^{\prime}}^{1}-B_{\ell\ell^{\prime}}%
^{1}),\quad\sigma_{\ell\ell^{\prime}}^{\downarrow\uparrow}(k_{\rho}%
)=\frac{k_{\ell z}}{2k_{\ell^{\prime}z}}(A_{\ell\ell^{\prime}}^{2}-B_{\ell
\ell^{\prime}}^{2}),\\
\sigma_{\ell\ell^{\prime}}^{\uparrow\downarrow}(k_{\rho}) &  =\frac{k_{\ell
z}}{2k_{\ell^{\prime}z}}(A_{\ell\ell^{\prime}}^{1}+B_{\ell\ell^{\prime}}%
^{1}),\quad\sigma_{\ell\ell^{\prime}}^{\uparrow\uparrow}(k_{\rho}%
)=\frac{k_{\ell z}}{2k_{\ell^{\prime}z}}(A_{\ell\ell^{\prime}}^{2}+B_{\ell
\ell^{\prime}}^{2}).
\end{split}
\end{equation}
Therefore, taking inverse Fourier transform in \eqref{updownspectraldoamin}
gives expression of Green's function in the physical domain using Sommerfeld
integrals as follows:
\begin{equation}%
\begin{cases}
\displaystyle u_{\ell\ell^{\prime}}^{\uparrow}%
(\boldsymbol{r},\boldsymbol{r}^{\prime})=\frac{\ri}{4\pi}\int_{0}^{\infty
}k_{\rho}J_{0}(k_{\rho}\rho)\frac{e^{\ri k_{\ell z}(z-d_{\ell})}}{k_{\ell z}%
}\tilde{\sigma}_{\ell\ell^{\prime}}^{\uparrow}(k_{\rho},z^{\prime})dk_{\rho
},\quad\ell<L,\\[8pt]%
\displaystyle u_{\ell\ell^{\prime}}^{\downarrow}%
(\boldsymbol{r},\boldsymbol{r}^{\prime})=\frac{\ri}{4\pi}\int_{0}^{\infty
}k_{\rho}J_{0}(k_{\rho}\rho)\frac{e^{-\ri k_{\ell z}(z-d_{\ell})}}{k_{\ell z}%
}\tilde{\sigma}_{\ell\ell^{\prime}}^{\downarrow}(k_{\rho},z^{\prime})dk_{\rho
},\quad0<\ell<L,\\[8pt]%
\displaystyle u_{L\ell^{\prime}}^{\downarrow}%
(\boldsymbol{r},\boldsymbol{r}^{\prime})=\frac{\ri}{4\pi}\int_{0}^{\infty
}k_{\rho}J_{0}(k_{\rho}\rho)\frac{e^{-\ri k_{\ell z}z}}{k_{\ell z}}%
\tilde{\sigma}_{L\ell^{\prime}}^{\downarrow}(k_{\rho},z^{\prime})dk_{\rho},
\end{cases}
\label{greenfuncomponent}%
\end{equation}
where
\begin{equation}%
\begin{cases}
\displaystyle\tilde{\sigma}_{\ell0}^{\uparrow}(k_{\rho},z^{\prime})=e^{\ri
k_{0z}z^{\prime}}\sigma_{\ell0}^{\uparrow\uparrow}(k_{\rho
}),\\[8pt]%
\displaystyle\tilde{\sigma}_{\ell\ell^{\prime}}^{\uparrow}(k_{\rho},z^{\prime
})=e^{\ri k_{\ell^{\prime}z}(z^{\prime}-d_{\ell^{\prime}})}\sigma_{\ell
\ell^{\prime}}^{\uparrow\uparrow}(k_{\rho})+e^{-\ri k_{\ell^{\prime}%
z}(z^{\prime}-d_{\ell^{\prime}-1})}\sigma_{\ell\ell^{\prime}}^{\uparrow
\downarrow}(k_{\rho}),\quad0<\ell^{\prime}<L,\\[8pt]%
\displaystyle\tilde{\sigma}_{\ell\ell^{\prime}}^{\downarrow}(k_{\rho
},z^{\prime})=e^{\ri k_{\ell^{\prime}z}(z^{\prime}-d_{\ell^{\prime}})}%
\sigma_{\ell\ell^{\prime}}^{\downarrow\uparrow}(k_{\rho})+e^{-\ri
k_{\ell^{\prime}z}(z^{\prime}-d_{\ell^{\prime}-1})}\sigma_{\ell\ell^{\prime}%
}^{\downarrow\downarrow}(k_{\rho}),\quad0<\ell^{\prime}<L,\\[8pt]%
\displaystyle\tilde{\sigma}_{\ell L}^{\downarrow}(k_{\rho},z^{\prime})=e^{-\ri
k_{\ell^{\prime}z}(z^{\prime}-d_{L-1})}\sigma_{\ell L}^{\downarrow\downarrow
}(k_{\rho}).
\end{cases}
\label{totaldensity}%
\end{equation}
Note that the Green's function in the interior layers are given by
\[
u_{\ell\ell^{\prime}}(\boldsymbol{r},\boldsymbol{r}^{\prime})=%
\begin{cases}
\displaystyle u_{\ell\ell^{\prime}}^{\uparrow}%
(\boldsymbol{r},\boldsymbol{r}^{\prime})+u_{\ell\ell^{\prime}}^{\downarrow
}(\boldsymbol{r},\boldsymbol{r}^{\prime}), & \ell\neq\ell^{\prime},\\
\displaystyle u_{\ell\ell^{\prime}}^{\uparrow}%
(\boldsymbol{r},\boldsymbol{r}^{\prime})+u_{\ell\ell^{\prime}}^{\downarrow
}(\boldsymbol{r},\boldsymbol{r}^{\prime})+\frac{\ri k_{\ell^{\prime}}}{4\pi
}h_{0}^{(1)}(k_{\ell^{\prime}}|\boldsymbol{r}-\boldsymbol{r}^{\prime}|), &
\ell=\ell^{\prime}.
\end{cases}
\]

The derivation above is applicable to multi-layered media, in Appendix A, we
give explicit formulas (see \eqref{densitytwolayer1}, \eqref{densitytwolayer2}
and \eqref{densitythreelayer1},\eqref{densitythreelayer2},
\eqref{densitythreelayer3}) for the cases of two and three layers for
numerical tests of the fast algorithms as these cases cover a wide range of applications.

\section{First Taylor-expansion based FMM in multi-layered media}


\subsection{Free space}

First, we briefly review the TE-FMM for Helmholtz equations in the free space.
Consider $N$ source particles with source strength $q_{j}$ placed at
$\boldsymbol{r}_{j}=(x_{j},y_{j},z_{j})$. The field at $\boldsymbol{r}_{i}%
=(x_{i},y_{i},z_{i})$ due to all other sources is given by
\begin{equation}
u(\boldsymbol{r}_{i})=\sum\limits_{j=1}^{N}q_{j}h_{0}^{(1)}%
(k|\boldsymbol{r}_{i}-\boldsymbol{r}_{j}|),\quad i=1,2,\cdots,N,
\end{equation}
where $h_{0}^{(1)}(z)$ is the first kind spherical Hankel function of order
zero. Hereafter, we omit the factor $\frac{\ri k}{4\pi}$ and $\frac
{\ri k_{\ell}}{4\pi}$ in the free space and layered media Green's functions,
respectively. A TE-FMM will use the following Taylor expansions :

\begin{itemize}
\item \textbf{TE in a source box centered at $\boldsymbol{r}_{c}$: }we
have\textbf{ }
\begin{equation}
\sum\limits_{j\in J_{m}}q_{j}h_{0}^{(1)}(k|\boldsymbol{r}-\boldsymbol{r}_{j}%
|)\approx\sum\limits_{|\boldsymbol{k}|=0}^{p}\alpha_{\boldsymbol{k}}%
\frac{D_{\boldsymbol{r}^{\prime}}^{\boldsymbol{k}}h_{0}^{(1)}(k\Vert
{\boldsymbol{r}}-{\boldsymbol{r}}_{c}\Vert)}{\boldsymbol{k}!}%
,\label{taylorexpfarxcfree}%
\end{equation}
where
\begin{equation}
{\alpha}_{\boldsymbol{k}}=\sum\limits_{j\in J_{m}}q_{j}(\boldsymbol{r}_{j}%
-\boldsymbol{r}_{c})^{\boldsymbol{k}},\quad D_{\boldsymbol{r}^{\prime}%
}^{\boldsymbol{k}}:=\frac{\partial^{|\boldsymbol{k}|}}{\partial(x^{\prime
})^{k_{1}}\partial(y^{\prime})^{k_{2}}\partial(z^{\prime})^{k_{3}}},
\end{equation}
$J_{m}$ is the set of indices of particles in a source box centered at
$\boldsymbol{r}_{c}$ and the $\boldsymbol{r}$ is far from this box.

\item \textbf{TE in a target box centered at $\boldsymbol{r}_{c}^{l}$: }we
have\textbf{ }
\begin{equation}
\sum\limits_{j\in J_{m}}q_{j}h_{0}^{(1)}(k|\boldsymbol{r}-\boldsymbol{r}_{j}%
|)\approx\sum\limits_{|\boldsymbol{k}|=0}^{p}\beta_{\boldsymbol{k}}%
(\boldsymbol{r}-\boldsymbol{r}_{c}^{l})^{\boldsymbol{k}}%
,\label{taylorexplocxclfree}%
\end{equation}
where
\begin{equation}
{\beta}_{\boldsymbol{k}}=\sum\limits_{j\in J_{m}}q_{j}\frac{D_{\boldsymbol{r}}%
^{\boldsymbol{k}}h_{0}^{(1)}(k(\boldsymbol{r}_{c}^{l}-\boldsymbol{r}_{j}%
))}{\boldsymbol{k}!},\quad D_{\boldsymbol{r}}^{\boldsymbol{k}}:=\frac
{\partial^{|\boldsymbol{k}|}}{\partial x^{k_{1}}\partial y^{k_{2}}\partial
z^{k_{3}}},\label{localexpcoeff}%
\end{equation}
$\{(q_{j},\boldsymbol{r}_{j})\}_{j\in J_{m}}$ are particles in a source box
far from the target box.
\end{itemize}

\bigskip

Next, we present the translation operators.

\begin{itemize}
\item \textbf{Translation from a TE in a source box centered at
$\boldsymbol{r}_{c}$ to a TE in a target box centered at $\boldsymbol{r}_{c}%
^{l}$:}
\begin{equation}
{\beta}_{\boldsymbol{k}}\approx\sum\limits_{|\boldsymbol{k}^{\prime}|=0}%
^{p}{\alpha}_{\boldsymbol{k}^{\prime}}L_{\boldsymbol{k}}%
^{\boldsymbol{k}^{\prime}},
\end{equation}
where
\begin{equation}
L_{\boldsymbol{k}}^{\boldsymbol{k}^{\prime}}=\frac{D_{\boldsymbol{r}}%
^{\boldsymbol{k}}D_{\boldsymbol{r}^{\prime}}^{\boldsymbol{k}^{\prime}}%
h_{0}^{(1)}(k(\boldsymbol{r}_{c}^{l}-\boldsymbol{r}_{c}))}%
{\boldsymbol{k}!\boldsymbol{k}^{\prime}!}=(-1)^{|\boldsymbol{k}|}%
\frac{D_{\boldsymbol{r}^{\prime}}^{\boldsymbol{k}+\boldsymbol{k}^{\prime}%
}h_{0}^{(1)}(k(\boldsymbol{r}_{c}^{l}-\boldsymbol{r}_{c}))}%
{\boldsymbol{k}!\boldsymbol{k}^{\prime}!}.
\end{equation}

\item \textbf{Trnaslation from a TE in a source box centered at
$\boldsymbol{r}_{c}$ to a TE in another source box centered at
$\boldsymbol{r}_{c}^{\prime}$: } Let
\[
{\gamma}_{\boldsymbol{k}}=\sum\limits_{j\in J_{m}}q_{j}(\boldsymbol{r}_{j}%
-\boldsymbol{r}_{c}^{\prime})^{\boldsymbol{k}},
\]
be the coefficients of TE in the source box centered at $\boldsymbol{r}_{c}%
^{\prime}$, then the bi-nominal formula
\begin{equation}
(\boldsymbol{r}_{j}-\boldsymbol{r}_{c}^{\prime})^{\boldsymbol{k}}%
=\sum\limits_{k_{1}^{\prime}=0}^{k_{1}}\sum\limits_{k_{2}^{\prime}=0}^{k_{2}%
}\sum\limits_{k_{3}^{\prime}=0}^{k_{3}}B_{\boldsymbol{k}}%
^{\boldsymbol{k}^{\prime}}(\boldsymbol{r}_{j}-\boldsymbol{r}_{c}%
)^{\boldsymbol{k}^{\prime}},\label{LETOLEmat}%
\end{equation}
gives that
\begin{equation}
\gamma_{\boldsymbol{k}}=\sum\limits_{k_{1}^{\prime}=0}^{k_{1}}\sum
\limits_{k_{2}^{\prime}=0}^{k_{2}}\sum\limits_{k_{3}^{\prime}=0}^{k_{3}%
}B_{\boldsymbol{k}}^{\boldsymbol{k}^{\prime}}{\alpha}_{\boldsymbol{k}^{\prime
}},\quad\boldsymbol{k}^{\prime}=(k_{1}^{\prime},k_{2}^{\prime},k_{3}^{\prime
}),\label{TETartoTETar}%
\end{equation}
where
\begin{equation}
B_{\boldsymbol{k}}^{\boldsymbol{k}^{\prime}}=\frac
{\boldsymbol{k}!(\boldsymbol{r}_{c}-\boldsymbol{r}_{c}^{\prime}%
)^{\boldsymbol{k}-\boldsymbol{k}^{\prime}}}{k_{1}^{\prime}!(k_{1}%
-k_{1}^{\prime})!k_{2}^{\prime}!(k_{2}-k_2^{\prime})!k_{3}^{\prime}!(k_{3}%
-k_{3}^{\prime})!}.
\end{equation}

\item \textbf{Translation from a TE in a target box centered at
$\boldsymbol{r}_{c}^{l}$ to a TE in another target box centered at
$\tilde{\boldsymbol{r}}_{c}^{l}$:} Let
\begin{equation}
{\lambda}_{\boldsymbol{k}}=\sum\limits_{j\in J_{m}}\frac{q_{j}%
D_{\boldsymbol{r}}^{\boldsymbol{k}}h_{0}^{(1)}(k(\tilde{\boldsymbol{r}}%
_{c}^{l}-\boldsymbol{r}_{j}))}{\boldsymbol{k}!},
\end{equation}
be the coefficients of a TE in the target box centered at $\tilde
{\boldsymbol{r}}_{c}^{l}$. Then, the Taylor expansion at $\boldsymbol{r}_{c}%
^{l}$ gives
\begin{equation}
{\lambda}_{\boldsymbol{k}}\approx\frac{1}{\boldsymbol{k}!}\sum
\limits_{|\boldsymbol{k}^{\prime}|=0}^{p}\beta_{\boldsymbol{k}^{\prime}%
}D_{\tilde{\boldsymbol{r}}_{c}^{l}}^{\boldsymbol{k}}(\tilde{\boldsymbol{r}}%
_{c}^{l}-\boldsymbol{r}_{c}^{l})^{\boldsymbol{k}^{\prime}}.
\end{equation}
Note that
\begin{equation}
D_{\tilde{\boldsymbol{r}}_{c}^{l}}^{\boldsymbol{k}}(\tilde{\boldsymbol{r}}%
_{c}^{l}-\boldsymbol{r}_{c}^{l})^{\boldsymbol{k}^{\prime}}=%
\begin{cases}
\displaystyle0,\quad k_{1}>k_{1}^{\prime}\;\;\mathrm{or}\;\;k_{2}%
>k_{2}^{\prime}\;\;\mathrm{or}\;\;k_{3}>k_{3}^{\prime},\\[6pt]%
\displaystyle\frac{\boldsymbol{k}^{\prime}!}{(\boldsymbol{k}^{\prime
}-\boldsymbol{k})!}(\tilde{\boldsymbol{r}}_{c}^{l}-\boldsymbol{r}_{c}%
^{l})^{\boldsymbol{k}^{\prime}-\boldsymbol{k}},\quad\mathrm{otherwise},
\end{cases}
\label{LE2LErans}%
\end{equation}
then
\begin{equation}
{\lambda}_{\boldsymbol{k}}=\sum\limits_{n^{\prime}=|\boldsymbol{k}|}^{p}%
\sum\limits_{\boldsymbol{k}^{\prime}\geq\boldsymbol{k}}%
^{|\boldsymbol{k}^{\prime}|\leq n^{\prime}}\beta_{\boldsymbol{k}^{\prime}%
}\frac{\boldsymbol{k}^{\prime}!}{\boldsymbol{k}!(\boldsymbol{k}^{\prime
}-\boldsymbol{k})!}(\tilde{\boldsymbol{r}}_{c}^{l}-\boldsymbol{r}_{c}%
^{l})^{\boldsymbol{k}^{\prime}-\boldsymbol{k}}.\label{freeletole}%
\end{equation}

\end{itemize}

\subsection{Multi-layered media}

Let $\mathscr{P}_{\ell}=\{(Q_{\ell j},\boldsymbol{r}_{\ell j}),j=1,2,\cdots
,N_{\ell}\}$ be a group of source particles distributed in the $\ell$-th layer
of a multi-layered medium with $L+1$ layers (see Fig. \ref{layerstructure}).
The interactions between all $N:=N_{0}+N_{1}+\cdots+N_{L}$ particles given by
the sum
\begin{equation}
\Phi_{\ell}(\boldsymbol{r}_{\ell i})=\Phi_{\ell}^{free}(\boldsymbol{r}_{\ell
i})+\sum\limits_{\ell^{\prime}=0}^{L}[\Phi_{\ell\ell^{\prime}}^{\uparrow
}(\boldsymbol{r}_{\ell i})+\Phi_{\ell\ell^{\prime}}^{\downarrow}%
(\boldsymbol{r}_{\ell i})],\label{totalinteraction}%
\end{equation}
for $\ell=0,1,\cdots,L;\;\;i=1,2,\cdots,N_{\ell}$, where
\begin{equation}%
\begin{split}
&  \Phi_{\ell}^{free}(\boldsymbol{r}_{\ell i}):=\sum\limits_{j=1,j\neq
i}^{N_{\ell}}Q_{\ell j}h_{0}^{(1)}(k_{\ell}|\boldsymbol{r}_{\ell
i}-\boldsymbol{r}_{\ell j}|),\\
&  \Phi_{\ell\ell^{\prime}}^{\uparrow}(\boldsymbol{r}_{\ell i}):=\sum
\limits_{j=1}^{N_{\ell^{\prime}}}Q_{\ell^{\prime}j}u_{\ell\ell^{\prime}%
}^{\uparrow}(\boldsymbol{r}_{\ell i},\boldsymbol{r}_{\ell^{\prime}j}%
),\quad\Phi_{\ell\ell^{\prime}}^{\downarrow}(\boldsymbol{r}_{\ell i}%
):=\sum\limits_{j=1}^{N_{\ell^{\prime}}}Q_{\ell^{\prime}j}u_{\ell\ell^{\prime
}}^{\downarrow}(\boldsymbol{r}_{\ell i},\boldsymbol{r}_{\ell^{\prime}j})).
\end{split}
\end{equation}
Here $u_{\ell\ell^{\prime}}^{\uparrow},u_{\ell\ell^{\prime}}^{\downarrow}$ are
the general scattering component of the domain Green's function in the $\ell
$-th layer due to a source $\boldsymbol{r}_{\ell^{\prime}j}$ in the
$\ell^{\prime}$-th layer. We also omit the factor $\frac{\ri k_{\ell}}{4\pi}$
in $u_{\ell\ell^{\prime}}^{\uparrow}$ and $u_{\ell\ell^{\prime}}^{\downarrow}$
for consistency with the free space case. In the top and bottom most layer, we
have
\[
u_{0\ell^{\prime}}^{\downarrow}(\boldsymbol{r},\boldsymbol{r}^{\prime
})=0,\quad u_{L\ell^{\prime}}^{\uparrow}(\boldsymbol{r},\boldsymbol{r}^{\prime
})=0,\quad0\leq\ell^{\prime}\leq L.
\]
General formulas for $u_{\ell\ell^{\prime}}^{\uparrow},u_{\ell\ell^{\prime}%
}^{\downarrow}$ are given in \eqref{greenfuncomponent}-\eqref{totaldensity}
while densities for two and three layered cases are presented in the Appendix
A (see expressions in \eqref{densitytwolayer1}, \eqref{densitytwolayer2} and
\eqref{densitythreelayer1},\eqref{densitythreelayer2}, \eqref{densitythreelayer3}).

Since the domain Green's function in multi-layer media has different
representations \eqref{greenfuncomponent} for source and target particles in
different layers, it is necessary to perform calculation individually for
interactions between any two groups of particles among the $L+1$ groups
$\{\mathscr{P}_{\ell}\}_{\ell=0}^{L}$. Without a loss of generality, let us
focus on the computation of upgoing component of the interaction between
$\ell$-th and $\ell^{\prime}$-th groups, i.e.,
\begin{equation}
\Phi_{\ell\ell^{\prime}}^{\uparrow}(\boldsymbol{r}_{\ell i})=\sum
\limits_{j=1}^{N_{\ell^{\prime}}}Q_{\ell^{\prime}j}u_{\ell\ell^{\prime}%
}^{\uparrow}(\boldsymbol{r}_{\ell i},\boldsymbol{r}_{\ell^{\prime}j}),\quad
i=1,2,\cdots,N_{\ell}.\label{generalsum}%
\end{equation}
Let
\begin{equation}
\Phi_{\ell\ell^{\prime}}^{b\uparrow}(\boldsymbol{r}_{\ell i}):=\sum
\limits_{j\in J_{m}}Q_{\ell^{\prime}j}u_{\ell\ell^{\prime}}^{\uparrow
}(\boldsymbol{r}_{\ell i},\boldsymbol{r}_{\ell^{\prime}j}%
),\label{generalsuminbox}%
\end{equation}
be the field at $\boldsymbol{r}_{\ell i}$ generated by the particles in a
source box centered at $\boldsymbol{r}_{c}=(x_{c},y_{c},z_{c})$ in the tree
structure. Here, $J_{m}$ is the set of indices of particles in the source box. The Taylor expansion based FMM for \eqref{generalsum} will use TE
approximations
\begin{equation}
\Phi_{\ell\ell^{\prime}}^{b\uparrow}(\boldsymbol{r}_{\ell i})\approx
\sum_{|\boldsymbol{k}|=0}^{p}\alpha_{\boldsymbol{k}}\frac
{D_{\boldsymbol{r}^{\prime}}^{\boldsymbol{k}}u_{\ell\ell^{\prime}}^{\uparrow
}(\boldsymbol{r}_{\ell i},\boldsymbol{r}_{c})}{\boldsymbol{k}!},\quad{\alpha
}_{\boldsymbol{k}}=\sum\limits_{j\in J_{m}}Q_{\ell^{\prime}j}%
(\boldsymbol{r}_{\ell^{\prime}j}-\boldsymbol{r}_{c})^{\boldsymbol{k}}%
,\label{taylorexpfar3layers}%
\end{equation}
in the source box and
\begin{equation}
\Phi_{\ell\ell^{\prime}}^{b\uparrow}(\boldsymbol{r}_{\ell i})\approx
\sum\limits_{|\boldsymbol{k}|=0}^{p}\beta_{\boldsymbol{k}}%
(\boldsymbol{r}_{\ell i}-\boldsymbol{r}_{c}^{l})^{\boldsymbol{k}},\quad{\beta
}_{\boldsymbol{k}}=\sum\limits_{j\in J_{m}}\frac{Q_{\ell^{\prime}%
j}D_{\boldsymbol{r}}^{\boldsymbol{k}}u_{\ell\ell^{\prime}}^{\uparrow
}(\boldsymbol{r}_{c}^{l},\boldsymbol{r}_{\ell^{\prime}j})}{\boldsymbol{k}!}%
,\label{taylorexploc2ld}%
\end{equation}
in the target box centered at $\boldsymbol{r}_{c}^{l}=(x_{c}^{l},y_{c}%
^{l},z_{c}^{l})$, respectively. Note that $u_{\ell\ell^{\prime}}^{\uparrow
}(\boldsymbol{r},\boldsymbol{r}^{\prime})$ has a Sommerfeld integral
representation with an integrand involving an exponential function $e^{\ri
k_{\ell z}(z-d_{\ell})}\tilde{\sigma}_{\ell\ell^{\prime}}^{\uparrow}(k_{\rho
},z^{\prime})$. It is worthy to point out that the integrand has an
exponential decay when $d_{\ell}<z<d_{\ell}-1,d_{\ell^{\prime}}<z^{\prime
}<d_{\ell^{\prime}}-1,$ which ensures the convergence of the Sommerfeld
integral.

According to the Taylor expansions \eqref{taylorexpfar3layers} and
\eqref{taylorexploc2ld}, we conclude that the translation operators for center
shifting from source boxes to their parents and from target boxes to their
children are exactly the same as in free space case which are given by
\eqref{TETartoTETar} and \eqref{freeletole}. The translation from a TE in
source box to a TE in target box is given by
\begin{equation}
{\beta}_{\boldsymbol{k}}\approx\sum\limits_{|\boldsymbol{k}^{\prime}|=0}%
^{p}{\alpha}_{\boldsymbol{k}^{\prime}}\frac{D_{\boldsymbol{r}}%
^{\boldsymbol{k}}D_{\boldsymbol{r}^{\prime}}^{\boldsymbol{k}^{\prime}}%
u_{\ell\ell^{\prime}}^{\uparrow}(\boldsymbol{r}_{c}^{l},\boldsymbol{r}_{c}%
)}{\boldsymbol{k}!\boldsymbol{k}^{\prime}!}.\label{tesourceboxtotargetbox}%
\end{equation}
In the next section, an efficient algorithm for the computation of
$D_{\boldsymbol{r}}^{\boldsymbol{k}}D_{\boldsymbol{r}^{\prime}}%
^{\boldsymbol{k}^{\prime}}u_{\ell\ell^{\prime}}^{\uparrow}(\boldsymbol{r}_{c}%
^{l},\boldsymbol{r}_{c})$ will be presented.

\subsection{Discrete complex-image approximation of derivatives of Green's
functions for layered media}

The TE-FMM demands an efficient algorithm for the computation of derivatives
of Green's function. For free space case, recurrence formulas are available
(cf. \cite{li2009cartesian, tausch2003fast}). In layered media, the following
derivatives are needed
\begin{align}
\frac{D_{\boldsymbol{r}^{\prime}}^{\boldsymbol{k}}u_{\ell\ell^{\prime}%
}^{\uparrow}(\boldsymbol{r},\boldsymbol{r}^{\prime})}{\boldsymbol{k}!}  &
=\frac{1}{\boldsymbol{k}!k_{\ell}}D_{\boldsymbol{r}^{\prime}}^{\boldsymbol{k}}%
\int_{0}^{\infty}k_{\rho}J_{0}(k_{\rho}\rho)\frac{e^{\ri k_{\ell z}(z-d_{\ell
})}}{k_{\ell z}}\tilde{\sigma}_{\ell\ell^{\prime}}^{\uparrow}(k_{\rho
},z^{\prime})dk_{\rho},\label{SIs1}\\
\frac{D_{\boldsymbol{r}}^{\boldsymbol{k}}D_{\boldsymbol{r}^{\prime}%
}^{\boldsymbol{k}^{\prime}}u_{\ell\ell^{\prime}}^{\uparrow}%
(\boldsymbol{r},\boldsymbol{r}^{\prime})}%
{\boldsymbol{k}!\boldsymbol{k}^{\prime}!}  & =\frac{1}%
{\boldsymbol{k}!\boldsymbol{k}^{\prime}!k_{\ell}}D_{\boldsymbol{r}}%
^{\boldsymbol{k}}D_{\boldsymbol{r}^{\prime}}^{\boldsymbol{k}^{\prime}}\int
_{0}^{\infty}k_{\rho}J_{0}(k_{\rho}\rho)\frac{e^{\ri k_{\ell z}(z-d_{\ell})}%
}{k_{\ell z}}\tilde{\sigma}_{\ell\ell^{\prime}}^{\uparrow}(k_{\rho},z^{\prime
})dk_{\rho},\label{SIs2}%
\end{align}
where $\boldsymbol{k}=(k_{1},k_{2},k_{3}),\boldsymbol{k}^{\prime}%
=(k_{1}^{\prime},k_{2}^{\prime},k_{3}^{\prime})$ are multi-indices. They are
derivatives of a function represented in terms of Sommerfeld integral (SI). It
is well known that SI has oscillatory integrand with pole singularities due to
the existence of surface waves. Over the past decades, much effort has been
made on the computation of this integral, including using ideas from
high-frequency asymptotics, rational approximation, contour deformation (cf.
\cite{cai2013computational,cai2000fast,cho2012parallel,okhmatovski2004evaluation,paulus2000accurate}%
), complex images (cf.
\cite{fang1988discrete,paulus2000accurate,ochmann2004complex,alparslan2010closed}%
), and methods based on special functions (cf. \cite{koh2006exact}) or
physical images (cf.
\cite{li1996near,ling2000discrete,o2014efficient,lai2016new}).

Since \eqref{SIs1} is just a special case of \eqref{SIs2}, our discussion will
only focus on the latter. This integral is convergent when the target and
source particles are not exactly on the interfaces of a layered medium.
Contour deformation with high order quadrature rules could be used for direct
numerical computation. However, this becomes prohibitively expensive due to a
large number of derivatives needed in the FMM. In fact, $O(p^{6})$ derivatives
will be needed for each source box to target box translation. Moreover, the
involved integrand decays more and more slowly as the derivative order is
getting higher. The length of contour needs to be very long to obtain a
required accuracy for the computation of high order derivatives. Therefore,
putting all derivatives inside the integral and then applying quadratures with
contour deformation is too expensive in terms of CPU time.

Moreover, despite that $u_{\ell\ell^{\prime}}^{\uparrow}$ is a function of
$(\rho,z,z^{\prime})$ only, the derivative $D_{\boldsymbol{r}}%
^{\boldsymbol{k}}D_{\boldsymbol{r}^{\prime}}^{\boldsymbol{k}^{\prime}}%
u_{\ell\ell^{\prime}}^{\uparrow}(\boldsymbol{r},\boldsymbol{r}^{\prime})$
depends on all coordinates in ${\boldsymbol{r}}$ and $\boldsymbol{r}^{\prime}$
due to the nonsymmetric derivative, it is not feasible to make a precomputed
table on a fine grid and then use interpolation to obtain approximation for
the derivative $D_{\boldsymbol{r}}^{\boldsymbol{k}}D_{\boldsymbol{r}^{\prime}%
}^{\boldsymbol{k}^{\prime}}u_{\ell\ell^{\prime}}^{\uparrow}%
(\boldsymbol{r},\boldsymbol{r}^{\prime})$. Instead, for this TE-FMM, we will
use a complex image approximation of the integrand to simplify the calculation
of the derivatives.


Exchanging the order of the derivative and the integral leads to
\begin{equation}
\frac{D_{\boldsymbol{r}}^{\boldsymbol{k}}D_{\boldsymbol{r}^{\prime}%
}^{\boldsymbol{k}^{\prime}}u_{\ell\ell^{\prime}}^{\uparrow}%
(\boldsymbol{r},\boldsymbol{r}^{\prime})}%
{\boldsymbol{k}!\boldsymbol{k}^{\prime}!}=\frac{D_{\boldsymbol{r}}%
^{\boldsymbol{k}}D_{\boldsymbol{r}^{\prime}}^{\boldsymbol{k}_{0}^{\prime}}%
}{\boldsymbol{k}!\boldsymbol{k}_{0}^{\prime}!}\Big(\frac{1}{k_{\ell}}\int
_{0}^{\infty}k_{\rho}J_{0}(k_{\rho}\rho)\frac{e^{\ri k_{\ell z}(z-d_{\ell})}%
}{k_{\ell z}}\frac{\partial_{z^{\prime}}^{k_{3}^{\prime}}\tilde{\sigma}%
_{\ell\ell^{\prime}}^{\uparrow}(k_{\rho},z^{\prime})}{k_{3}^{\prime}!}%
dk_{\rho}\Big),\label{integralformderi}%
\end{equation}
where $\boldsymbol{k}_{0}^{\prime}=(k_{1}^{\prime},k_{2}^{\prime},0)$ are
multi-indices reduced from $\boldsymbol{k}^{\prime}$,
\begin{equation}
\frac{\partial_{z^{\prime}}^{k_{3}^{\prime}}\tilde{\sigma}_{\ell\ell^{\prime}%
}^{\uparrow}(k_{\rho},z^{\prime})}{k_{3}^{\prime}!}=\frac{(\ri k_{\ell
^{\prime}z})^{k_{3}^{\prime}}\big[e^{\ri k_{\ell^{\prime}z}(z^{\prime}%
-d_{\ell^{\prime}})}\sigma_{\ell\ell^{\prime}}^{\uparrow\uparrow}%
+(-1)^{k_{3}^{\prime}}e^{\ri k_{\ell^{\prime}z}(d_{\ell^{\prime}-1}-z^{\prime
})}\sigma_{\ell\ell^{\prime}}^{\uparrow\downarrow}\big]}{k_{3}^{\prime}%
!},\label{layermediumdensity}%
\end{equation}
corresponds to the derivatives with respect to $z^{\prime}$. Note that the
variables $x,y,z,x^{\prime},y^{\prime}$ and $z^{\prime}$ are in separate
functions in the Sommerfeld integral \eqref{greenfuncomponent}. Let us first
consider the derivatives with respect to $z^{\prime}$. Recalling
\eqref{integralformderi}, we have
\begin{equation}
\frac{1}{k_{3}^{\prime}!}\partial_{z^{\prime}}^{k_{3}^{\prime}}u_{\ell
\ell^{\prime}}(\boldsymbol{r},\boldsymbol{r}^{\prime})=\frac{1}{k_{\ell}}%
\int_{0}^{\infty}k_{\rho}J_{0}(k_{\rho}\rho)\frac{e^{\ri k_{\ell z}(z-d_{\ell
})}}{k_{\ell z}}\frac{\partial_{z^{\prime}}^{k_{3}^{\prime}}\tilde{\sigma
}_{\ell\ell^{\prime}}^{\uparrow}(k_{\rho},z^{\prime})}{k_{3}^{\prime}%
!}dk_{\rho}.\label{m2lderivative}%
\end{equation}
The derivatives of $u_{\ell\ell^{\prime}}^{\uparrow}%
(\boldsymbol{r},\boldsymbol{r}^{\prime})$ with respect to $z^{\prime}$ are
represented by Sommerfeld integrals with densities $\frac{\partial_{z^{\prime
}}^{k_{3}^{\prime}}\tilde{\sigma}_{\ell\ell^{\prime}}^{\uparrow}(k_{\rho
},z^{\prime})}{k_{3}^{\prime}!}$. Now, we will use the discrete complex image
method (DCIM) (cf. \cite{aksun1996robust, alparslan2010closed}) to generate an
approximation using a sum of free space Green's function with complex
coordinates. To use the decay from $e^{\ri k_{\ell}(z-d_{\ell})}$, we define
\begin{equation}
\Theta_{\ell\ell^{\prime}}^{k_{3}^{\prime}}(k_{\rho},z^{\prime})=\frac{e^{\ri
k_{\ell z}(z_{min}-d_{\ell})}\partial_{z^{\prime}}^{k_{3}^{\prime}}%
\tilde{\sigma}_{\ell\ell^{\prime}}^{\uparrow}(k_{\rho},z^{\prime})}%
{k_{3}^{\prime}!}.
\end{equation}
Here, we choose $z_{min}$ to be the minimum $z$ coordinates of all target
particles in $\ell$-th layer, so the remaining term $e^{\ri k_{\ell
z}(z-z_{min})}$ still decays as $k_{\rho}\rightarrow\infty$. A two level DCIM
method is used to approximate $\Theta_{\ell\ell^{\prime}}^{k_{3}^{\prime}%
}(k_{\rho},z^{\prime})$ as follows:

\begin{itemize}
\item[Step 1:] Sample density function $\Theta_{\ell\ell^{\prime}}%
^{k_{3}^{\prime}}(k_{\rho}, z^{\prime})$ over a path defined by the following
mappings (see Fig. \ref{twolevelcontour})
\begin{equation}%
\begin{split}
k_{\ell z}=\ri k_{\ell}(T_{0}+t),\quad0\leq t\leq T_{1},\quad C_{ap1}:
1^{\mathrm{st}}\;\; \mathrm{level},\\
k_{\ell z}=k_{\ell} \Big(1-\frac{t}{T_{0}}+\ri t\Big),\quad0\leq t\leq
T_{0},\quad C_{ap2}: 2^{\mathrm{nd}}\;\; \mathrm{level}.
\end{split}
\end{equation}

\item[Step 2:] Approximate the sampled $\Theta_{\ell\ell^{\prime}}%
^{k_{3}^{\prime}}(k_{\rho},z^{\prime})$ by summation of complex exponentials
as
\begin{equation}
\Theta_{\ell\ell^{\prime}}^{k_{3}^{\prime}}(k_{\rho},z^{\prime})\approx
\sum\limits_{j=1}^{M}A_{j}^{k_{3}^{\prime}}e^{-\ri k_{\ell z}Z_{j}%
^{k_{3}^{\prime}}},\label{compleximageapp}%
\end{equation}
using a generalized pencil-of-function method (GPOF) \cite{hua1989generalized}.

\item[Step 3:] Then, we have
\begin{equation}%
\begin{split}
\frac{\partial_{z^{\prime}}^{k_{3}^{\prime}}u_{\ell\ell^{\prime}%
}(\boldsymbol{r},\boldsymbol{r}^{\prime})}{k_{3}^{\prime}!}= &  \frac
{1}{k_{\ell}}\int_{0}^{\infty}k_{\rho}J_{0}(k_{\rho}\rho)\frac{e^{\ri k_{\ell
z}(z-z_{min})}}{k_{\ell z}}\Theta_{\ell\ell^{\prime}}^{k_{3}^{\prime}}%
(k_{\rho},z^{\prime})dk_{\rho}\\
\approx &  \sum\limits_{j=1}^{M}A_{j}^{k_{3}^{\prime}}\Big(\frac{1}{k_{\ell}%
}\int_{0}^{\infty}k_{\rho}J_{0}(k_{\rho}\rho)\frac{e^{\ri k_{\ell z}%
(z-z_{min}-Z_{j}^{k_{3}^{\prime}})}}{k_{\ell z}}dk_{\rho}\Big),
\end{split}
\end{equation}
and by applying the Sommerfeld identity
\begin{equation}
h_{0}^{(1)}(k|\boldsymbol{r}|)=\frac{1}{k_{\ell}}\int_{0}^{\infty}k_{\rho
}J_{0}(k_{\rho}\rho)\frac{e^{\ri k_{\ell z}|z|}}{k_{\ell z}}dk_{\rho},
\end{equation}
to the SI with complex $z$-coordinates, we arrive at the following
approximations to the derivatives,
\begin{equation}
\frac{\partial_{z^{\prime}}^{k_{3}^{\prime}}u_{\ell\ell^{\prime}%
}(\boldsymbol{r},\boldsymbol{r}^{\prime})}{k_{3}^{\prime}!}\approx
\sum\limits_{j=1}^{M}A_{j}^{k_{3}^{\prime}}h_{0}^{(1)}(k_{\ell}R_{j}%
^{k_{3}^{\prime}}),\label{closedformgreen}%
\end{equation}
where $R_{j}^{k_{3}^{\prime}}=\sqrt{(x-x^{\prime})^{2}+(y-y^{\prime}%
)^{2}+(z-z_{min}-Z_{j}^{k_{3}^{\prime}})^{2}}$ is the complex distance.
\end{itemize}

\begin{figure}[ptbh]
\subfigure[Contour in $k_{\rho}$-plane]{\includegraphics[scale=0.3]{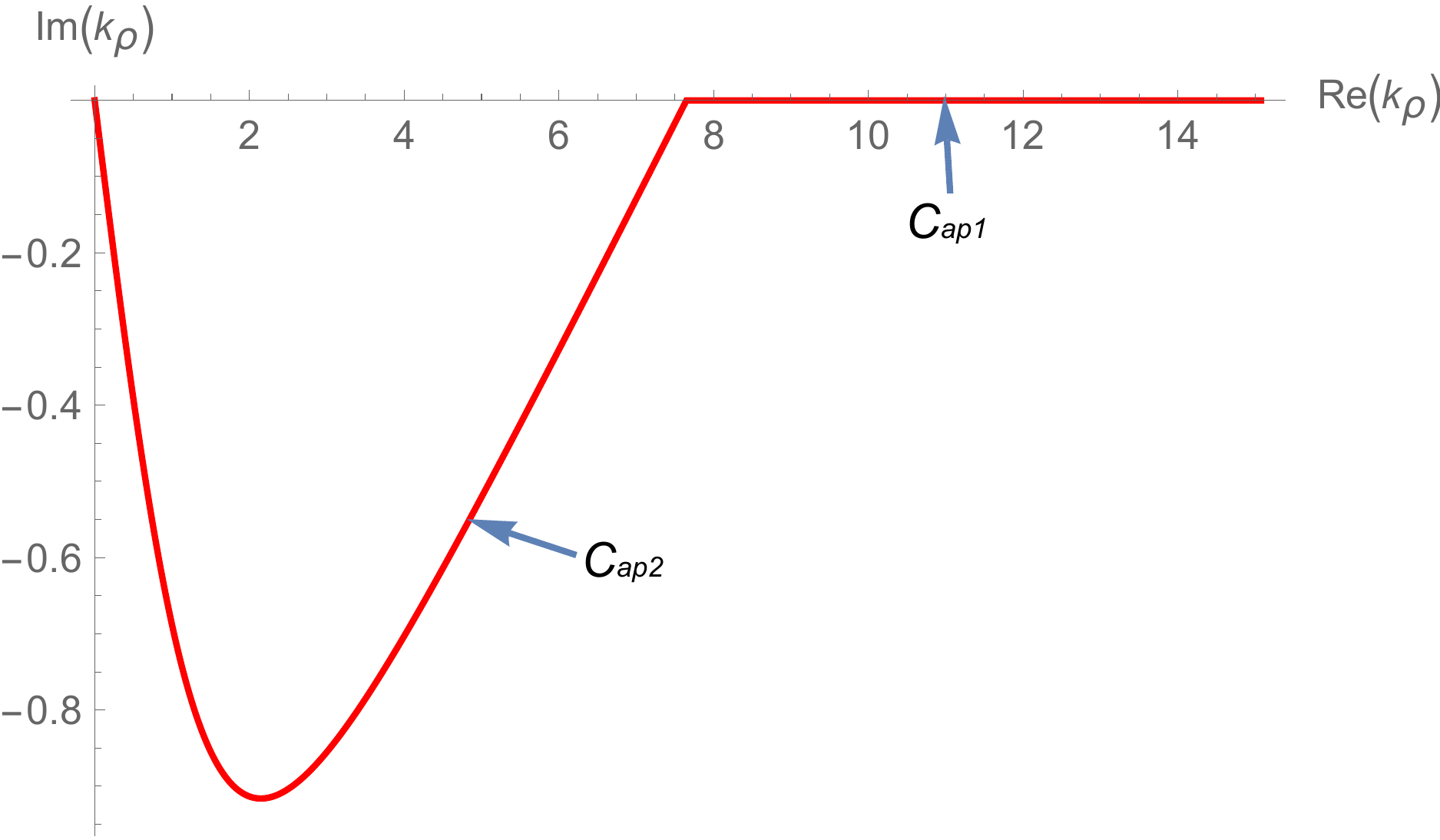}}\qquad
\subfigure[Contour in $k_{\ell z}$-plane]{\includegraphics[scale=0.3]{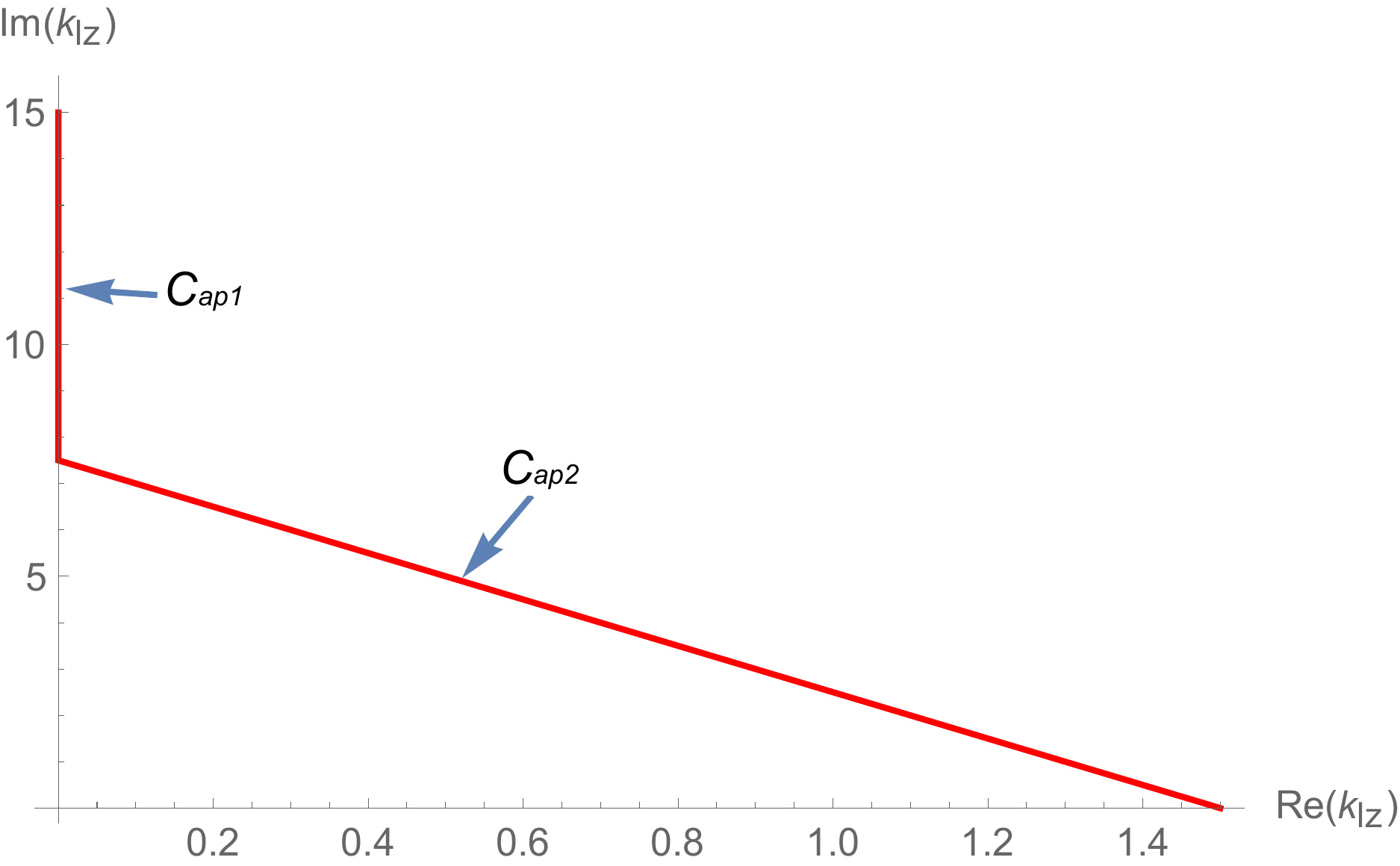}}
\caption{Plots of contour used in two level DCIM method ($k_{\ell}=1.5$,
$T_{0}=5$).}%
\label{twolevelcontour}%
\end{figure}By taking derivative $D_{\boldsymbol{r}}^{\boldsymbol{k}}%
D_{\boldsymbol{r}^{\prime}}^{\boldsymbol{k}_{0}^{\prime}}$ directly on
\eqref{closedformgreen}, we obtain approximation
\begin{equation}
\frac{D_{\boldsymbol{r}}^{\boldsymbol{k}}D_{\boldsymbol{r}^{\prime}%
}^{\boldsymbol{k}^{\prime}}u_{\ell\ell^{\prime}}%
(\boldsymbol{r},\boldsymbol{r}^{\prime})}%
{\boldsymbol{k}!\boldsymbol{k}^{\prime}!}\approx\sum\limits_{j=1}^{M}%
A_{j}^{k_{3}^{\prime}}\frac{D_{\boldsymbol{r}}^{\boldsymbol{k}}%
D_{\boldsymbol{r}^{\prime}}^{\boldsymbol{k}_{0}^{\prime}}h_{0}^{(1)}(k_{\ell
}R_{j}^{k_{3}^{\prime}})}{\boldsymbol{k}!\boldsymbol{k}_{0}^{\prime}%
!}\label{derivativeapp}%
\end{equation}
Note that the approximation \eqref{compleximageapp} is independent of
$x,y,z,x^{\prime},y^{\prime}$. Approximation \eqref{derivativeapp} is expected
to maintain the accuracy of the approximation \eqref{compleximageapp}.
Numerical results verify this fact at the end of this section. More
importantly, the derivatives on Hankel function with complex coordinates can
be calculated by using a recurrence formula.

Define
\begin{equation}
a^{\boldsymbol{k}}(\boldsymbol{r},\boldsymbol{r}^{\prime},k):=\frac
{1}{\boldsymbol{k}!}D_{\boldsymbol{r}}^{\boldsymbol{k}}h_{0}^{(1)}%
(k|\boldsymbol{r}-\boldsymbol{r}^{\prime}|),\quad b^{\boldsymbol{k}}%
(\boldsymbol{r},\boldsymbol{r}^{\prime},k)=\frac{1}{\boldsymbol{k}!}%
D_{\boldsymbol{r}}^{\boldsymbol{k}}\psi,
\end{equation}
where $\boldsymbol{r}^{\prime}=(x^{\prime},y^{\prime},Z^{\prime})$ is a
coordinate with complex $z$-coordinate $Z^{\prime}$. Then, we have the
following recurrence formula
\begin{equation}%
\begin{split}
|\boldsymbol{k}||\boldsymbol{r}-\boldsymbol{r}^{\prime}|^{2}a^{\boldsymbol{k}}
&  +2(|\boldsymbol{k}|-1)\sum\limits_{i=1}^{3}%
(\boldsymbol{r}-\boldsymbol{r}^{\prime})_{i}%
a^{\boldsymbol{k}-\boldsymbol{e}_{i}}+(|\boldsymbol{k}|+1)\sum\limits_{i=1}%
^{3}a^{\boldsymbol{k}-2\boldsymbol{e}_{i}}\\
= &  \ri k\Big(\sum\limits_{i=1}^{3}(\boldsymbol{r}-\boldsymbol{r}^{\prime
})_{i}b^{\boldsymbol{k}-\boldsymbol{e}_{i}}+\sum\limits_{i=1}^{2}%
b^{\boldsymbol{k}-2\boldsymbol{e}_{i}}\Big),\\
|\boldsymbol{k}|b^{\boldsymbol{k}}= &  \ri k\Big(\sum\limits_{i=1}%
^{3}(\boldsymbol{r}-\boldsymbol{r}^{\prime})_{i}%
a^{\boldsymbol{k}-\boldsymbol{e}_{i}}+\sum\limits_{i=1}^{3}%
a^{\boldsymbol{k}-2\boldsymbol{e}_{i}}\Big).
\end{split}
\label{recurrsion}%
\end{equation}
The derivation can be done by simply following the procedure in
\cite{li2009cartesian}, since the involved derivatives are independent of the
complex coordinate $z^{\prime}$. With this recurrence formula, derivatives
\begin{equation}
D_{\boldsymbol{r}}^{\boldsymbol{k}}D_{\boldsymbol{r}^{\prime}}%
^{\boldsymbol{k}_{0}^{\prime}}h_{0}^{(1)}(k_{\ell}R_{j}^{k_{3}^{\prime}%
})=(-1)^{|\boldsymbol{k}_{0}^{\prime}|}D_{\boldsymbol{r}}%
^{\boldsymbol{k}+\boldsymbol{k}_{0}^{\prime}}h_{0}^{(1)}(k_{\ell}R_{j}%
^{k_{3}^{\prime}}),
\end{equation}
can be efficiently calculated.

In the free space TE-FMM, the most time consuming part is the translation from
a source box to a target box where a recurrence formula is used to calculate
all $O(p^{6})$ derivatives. Note that $D_{\boldsymbol{r}}^{\boldsymbol{k}}%
D_{\boldsymbol{r}^{\prime}}^{\boldsymbol{k}^{\prime}}u_{\ell\ell^{\prime}%
}^{\uparrow}(\boldsymbol{r}_{c}^{l},\boldsymbol{r}_{c})$ only depends on the
center of the corresponding boxes. More precisely, the density
\eqref{layermediumdensity} approximated by DCIM only depends on the
$z$-coordinates of the center of boxes in the source tree. Once the tree
structure is fixed, we can pre-compute a table for all complex exponential
approximations used for the computation of $D_{\boldsymbol{r}}%
^{\boldsymbol{k}}D_{\boldsymbol{r}^{\prime}}^{\boldsymbol{k}^{\prime}}%
u_{\ell\ell^{\prime}}^{\uparrow}(\boldsymbol{r}_{c}^{l},\boldsymbol{r}_{c})$.
Assume that the depth of the source tree is $H$, then only $2^{H}(p+1)$ DCIM
approximations are needed to be precomputed.


Next, we will give some numerical results to show the accuracy of the two
level DCIM and show that taking derivative with respect to $x,y,z,x^{\prime
},y^{\prime}$ will not result in an accuracy loss. For this purpose, let us
consider the approximation of
\begin{equation}
\Big(\frac{\partial}{\partial x}+\ri\frac{\partial}{\partial y}\Big)^{s}%
\Big(\frac{\partial}{\partial z}\Big)^{k_{3}}\Big(\frac{\partial}{\partial
z^{\prime}}\Big)^{k_{3}^{\prime}}\frac{u_{11}^{\uparrow}%
(\boldsymbol{r},\boldsymbol{r}^{\prime})}{s!k_{3}!k_{3}^{\prime}%
!},\label{testderi}%
\end{equation}
for three layers case with $k_{0}=0.8$, $k_{1}=1.5$, $k_{2}=2.0$, $d=2.0$. In
the two level DCIM approximation, we set $T_{0}=\sqrt{\Big(\frac{k_{2}%
+0.8}{k_{1}}\Big)^{2}-1},T_{1}=10$ and use $101$ sample points in each level.
\begin{table}[ptbh]
\centering {\small
\begin{tabular}
[c]{|c|c|c|c|c|}\hline
& $(k_{3}, k_{3}^{\prime}, s)$ & direct quadrature & DCIM & error\\\hline
\multirow{5}{*}{\begin{tabular}{c} $\bs r=(0.5, 1.0, -0.5)$\\ $\bs r'=(0.3, 1.3, -0.5)$ \end{tabular}} &
(0, 0, 0) & 0.0636386627264339 & 0.063638662478093 & 2.4834e-10\\\cline{2-5}
& (3, 4, 0) & 0.00474777580070183 & 0.004747777003526 &
-1.2028e-09\\\cline{2-5}
& (8, 8, 0) & -7.40635683599036e-10 & -7.406632552555640e-10 &
2.7572e-14\\\cline{2-5}
& (0, 0, 4) & -1.77276908208051e-06 & -1.772752691103600e-06 &
-1.6391e-11\\\cline{2-5}
& (0, 0, 8) & 1.61980348471514e-11 & 1.619802703109298e-11 &
7.8161e-18\\\hline
\multirow{5}{*}{\begin{tabular}{c} $\bs r=(0.6, 0.3, -1.2)$\\ $\bs r'=(0.5, 1.0, -0.5)$ \end{tabular}} &
(0, 0, 0) & 0.0470021533117637 & 0.047002199376864 & -4.6065e-08\\\cline{2-5}
& (3, 4, 0) & 0.00185695910047338 & 0.001856957782404 &
1.3181e-09\\\cline{2-5}
& (8, 8, 0) & 5.85835080649916e-09 & 5.858364813170763e-09 &
-1.4007e-14\\\cline{2-5}
& (0, 0, 4) & -1.71372127668556e-05 & -1.713524114661759e-05 &
-1.9716e-9\\\cline{2-5}
& (0, 0, 8) & -1.26729956194435e-07 & -1.267386149809553e-07 &
8.6588e-12\\\hline
\end{tabular}
}\caption{Numerical results of $(k_{3},k_{3}^{\prime},s)-$derivatives in
\eqref{testderi} (real parts).}%
\label{Table:gpofapproximationrp}%
\end{table}\begin{table}[ht!]
\centering {\small
\begin{tabular}
[c]{|c|c|c|c|c|}\hline
& $(k_{3}, k_{3}^{\prime}, s)$ & direct quadrature & DCIM & error\\\hline
\multirow{5}{*}{\begin{tabular}{c} $\bs r=(0.5, 1.0, -0.5)$\\ $\bs r'=(0.3, 1.3, -0.5)$ \end{tabular}} &
(0, 0, 0) & 0.00236214962912961 & 0.002362151697708 & -2.0686e-09\\\cline{2-5}
& (3, 4, 0) & -0.00126663970537548 & - 0.001266638701878 &
-1.0035e-09\\\cline{2-5}
& (8, 8, 0) & 1.3083718652325e-06 & 1.308371795385077e-06 &
6.9847e-14\\\cline{2-5}
& (0, 0, 4) & -1.50190394931086e-06 & - 1.501883557929564e-06 &
-2.0391e-11\\\cline{2-5}
& (0, 0, 8) & -2.87922306729206e-13 & - 2.878701763048925e-13 &
-5.2130e-17\\\hline
\multirow{5}{*}{\begin{tabular}{c} $\bs r=(0.6, 0.3, -1.2)$\\ $\bs r'=(0.5, 1.0, -0.5)$ \end{tabular}} &
(0, 0, 0) & -0.0655662374392812 & - 0.065566216753017 &
-2.069e-08\\\cline{2-5}
& (3, 4, 0) & -0.00407200441147604 & - 0.004072001032057 &
-3.3794e-09\\\cline{2-5}
& (8, 8, 0) & -5.8052078071366e-05 & - 5.805409125522990e-05 &
2.0132e-09\\\cline{2-5}
& (0, 0, 4) & 0.000103591338132027 & 1.035831449609178e-04 &
8.1932e-09\\\cline{2-5}
& (0, 0, 8) & 5.90666673167792e-08 & 5.907018897856176e-08 &
-3.5217e-12\\\hline
\end{tabular}
}\caption{Numerical results of $(k_{3},k_{3}^{\prime},s)-$derivatives in
\eqref{testderi} (imaginary parts).}%
\label{Table:gpofapproximationip}%
\end{table}\begin{figure}[ht!]
\subfigure[$k_3'=0$]{\includegraphics[scale=0.25]{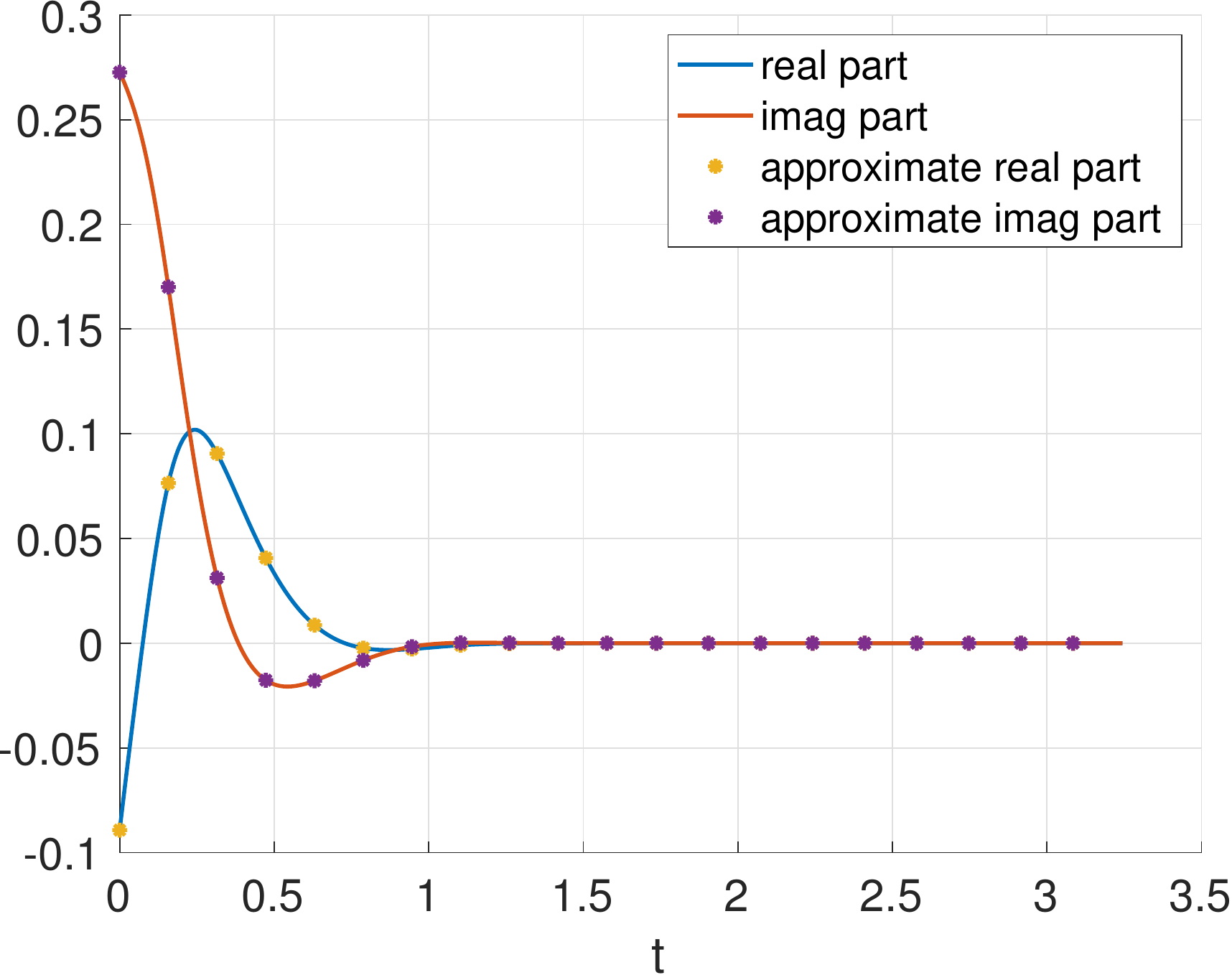}}\quad
\subfigure[$k_3'=4$]{\includegraphics[scale=0.25]{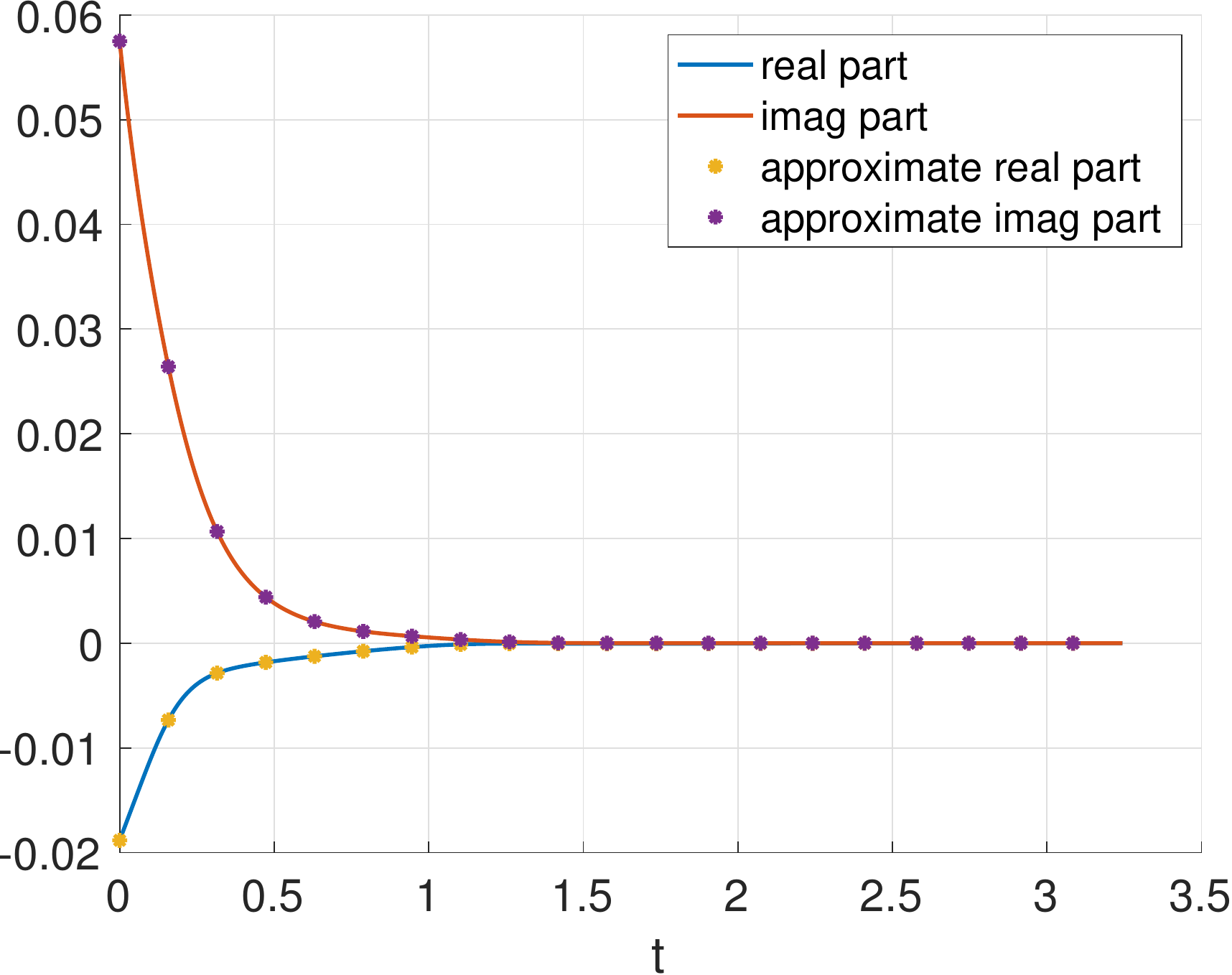}}\quad
\subfigure[$k_3'=8$]{\includegraphics[scale=0.25]{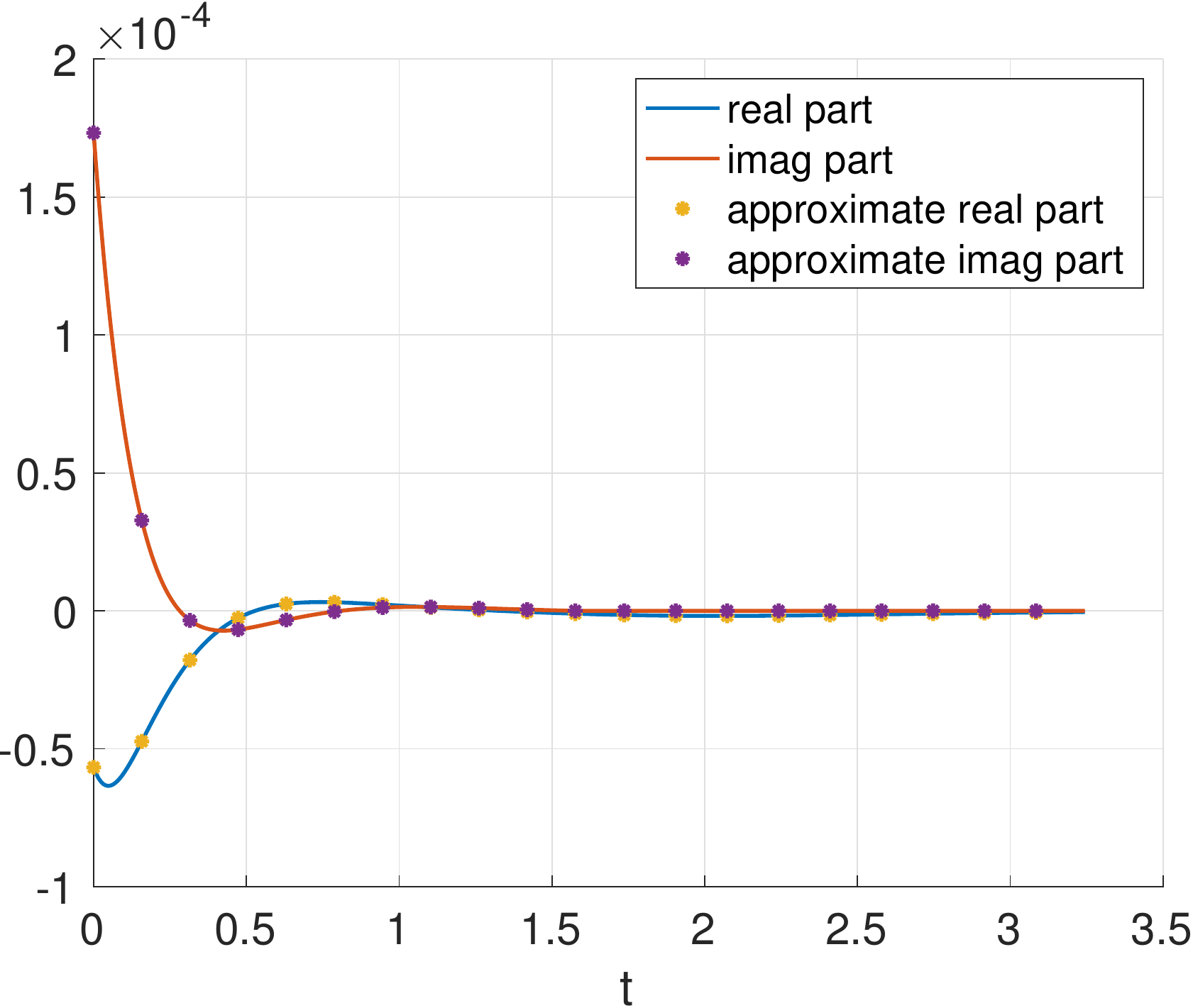}}
\subfigure[$k_3'=0$]{\includegraphics[scale=0.25]{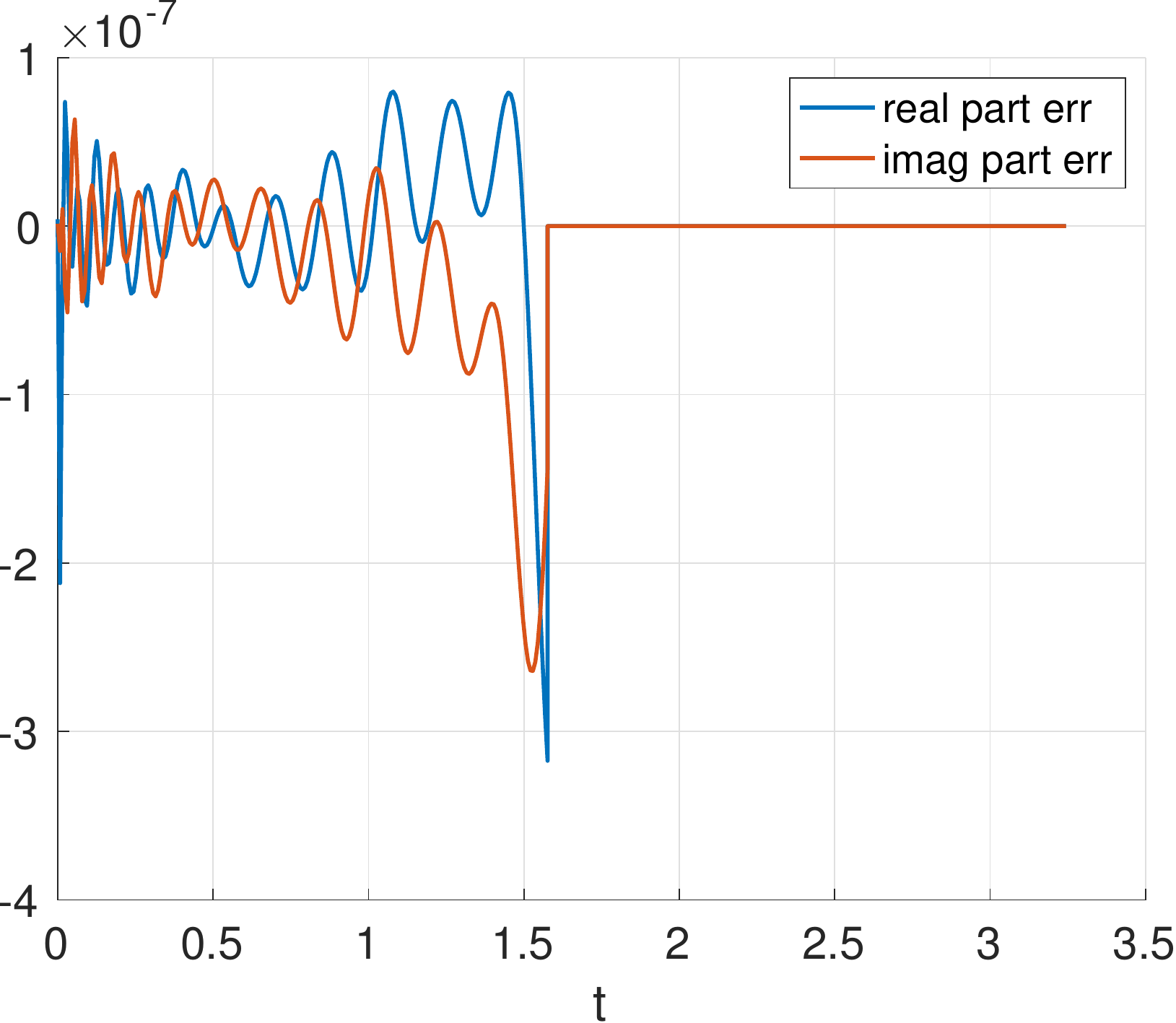}}\quad
\subfigure[$k_3'=4$]{\includegraphics[scale=0.25]{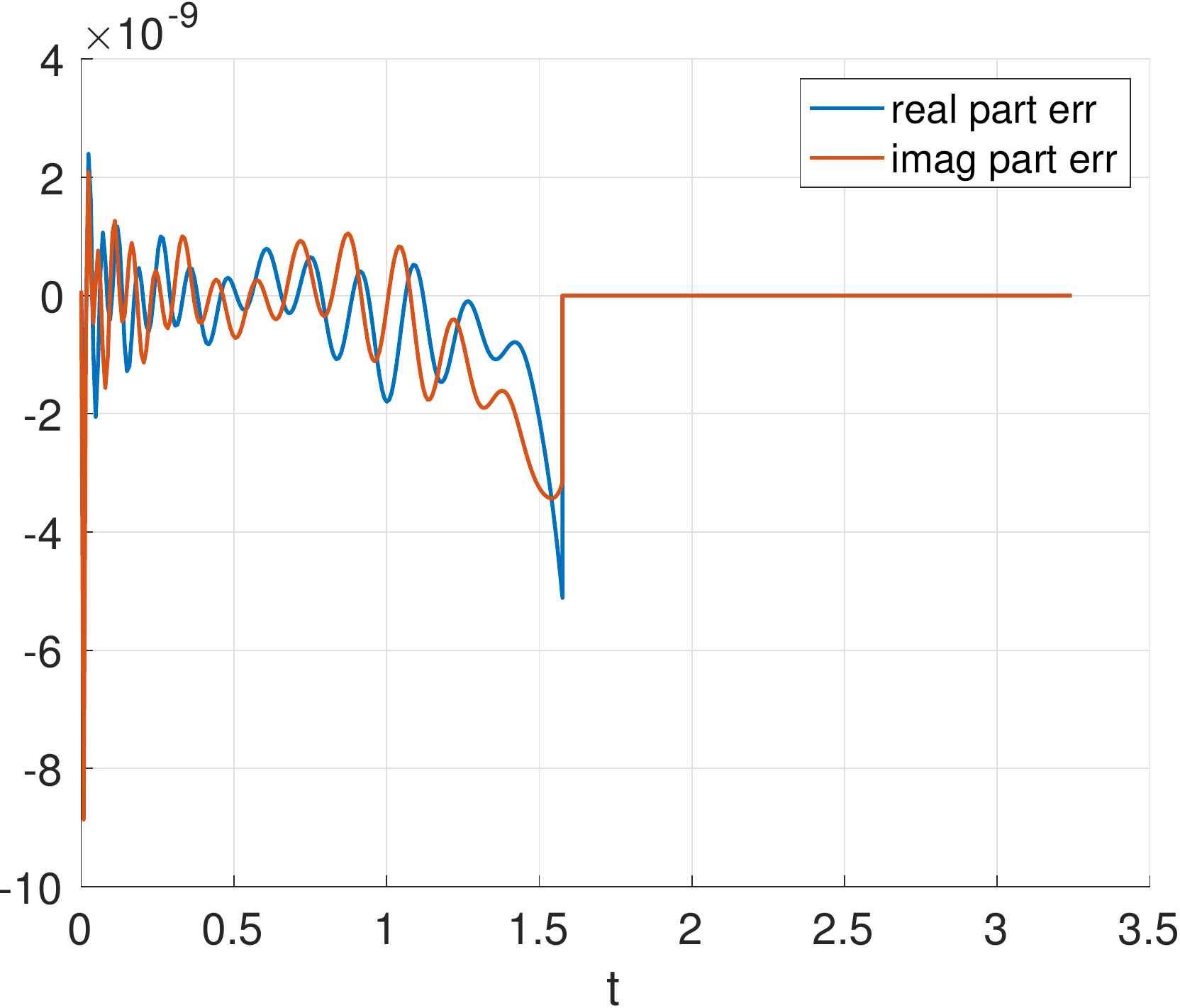}}\quad
\subfigure[$k_3'=8$]{\includegraphics[scale=0.25]{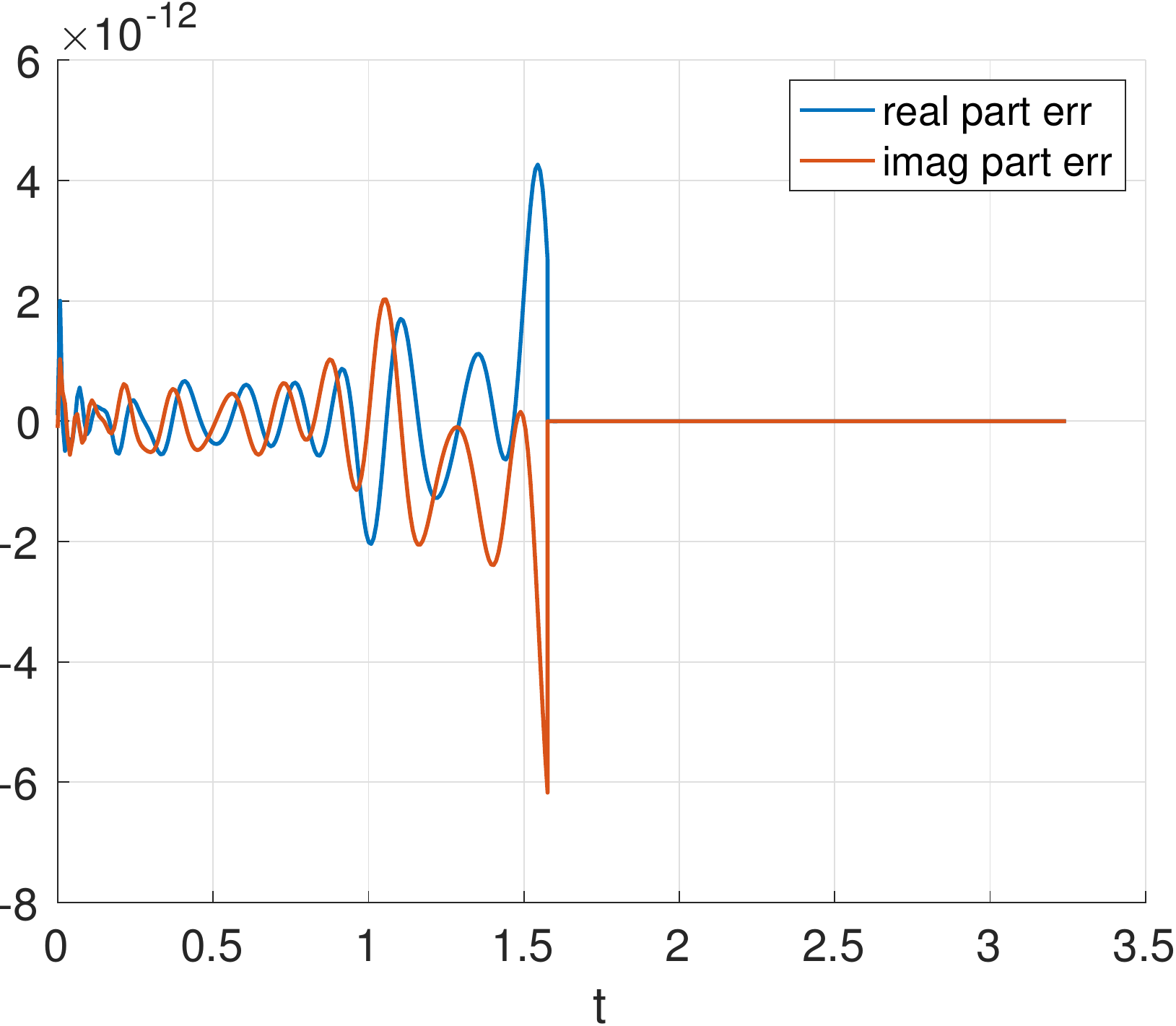}} \caption{Plots
of two level DCIM approximation for $\Theta_{11}^{k_{3}^{\prime}}(k_{\rho
},-0.5)$.}%
\label{gpofapproximation}%
\end{figure}Approximations of $\Theta_{11}^{k_{3}^{\prime}}(k_{\rho},-0.5)$
with $z_{min}=-1.5$ and corresponding errors for different order of
derivatives are depicted in Fig. \ref{gpofapproximation}. Numerical results
obtained by direct quadrature with contour deformation and DCIM approximation
are compared in Table \ref{Table:gpofapproximationrp}%
-\ref{Table:gpofapproximationip}. A large number of Gauss points is used for
the quadrature calculation so a machine accuracy is obtained to be used as
reference values. The numerical results presented in Table
\ref{Table:gpofapproximationrp}-\ref{Table:gpofapproximationip} show that DCIM
can produce approximation with high accuracy even for high order derivatives.
Taking derivative with respect to $x,y,z$ has no degeneracy on the accuracy.
Since derivatives with respect to $x^{\prime},y^{\prime}$ are just a sign
change of that with respect to $x,y$, it also shows that taking derivative
with respect to $x^{\prime},y^{\prime}$ will not degenerate the accuracy neither.

Now we can present two algorithms for the computation of general component
\eqref{generalsum} and total interaction \eqref{totalinteraction},
respectively. \begin{algorithm}\label{algorithm1}
	\caption{TEFMM-I for general component \eqref{generalsum}}
	\begin{algorithmic}
		\State Generate an adaptive hierarchical tree structure and precompute tables.
		\State{\bf Upward pass:}
		\For{$l=H \to 0$}
		\For{all boxes $j$ on source tree level $l$ }
		\If{$j$ is a leaf node}
		\State{form the free-space TE using Eq. \eqref{taylorexpfar3layers}.}
		\Else
		\State form the free-space TE by merging children's expansions using the free-space center shift translation operator \eqref{TETartoTETar}.
		\EndIf
		\EndFor
		\EndFor
		\State{\bf Downward pass:}
		\For{$l=1 \to H$}
		\For{all boxes $j$ on target tree level $l$ }
		\State shift the TE of $j$'s parent to $j$ itself using the free-space translation operator \eqref{freeletole}.
		\State collect interaction list contribution using the source box to target box translation operator in Eq. \eqref{tesourceboxtotargetbox} with  precomputed table for \eqref{compleximageapp} and recurrence formula \eqref{recurrsion}.
		\EndFor
		\EndFor
		\State {\bf Evaluate Local Expansions:}
		\For{each leaf node (childless box)}
		\State evaluate the local expansion at each particle location.
		\EndFor
		\State {\bf Local Direct Interactions:}
		\For{$i=1 \to N$ }
		\State compute Eq. \eqref{generalsum} of target particle $i$ in the neighboring boxes using precomputed table of $u_{\ell\ell'}^{\uparrow}(\bs r, \bs r')$.
		\EndFor
	\end{algorithmic}
\end{algorithm}\begin{algorithm}\label{algorithm2}
	\caption{Taylor expansion based heterogeneous 3-D FMM for \eqref{totalinteraction}}
	\begin{algorithmic}
		\For{$\ell=0 \to L$}
		\For{$\ell'=0 \to L$ }
		\If{$\ell=\ell'$}
		\State{use free space FMM to compute $\Phi_{\ell}^{free}$.}
		\EndIf
		\If{$\ell=0$}
		\State use {\bf Algorithm 1} to compute $\Phi_{0\ell'}^{\uparrow}$.
		\Else
		\If{$\ell=L$}
		\State use {\bf Algorithm 1} to compute $\Phi_{L\ell'}^{\downarrow}$.
		\Else
		\State use {\bf Algorithm 1} to compute $\Phi_{\ell\ell'}^{\uparrow}$.
		\State use {\bf Algorithm 1} to compute $\Phi_{\ell\ell'}^{\downarrow}$.
		\EndIf
		\EndIf
		\EndFor
		\EndFor
	\end{algorithmic}
\end{algorithm}

\section{Second Taylor-Expansion based FMM in multi-layered media}

As discussed in the last section, Taylor expansion based FMM in multi-layered
media depends on an efficient algorithm for the calculation of corresponding
Green's function and its derivatives. The algorithm using discrete complex
image and recurrence formula has good efficiency. However, discrete complex
image approximation may suffer stability problem in the calculation of high
order derivatives. Since the Green's
function in multi-layered media has a symmetry in the $x-y$ plane, it is worthy
to maintain this symmetry. For this purpose, we use notation $\mathcal H_k(\bs r, \bs r')=h_0^{(1)}(k|\bs r-\bs r'|)$ and introduce differential
operators
\begin{equation}%
\begin{split}
\mathscr{D}_{nm}^{s}= &  \Big(\frac{\partial}{\partial x}-\ri\frac{\partial
}{\partial y}\Big)^{s}\Big(\frac{\partial}{\partial x}+\ri\frac{\partial
}{\partial y}\Big)^{m-{s}}\Big(\frac{\partial}{\partial z}\Big)^{n-m},\\
\widehat{\mathscr{D}}_{nm}^{s}= &  \Big(\frac{\partial}{\partial x^{\prime}%
}-\ri\frac{\partial}{\partial y^{\prime}}\Big)^{s}\Big(\frac{\partial
}{\partial x^{\prime}}+\ri\frac{\partial}{\partial y^{\prime}}\Big)^{m-s}%
\Big(\frac{\partial}{\partial z^{\prime}}\Big)^{n-m}.
\end{split}
\label{sysdifferential}%
\end{equation}

\subsection{Free space}

We start with rearranging TE of $h_{0}^{(1)}%
(k|\boldsymbol{r}-\boldsymbol{r}^{\prime}|)$ by using operators
defined in \eqref{sysdifferential}.

\begin{theorem}
\label{taylorexptheorem} Suppose $|\boldsymbol{r}^{\prime}|\leq a$ for a given
small radius $a$, then the Taylor expansion of $h_{0}^{(1)}%
(k|\boldsymbol{r}-\boldsymbol{r}^{\prime}|)$ at origin with respect to
$\boldsymbol{r}^{\prime}$ is
\begin{equation}
\label{teatrc}h_{0}^{(1)}(k|\boldsymbol{r}-\boldsymbol{r}^{\prime}%
|)=\sum_{n=0}^{\infty}\sum_{m=0}^{n}\sum_{s =0}^{m}\tilde{\alpha}_{nm}%
^{s}(\boldsymbol{r}^{\prime})\frac{\widehat{\mathscr D}_{nm}^s\mathcal H_k(\bs r, \bs 0)}{2^{m}(n-m)!s!(m-s)!},
\end{equation}
where
\begin{equation}
\label{eqn:coefMExyz}\quad\tilde{\alpha}_{nm}^{s}(\boldsymbol{r}^{\prime
})=(x^{\prime}+\ri y^{\prime})^{s}(x^{\prime}-\ri y^{\prime})^{m-s}(z^{\prime
})^{n-m}.
\end{equation}

\end{theorem}

\begin{proof}
Denote the spherical coordinates of ${\boldsymbol{r}}^{\prime}$ and
${\boldsymbol{r}}$ as $(\rho,\alpha,\beta)$ and $(r,\theta,\varphi)$,
repsectively. Applying Taylor expansion on $h_{0}^{(1)}(k|{\boldsymbol{r}}%
-{\boldsymbol{r}}^{\prime}|)$ with respect to $\boldsymbol{r}^{\prime
}=(x^{\prime},y^{\prime},z^{\prime})$ at the origin and changing the
derivative with respect to $(x,y,z)$, we have
\begin{equation}%
\begin{split}
h_{0}^{(1)}(k|{\boldsymbol{r}}-{\boldsymbol{r}}^{\prime}|)= &
\mathcal H_k({\boldsymbol{r}}, \bs 0)+\sum_{n=1}^{\infty}\frac{(-1)^{n}%
}{n!}\left(  x^{\prime}\frac{\partial}{\partial x}+y^{\prime}\frac{\partial
}{\partial y}+z^{\prime}\frac{\partial}{\partial z}\right)  ^{n}\mathcal H_k({\boldsymbol{r}}, \bs 0)\\
= &\mathcal H_k({\boldsymbol{r}}, \bs 0)+\sum_{n=1}^{\infty}%
\frac{(-1)^{n}\rho^{n}}{n!}\sum_{m=0}^{n}\binom{n}{m}\sin^{m}\alpha\left(
\cos{\beta}\frac{\partial}{\partial x}+\sin{\beta}\frac{\partial}{\partial
y}\right)  ^{m}\\
&  \qquad\qquad\qquad\qquad\qquad\qquad\qquad\times\cos^{n-m}\alpha\left(
\frac{\partial}{\partial z}\right)  ^{n-m}\mathcal H_k({\boldsymbol{r}}, \bs 0).
\end{split}
\label{eqn:taylor}%
\end{equation}
Notice that for any function $f(z),z\in\mathbb{C}$, we have
\[%
\left(  \cos{\beta}\frac{\partial}{\partial x}+\sin{\beta}\frac{\partial
}{\partial y}\right)  f=\frac{1}{2}\Big[e^{\ri\beta}\left(  \frac{\partial
}{\partial x}-\ri\frac{\partial}{\partial y}\right)  +e^{-\ri\beta}\left(
\frac{\partial}{\partial x}+\ri\frac{\partial}{\partial y}\right)  \Big]f.
\]
Therefore, we can rewrite Eq. (\ref{eqn:taylor}) as
\begin{equation}%
\begin{split}
h_{0}^{(1)}(k|{\boldsymbol{r}}-{\boldsymbol{r}}^{\prime}|)= &  \mathcal H_k({\boldsymbol{r}}, \bs 0)+\sum_{n=1}^{\infty}\sum_{m=0}^{n}\frac
{(-1)^{n}\rho^{n}\sin^{m}\alpha\cos^{n-m}\alpha}{2^{m}(n-m)!}\\
&  \qquad\qquad\qquad\qquad\times\sum_{s=0}^{m}\frac{e^{-\ri(m-s)\beta}e^{\ri
s\beta}}{s!(m-s)!}\mathscr{D}_{nm}^{s}\mathcal H_k({\boldsymbol{r}}, \bs 0)\\
= & \mathcal H_k({\boldsymbol{r}}, \bs 0)+\sum_{n=1}^{\infty}\sum_{m=0}^{n}%
\sum_{s=0}^{m}\tilde{\alpha}_{nm}^{s}(\boldsymbol{r}^{\prime})\frac
{(-1)^{n}\mathscr{D}_{nm}^{s}\mathcal H_k({\boldsymbol{r}}, \bs 0)}{2^{m}(n-m)!s!(m-s)!},
\end{split}
\label{eqn:taylor1}%
\end{equation}
where
\begin{equation}%
\tilde{\alpha}_{nm}^{s}(\boldsymbol{r}^{\prime})=\rho^{n}e^{\ri s\beta
}e^{-\ri(m-s)\beta}\sin^{m}\alpha\cos^{n-m}\alpha=(x^{\prime}+\ri y^{\prime
})^{s}(x^{\prime}-\ri y^{\prime})^{m-s}(z^{\prime})^{n-m}.
\label{eqn:coef1}%
\end{equation}
We finish the proof by using the fact $(-1)^n\mathscr{D}_{nm}^{s}\mathcal H_k({\boldsymbol{r}}, \bs 0)=\widehat{\mathscr{D}}_{nm}^{s}\mathcal H_k({\boldsymbol{r}}, \bs 0)$ in \eqref{eqn:taylor1}.
\end{proof}

\begin{rem}
	The notation $\mathcal H_k({\boldsymbol{r}}, \bs r')$ is used to clearly show that the derivatives $\mathscr{D}_{nm}^{s}h_0^{(1)}(k|{\boldsymbol{r}}-\bs r'|)$ and  $\widehat{\mathscr{D}}_{nm}^{s}h_0^{(1)}(k|{\boldsymbol{r}}-\bs r'|)$ are not just depend on $|\bs r-\bs r'|$ but are functions of $(\bs r, \bs r')$.
\end{rem}

\begin{corollary}
\label{taylorexpcorollary} Suppose $|\boldsymbol{r}|\leq a$ for a given small
radius $a$, then the Taylor expansion of $h_{0}^{(1)}%
(k|\boldsymbol{r}-\boldsymbol{r}^{\prime}|)$ at the origin with respect to
$\boldsymbol{r}$ is
\begin{equation}
h_{0}^{(1)}(k|\boldsymbol{r}-\boldsymbol{r}^{\prime}|)=\sum_{n=0}^{\infty}%
\sum_{m=0}^{n}\sum_{s=0}^{m}{\beta}_{nm}^{s}\tilde{\alpha}_{nm}^{s}%
(\boldsymbol{r})\label{taylorexploc}%
\end{equation}
where
\begin{equation}
{\beta}_{nm}^{s}=\frac{\mathscr{D}_{nm}^{s}\mathcal H_k(\bs 0, \boldsymbol{r}^{\prime})}{2^{m}(n-m)!s!(m-s)!}.\label{eqn:coefM2L}%
\end{equation}

\end{corollary}

With Taylor expansions given in \eqref{teatrc} and \eqref{taylorexploc}, we
can have the second Taylor expansion based FMM which uses the following expansions:

\begin{itemize}
\item \textbf{Taylor expansion (TE) in a source box centered at
$\boldsymbol{r}_{c}$: }
\begin{equation}
\sum\limits_{j\in J_{m}}q_{j}h_{0}^{(1)}(k|\boldsymbol{r}-\boldsymbol{r}_{j}%
|)\approx\sum_{n=0}^{p}\sum_{m=0}^{n}\sum_{s=0}^{m}{\alpha}_{nm}^{s}%
\frac{\widehat{\mathscr{D}}_{nm}^{s}\mathcal H_{k}(\boldsymbol{r}, \boldsymbol{r}_{c})}{2^{m}(n-m)!s!(m-s)!}%
,\label{taylorexpfarxcfree2}%
\end{equation}
where
\begin{equation}
{\alpha}_{nm}^{s}=\sum\limits_{j\in J_{m}}q_{j}\tilde{\alpha}_{nm}%
^{s}(\boldsymbol{r}_{j}-\boldsymbol{r}_{c}).\label{mecoefficients}%
\end{equation}

\item \textbf{Taylor expansion (TE) in a target box centered at
$\boldsymbol{r}_{c}^l$: }
\begin{equation}
\sum\limits_{j\in J_{m}}q_{j}h_{0}^{(1)}(k|\boldsymbol{r}-\boldsymbol{r}_{j}%
|)\approx\sum_{n=0}^{p}\sum_{m=0}^{n}\sum_{s=0}^{m}{\beta}_{nm}^{s}%
\tilde{\alpha}_{nm}^{s}(\boldsymbol{r}-\boldsymbol{r}_{c}^{l}%
),\label{taylorexplocxclfree2}%
\end{equation}
where
\begin{equation}
{\beta}_{nm}^{s}=\sum\limits_{j\in J_{m}}\frac{q_{j}\mathscr{D}_{nm}^{s}\mathcal{H}_{k}(\boldsymbol{r}_{c}^{l}, \boldsymbol{r}_{j})}{2^{m}%
(n-m)!s!(m-s)!}.\label{localexpcoeff2}%
\end{equation}

\end{itemize}

The translation operators used in the FMM algorithm can be derived similarly
as in the conventional way. Firstly, by applying Taylor expansion
\eqref{taylorexpfarxcfree2} in \eqref{localexpcoeff2} and using the fact $\mathcal{H}_{k}(\boldsymbol{r}_{c}^{l}, \boldsymbol{r}_{j})=h_0^{(1)}(k|\bs r^l_c-\bs r_j|)$, we have
\[%
\begin{split}
{\beta}_{nm}^{s}= &  \frac{1}{2^{m}(n-m)!s!(m-s)!}\mathscr{D}_{nm}^{s}%
\sum_{n^{\prime}=0}^{p}\sum_{m^{\prime}=0}^{n^{\prime}}\sum_{s^{\prime}%
=0}^{m^{\prime}}{\alpha}_{n^{\prime}m^{\prime}}^{s^{\prime}}\frac
{\widehat{\mathscr{D}}_{n^{\prime}m^{\prime}}^{s^{\prime}}\mathcal{H}_k(\boldsymbol{r}_{c}^{l}, \boldsymbol{r}_{c})}{2^{m^{\prime}}(n^{\prime
}-m^{\prime})!s^{\prime}!(m^{\prime}-s^{\prime})!}\\
= &  \sum_{n^{\prime}=0}^{p}\sum_{m^{\prime}=0}^{n^{\prime}}\sum_{s^{\prime
}=0}^{m^{\prime}}{\alpha}_{n^{\prime}m^{\prime}}^{s^{\prime}}\frac
{\mathscr{D}_{nm}^{s}\widehat{\mathscr{D}}_{n^{\prime}m^{\prime}}^{s^{\prime}}\mathcal{H}_k(\boldsymbol{r}_{c}^{l}, \boldsymbol{r}_{c})}{2^{m+m^{\prime}%
}(n-m)!s!(m-s)!(n^{\prime}-m^{\prime})!s^{\prime}!(m^{\prime}-s^{\prime})!}.
\end{split}
\]
Therefore, the translation operator from TE in a source box centered
at $\boldsymbol{r}_{c}$ to TE in a target box centered at $\boldsymbol{r}_{c}%
^{l}$ is given by
\begin{equation}
{\beta}_{nm}^{s}=\sum_{n^{\prime}=0}^{p}\sum_{m^{\prime}=0}^{n^{\prime}}%
\sum_{s^{\prime}=0}^{m^{\prime}}{\alpha}_{n^{\prime}m^{\prime}}^{s^{\prime}%
}L_{nms}^{n^{\prime}m^{\prime}s^{\prime}},
\end{equation}
where
\begin{equation}%
L_{nms}^{n^{\prime}m^{\prime}s^{\prime}}=  \frac{\mathscr{D}_{nm}^{s}\widehat{\mathscr{D}}_{n^{\prime}m^{\prime}}^{s^{\prime}}\mathcal{H}_k(\boldsymbol{r}_{c}^{l}, \boldsymbol{r}_{c})}{2^{m+m^{\prime}}(n-m)!s!(m-s)!(n^{\prime}-m^{\prime})!s^{\prime
}!(m^{\prime}-s^{\prime})!}.
\label{M2Ltransop}%
\end{equation}
Denote the coefficients of TE in the source box centered at
$\boldsymbol{r}_{c}^{\prime}$ by
\[
{\gamma}_{nm}^{s}=\sum\limits_{j=1}^{N}q_{j}\tilde{\alpha}_{nm}^{s}%
(\boldsymbol{r}_{j}-\boldsymbol{r}_{c}^{\prime}).
\]
Direct calculation gives
\begin{equation}%
\begin{split}
&  [(x_{j}-x_{c}^{\prime})+\ri(y_{j}-y_{c}^{\prime})]^{s}[(x_{j}-x_{c}%
^{\prime})-\ri(y_{j}-y_{c}^{\prime})]^{m-s}(z_{j}-z_{c}^{\prime})^{n-m}\\
= &  \sum\limits_{s^{\prime}=0}^{s}\sum\limits_{m^{\prime}=0}^{m-s}%
\sum\limits_{n^{\prime}=0}^{n-m}B_{nms}^{n^{\prime}m^{\prime}s^{\prime}%
}[(x_{j}-x_{c})+\ri(y_{j}-y_{c})]^{s^{\prime}}[(x_{j}-x_{c})-\ri(y_{j}%
-y_{c})]^{m^{\prime}}(z_{j}-z_{c})^{n^{\prime}}%
\end{split}
\end{equation}
where
\begin{equation}%
\begin{split}
B_{nms}^{n^{\prime}m^{\prime}s^{\prime}}= &  \frac{s!(m-s)!(n-m)![(x_{c}%
-x_{c}^{\prime})+\ri(y_{c}-y_{c}^{\prime})]^{s-s^{\prime}}}{s^{\prime
}!(s-s^{\prime})!m^{\prime}!(m-s-m^{\prime})!n^{\prime}!(n-m-n^{\prime})!}\\
&  \times\lbrack(x_{c}-x_{c}^{\prime})-\ri(y_{c}-y_{c}^{\prime}%
)]^{m-s-m^{\prime}}(z_{c}-z_{c}^{\prime})^{n-m-n^{\prime}}.
\end{split}
\end{equation}
Therefore,
\begin{equation}
\tilde{\alpha}_{nm}^{s}(\boldsymbol{r}_{j}-\boldsymbol{r}_{c}^{\prime}%
)=\sum\limits_{s^{\prime}=0}^{s}\sum\limits_{m^{\prime}=0}^{m-s}%
\sum\limits_{n^{\prime}=0}^{n-m}B_{nms}^{n^{\prime}m^{\prime}s^{\prime}}%
\tilde{\alpha}_{n^{\prime}+m^{\prime}+s^{\prime},m^{\prime}+s^{\prime}%
}^{s^{\prime}}(\boldsymbol{r}_{j}-\boldsymbol{r}_{c}),
\end{equation}
which implies that the translation operator from TE in a source box
centered at $\boldsymbol{r}_{c}$ to TE in another source box centered at
$\boldsymbol{r}_{c}^{\prime}$ has the form
\begin{equation}
\gamma_{nm}^{s}=\sum\limits_{s^{\prime}=0}^{s}\sum\limits_{m^{\prime}=0}%
^{m-s}\sum\limits_{n^{\prime}=0}^{n-m}B_{nms}^{n^{\prime}m^{\prime}s^{\prime}%
}{\alpha}_{n^{\prime}+m^{\prime}+s^{\prime},m^{\prime}+s^{\prime}}^{s^{\prime
}}.\label{me2mesymmetricte}%
\end{equation}

Let
\begin{equation}
{\lambda}_{nm}^{s}=\sum\limits_{j\in J_{m}}\frac{q_{j}\mathscr{D}_{nm}^{s}\mathcal{H}_{k}(\tilde{\boldsymbol{r}}_{c}^{l},\boldsymbol{r}_{j})}%
{2^{m}(n-m)!s!(m-s)!},
\end{equation}
be the coefficients of TE in a target box centered at $\tilde{\boldsymbol{r}}%
_{c}^{l}$. By applying Taylor expansion at $\bs r_c^l$ we obtain
\begin{equation}%
\begin{split}
{\lambda}_{nm}^{s}= &  \frac{1}{2^{m}(n-m)!s!(m-s)!}{\mathscr D}_{nm}^{s}%
\sum\limits_{j\in J_{m}}q_{j}h_{0}^{(1)}(k|\tilde{\boldsymbol{r}}_{c}%
^{l}-\boldsymbol{r}_{j}|)\\
\approx &  \frac{1}{2^{m}(n-m)!s!(m-s)!}\sum_{n^{\prime}=0}^{p}\sum
_{m^{\prime}=0}^{n^{\prime}}\sum_{s^{\prime}=0}^{m^{\prime}}{\beta}%
_{n^{\prime}m^{\prime}}^{s^{\prime}}{\mathscr D}_{nm}^{s}\tilde{\alpha
}_{n^{\prime}m^{\prime}}^{s^{\prime}}(\tilde{\boldsymbol{r}}_{c}%
^{l}-\boldsymbol{r}_{c}^{l}).
\end{split}
\label{taylorexplocxcltrans}%
\end{equation}
Note that
\[%
\begin{split}
&  \Big(\frac{\partial}{\partial\tilde{x}_{c}^{l}}-\ri\frac{\partial}%
{\partial\tilde{y}_{c}^{l}}\Big)^{s}\Big(\frac{\partial}{\partial\tilde{x}%
_{c}^{l}}+\ri\frac{\partial}{\partial\tilde{y}_{c}^{l}}\Big)^{m-s}[(\tilde
{x}_{c}^{l}-x_{c}^{l})+\ri(\tilde{y}_{c}^{l}-y_{c}^{l})]^{s^{\prime}}%
[(\tilde{x}_{c}^{l}-x_{c}^{l})-\ri(\tilde{y}_{c}^{l}-y_{c}^{l})]^{m^{\prime
}-s^{\prime}}\\
= &
\begin{cases}
\displaystyle0,\quad m-s>m^{\prime}-s^{\prime}\;\;\mathrm{or}\;\;s>s^{\prime
},\\[6pt]%
\displaystyle\frac{2^{m}s^{\prime}!(m^{\prime}-s^{\prime})![(\tilde{x}_{c}%
^{l}-x_{c}^{l})+\ri(\tilde{y}_{c}^{l}-y_{c}^{l})]^{s^{\prime}-s}}{(s^{\prime
}-s)!(m^{\prime}-s^{\prime}-m+s)!}[(\tilde{x}_{c}^{l}-x_{c}^{l})-\ri(\tilde
{y}_{c}^{l}-y_{c}^{l})]^{m^{\prime}-s^{\prime}-m+s},\quad\mathrm{otherwise},
\end{cases}
\\
&  \Big(\frac{\partial}{\partial\tilde{z}_{c}^{l}}\Big)^{n-m}(\tilde{z}%
_{c}^{l}-z_{c}^{l})^{n^{\prime}-m^{\prime}}=%
\begin{cases}
\displaystyle0,\quad\mathrm{if}\;n-m>n^{\prime}-m^{\prime},\\
\displaystyle\frac{(n^{\prime}-m^{\prime})!}{(n^{\prime}-m^{\prime}%
-(n-m))!}(\tilde{z}_{c}^{l}-z_{c}^{l})^{n^{\prime}-m^{\prime}-n+m}%
,\;\mathrm{otherwise}.\\
\end{cases}
\end{split}
\]
Therefore,
\begin{equation}%
{\mathscr D}_{nm}^{s}\tilde{\alpha}_{n^{\prime}m^{\prime}}^{s^{\prime}}%
(\tilde{\boldsymbol{r}}_{c}^{l}-\boldsymbol{r}_{c}^{l})=%
\begin{cases}
\displaystyle0,\quad\mathrm{if}\;\;m-s>m^{\prime}-s^{\prime}\;\;\mathrm{or}%
\;\;s>s^{\prime}\;\;\mathrm{or}\;\;\;n-m>n^{\prime}-m^{\prime};\\[10pt]%
\displaystyle2^{m}C_{nms}^{n^{\prime}m^{\prime}s^{\prime}}\tilde{\alpha
}_{n^{\prime}-n,m^{\prime}-m}^{s^{\prime}-s}(\tilde{\boldsymbol{r}}_{c}%
^{l}-\boldsymbol{r}_{c}^{l}),\quad\mathrm{otherwise},
\end{cases}
\label{LE2LErans2}%
\end{equation}
where
\[
C_{nms}^{n^{\prime}m^{\prime}s^{\prime}}=\frac{(m^{\prime}-s^{\prime
})!s^{\prime}!(n^{\prime}-m^{\prime})!}{(m^{\prime}-s^{\prime}-m+s)!(s^{\prime
}-s)!(n^{\prime}-m^{\prime}-(n-m))!}.
\]
Substituting \eqref{LE2LErans2} into \eqref{taylorexplocxcltrans} gives the translation
operator from TE in a target box centered at $\boldsymbol{r}_{c}^{l}$ to TE in
another target box centered at $\tilde{\boldsymbol{r}}_{c}^{l}$
\begin{equation}%
{\lambda}_{nm}^{s}=\sum\limits_{n^{\prime}=n}^{p}\sum\limits_{m^{\prime}%
=m}^{m+n^{\prime}-n}\sum_{s^{\prime}=s}^{s+m^{\prime}-m}\beta_{n^{\prime
}m^{\prime}}^{s^{\prime}}\frac{C_{nms}^{n^{\prime}m^{\prime}s^{\prime}}%
\tilde{\alpha}_{n^{\prime}-n,m^{\prime}-m}^{s^{\prime}-s}(\tilde
{\boldsymbol{r}}_{c}^{l}-\boldsymbol{r}_{c}^{l})}{(n-m)!s!(m-s)!}.
\label{le2lesymmetricte}%
\end{equation}

To ensure the efficiency of this algorithm, a fast algorithm is needed for the
calculation of derivatives $\mathscr{D}_{nm}^{s}\mathcal H_k(\bs r, \bs r')$ and $\mathscr{D}_{nm}^{s}\widehat{\mathscr{D}}_{n^{\prime}m^{\prime}}^{s^{\prime}}\mathcal H_k(\bs r, \bs r')$ which are used in the computation of coefficients
\eqref{localexpcoeff2} and translation operators \eqref{M2Ltransop}. A
recurrence formula can be derived from the following result (cf.\cite{martin2006multiple}). Define
\begin{equation}
\Omega_{n}^{m}(\boldsymbol{r})=h_{n}^{(1)}(k|\boldsymbol{r}|)Y_{n}^{m}%
(\theta,\phi),\quad\mathscr D^{\pm}=\frac{1}{k}\Big(\frac{\partial}{\partial
x}\pm\ri\frac{\partial}{\partial y}\Big),\quad\mathscr D^{0}=-\frac{1}{k}%
\frac{\partial}{\partial z},
\end{equation}
where
\begin{equation}
Y_{n}^{m}(\theta,\phi)=(-1)^{m}\sqrt{\frac{2n+1}{4\pi}\frac{(n-m)!}{(n+m)!}%
}P_{l}^{m}(\cos\theta)e^{\ri m\phi},
\end{equation}
is the spherical harmonics.

\begin{theorem}
\label{Thm:hankelderiative} For $0\leq|m|\leq n$,
\begin{equation}%
\begin{split}
&  \mathscr D^{+}\Omega_{n}^{m}=A_{nm}^{+}\Omega_{n+1}^{m+1}+B_{nm}^{+}%
\Omega_{n-1}^{m+1},\\
&  \mathscr D^{-}\Omega_{n}^{m}=A_{nm}^{-}\Omega_{n+1}^{m-1}+B_{nm}^{-}%
\Omega_{n-1}^{m-1},\\
&  \mathscr D^{0}\Omega_{n}^{m}=A_{nm}^{0}\Omega_{n+1}^{m}+B_{nm}^{0}%
\Omega_{n-1}^{m}.
\end{split}
\end{equation}
where
\[%
\begin{split}
&  A_{nm}^{+}=\sqrt{\frac{(n+m+2)(n+m+1)}{(2n+1)(2n+3)}},\quad B_{nm}%
^{+}=\sqrt{\frac{(n-m)(n-m-1)}{4n^{2}-1}},\\
&  A_{nm}^{-}=-\sqrt{\frac{(n-m+2)(n-m+1)}{(2n+1)(2n+3)}},\quad B_{nm}%
^{-}=-\sqrt{\frac{(n+m)(n+m-1)}{4n^{2}-1}},\\
&  A_{nm}^{0}=-\sqrt{\frac{(n+1)^{2}-m^{2}}{(2n+1)(2n+3)}},\quad B_{nm}%
^{0}=\sqrt{\frac{n^{2}-m^{2}}{4n^{2}-1}}.
\end{split}
\]
In particular, for $n\ge0$
\begin{equation}
\mathscr D^{0}\Omega_{n}^{\pm n}=\frac{1}{\sqrt{2n+3}}\Omega_{n+1}^{\pm n}.
\end{equation}

\end{theorem}

From the theorem \ref{Thm:hankelderiative}, the high order derivatives can be expressed as
\[
\{(\mathscr D^{+})^{s},(\mathscr D^{-})^{s},(\mathscr D^{0})^{s}\}\Omega
_{n}^{m}=\sum\limits_{r=0}^{s}\big\{{C}_{rs}^{+}\Omega_{n-s+2r}^{m+s},{C}%
_{rs}^{-}\Omega_{n-s+2r}^{m-s},C_{rs}^{0}\Omega_{n-s+2r}^{m}\big\},
\]
where the coefficients $\{{C}_{rs}^{+}\},\{C_{rs}^{-}\},\{{C}_{rs}^{0}\}$ have
a recurrence formula
\[%
\begin{split}
&  C_{rs}^{+}=%
\begin{cases}
\displaystyle0, & n-s+2r<m+s,\\[5pt]%
\displaystyle B_{n-s+1,m+s-1}^{+}C_{0,s-1}^{+}, & r=0,\\[5pt]%
\displaystyle A_{n+s-1,m+s-1}^{+}C_{s-1,s-1}^{+}, & r=s,\\[5pt]%
\displaystyle A_{n-s+2r-1,m+s-1}^{+}C_{r-1,s-1}^{+}+B_{n-s+2r+1,m+s-1}%
^{+}C_{r,s-1}^{+}, & 0<r<s,
\end{cases}
\\
&  C_{rs}^{-}=%
\begin{cases}
\displaystyle0, & n-s+2r<|m-s|,\\[5pt]%
\displaystyle B_{n-s+1,m-s+1}^{-}C_{0,s-1}^{-}, & r=0,\\[5pt]%
\displaystyle A_{n+s-1,m-s+1}^{-}C_{s-1,s-1}^{-}, & r=s,\\[5pt]%
\displaystyle A_{n-s+2r-1,m-s+1}^{-}C_{r-1,s-1}^{-}+B_{n-s+2r+1,m-s+1}%
^{-}C_{r,s-1}^{-}, & 0<r<s,
\end{cases}
\\
&  {C}_{rs}^{0}=%
\begin{cases}
\displaystyle0, & n-s+2r<m,\\[5pt]%
\displaystyle B_{n-s+1,m}^{0}C_{0,s-1}^{0}, & r=0,\\[5pt]%
\displaystyle A_{n+s-1,m}^{0}C_{s-1,s-1}^{0}, & r=s,\\[5pt]%
\displaystyle A_{n-(s-2r)-1,m}^{0}C_{r-1,s-1}^{0}+B_{n-(s-2r)+1,m}%
^{0}C_{r,s-1}^{0}, & 0<r<s,
\end{cases}
\end{split}
\]
with initial values
\[
C_{00}^{+}=C_{00}^{-}={C}_{00}^{0}=1.
\]
Therefore,
\begin{equation}%
\begin{split}
\mathscr D_{nm}^{s}\mathcal H_k(\boldsymbol{r},\bs r^{\prime})= &  \Big(\frac{-1}%
{k}\Big)^{n}\sqrt{\frac{1}{4\pi}}(\mathscr D^{-})^{s}(\mathscr D^{+}%
)^{m-s}(\mathscr D^{0})^{n-m}\Omega_{0}^{0}(\boldsymbol{r}-\boldsymbol{r}^{\prime})\\
= &  \Big(\frac{-1}{k}\Big)^{n}\sqrt{\frac{1}{4\pi}}\sum\limits_{n^{\prime}%
=0}^{n-m}\sum\limits_{m^{\prime}=0}^{m-s}\sum\limits_{s^{\prime}=0}%
^{s}C_{n^{\prime}m^{\prime}}^{s^{\prime}}\Omega_{2(n^{\prime}+m^{\prime
}+s^{\prime})-n}^{m-2s}(\boldsymbol{r}-\boldsymbol{r}^{\prime}),
\end{split}
\end{equation}
where the coefficients $\{C_{n^{\prime}m^{\prime}}^{s^{\prime}}\}$ can be
computed using coefficients $\{{C}_{rs}^{+}\},\{C_{rs}^{-}\},\{{C}_{rs}^{0}\}$. This formula is also used for the calculation of $\mathscr{D}_{nm}^{s}\widehat{\mathscr{D}}_{n^{\prime}m^{\prime}}^{s^{\prime}}\mathcal H_k(\bs r, \bs r')=(-1)^{n^{\prime}}\mathscr{D}_{n+n^{\prime},m+m^{\prime}}^{s+s^{\prime}}\mathcal H_k(\bs r, \bs r')$.

\subsection{Multi-layer media}

Consider the calculation of interactions given by \eqref{totalinteraction}
with the setting presented in the last section. We only need to focus on a
general component given by \eqref{generalsum}. According to
\eqref{taylorexpfarxcfree2}-\eqref{localexpcoeff2}, the TE-FMM for
\eqref{generalsum} will use Taylor expansions
\begin{equation}
\Phi_{\ell\ell^{\prime}}^{b\uparrow}(\boldsymbol{r}_{\ell i})=\sum
_{n=0}^{\infty}\sum_{m=0}^{n}\sum_{s=0}^{m}{\alpha}_{nm}^{s}\frac
{\widehat{\mathscr{D}}_{nm}^{s}u_{\ell\ell^{\prime}}^{\uparrow}%
(\boldsymbol{r}_{\ell i},\boldsymbol{r}_{c})}{2^{m}(n-m)!s!(m-s)!}%
,\quad{\alpha}_{nm}^{s}=\sum\limits_{j\in J_{m}}Q_{\ell^{\prime}j}%
\tilde{\alpha}_{nm}^{s}(\boldsymbol{r}_{\ell^{\prime}j}-\boldsymbol{r}_{c}%
),\label{taylorexpfarlayers}%
\end{equation}
in a source box centered at $\boldsymbol{r}_{c}=(x_{c},y_{c},z_{c})$ and
\begin{equation}
\Phi_{\ell\ell^{\prime}}^{b\uparrow}(\boldsymbol{r}_{\ell i})=\sum_{n=0
}^{p}\sum_{m=0}^{n}\sum_{s=0}^{m}{\beta}_{nm}^{s}\tilde{\alpha}_{nm}%
^{s}(\boldsymbol{r}_{\ell i}-\boldsymbol{r}_{c}^{l}),\quad{\beta}_{nm}%
^{s}=\sum\limits_{j\in J_{m}}\frac{Q_{\ell^{\prime}j}\mathscr{D}_{nm}%
^{s}u_{\ell\ell^{\prime}}^{\uparrow}(\boldsymbol{r}_{c}^{l}%
,\boldsymbol{r}_{\ell^{\prime}j})}{2^{m}(n-m)!s!(m-s)!}%
,\label{taylorexploclayers}%
\end{equation}
in a target box centered at $\boldsymbol{r}_{c}^{l}=(x_{c}^{l},y_{c}^{l}%
,z_{c}^{l})$, respectively.

Applying TE \eqref{taylorexpfarlayers} in the expression of the coefficients
$\beta_{nm}^{s}$ in \eqref{taylorexploclayers}, we obtain
\begin{equation}
{\beta}_{nm}^{s}\approx\sum_{n^{\prime}=0}^{p}\sum_{m^{\prime}=0}^{n^{\prime}%
}\sum_{s^{\prime}=0}^{m^{\prime}}{\alpha}_{n^{\prime}m^{\prime}}^{s^{\prime}%
}\frac{\mathscr{D}_{nm}^{s}\widehat{\mathscr{D}}_{n^{\prime},m^{\prime}%
}^{s^{\prime}}u_{\ell\ell^{\prime}}^{\uparrow}(\boldsymbol{r}_{c}%
^{l},\boldsymbol{r}_{c})}{M_{nms}^{n^{\prime}m^{\prime}s^{\prime}}%
},\label{layerm2ltrans}%
\end{equation}
where
\[
M_{nms}^{n^{\prime}m^{\prime}s^{\prime}}=2^{m+m^{\prime}}%
(n-m)!s!(m-s)!(n^{\prime}-m^{\prime})!s^{\prime}!(m^{\prime}-s^{\prime})!.
\]
Due to the symmetry of differential operators $\mathscr{D}_{nm}^{s}$ and
$\widehat{\mathscr{D}}_{n^{\prime}m^{\prime}}^{s^{\prime}}$, the entry of the
translation matrix in \eqref{layerm2ltrans} also has a symmetry in the $x-y$
plane. This can be shown by using Sommerfeld integral representation
\eqref{greenfuncomponent}. In fact, by using the identities
\begin{equation}%
\begin{split}
&  \Big(\frac{\partial}{\partial x}+\ri\frac{\partial}{\partial y}%
\Big)^{m-{s}}J_{0}(kr)=(-k)^{m-s}J_{m-s}(kr)e^{\ri(m-s)\theta},\\
&  \Big(\frac{\partial}{\partial x}-\ri\frac{\partial}{\partial y}%
\Big)^{s}\Big(J_{m-s}(kr)e^{\ri(m-s)\theta}\Big)=k^{s}J_{m-2s}%
(kr)e^{\ri(m-2s)\theta},
\end{split}
\end{equation}
we have
\begin{equation}%
\begin{split}
&  \frac{\mathscr{D}_{nm}^{s}\widehat{\mathscr{D}}_{n^{\prime},m^{\prime}%
}^{s^{\prime}}u_{\ell\ell^{\prime}}^{\uparrow}(\boldsymbol{r}_{c}%
^{l},\boldsymbol{r}_{c})}{M_{nms}^{n^{\prime}m^{\prime}s^{\prime}}}%
=\frac{\mathscr{D}_{nm}^{s}\widehat{\mathscr{D}}_{n^{\prime}m^{\prime}%
}^{s^{\prime}}}{M_{nms}^{n^{\prime}m^{\prime}s^{\prime}}}\frac{1}{k_{\ell}}%
\int_{0}^{\infty}k_{\rho}J_{0}(k_{\rho}\rho)\frac{e^{\ri k_{\ell z}(z_{c}%
^{l}-d_{\ell})}}{k_{\ell z}}\tilde{\sigma}_{\ell\ell^{\prime}}^{\uparrow
}(k_{\rho},z_{c})dk_{\rho}\\
= &  \frac{(-1)^{m+s+s^{\prime}}e^{\ri(m+m^{\prime}-2(s+s^{\prime}))\phi}%
}{M_{nms}^{n^{\prime}m^{\prime}s^{\prime}}}\frac{1}{k_{\ell}}\int_{0}^{\infty}k_{\rho
}^{m+m^{\prime}+1}J_{m+m^{\prime}-2(s+s^{\prime})}(k_{\rho}\rho)\\
&  \times(\ri k_{\ell z})^{n-m}\frac{e^{\ri k_{\ell z}(z_{c}^{l}-d_{\ell})}%
}{k_{\ell z}}\frac{\partial^{n^{\prime}-m^{\prime}}\tilde{\sigma}_{\ell
\ell^{\prime}}^{\uparrow}(k_{\rho},z_{c})}{\partial z^{\prime}}dk_{\rho},
\end{split}
\end{equation}
where $(\rho,\phi)$ is the polar coordinates of $(x_{c}^{l}-x_{c},y_{c}%
^{l}-y_{c})$
\[
\frac{\partial^{n^{\prime}-m^{\prime}}\tilde{\sigma}_{\ell\ell^{\prime}%
}^{\uparrow}(k_{\rho},z_{c})}{\partial z^{\prime}}=(\ri k_{\ell^{\prime}%
z})^{n^{\prime}-m^{\prime}}\Big(e^{\ri k_{\ell^{\prime}z}(z_{c}-d_{\ell
^{\prime}})}\sigma_{\ell\ell^{\prime}}^{\uparrow\uparrow}(k_{\rho
})+(-1)^{n^{\prime}-m^{\prime}}e^{\ri k_{\ell^{\prime}z}(d_{\ell^{\prime}%
-1}-z_{c})}\sigma_{\ell\ell^{\prime}}^{\uparrow\downarrow}(k_{\rho})\Big).
\]
For general integer indices $n,n^{\prime},m,m^{\prime}$, define integrals
\begin{equation}
\mathcal{S}_{nn^{\prime}}^{mm^{\prime}}(\rho,z,z^{\prime})=\frac{1}{k_{\ell}}\int_{0}^{\infty}\frac{k_{\rho}^{m+1}J_{m-2m^{\prime}}(k_{\rho}\rho)(\ri
k_{\ell z})^{n}}{2^{m}m!n!n^{\prime}!}\frac{e^{\ri k_{\ell z}(z-d_{\ell})}%
}{k_{\ell z}}\frac{\partial^{n^{\prime}}\tilde{\sigma}_{\ell\ell^{\prime}%
}^{\uparrow}(k_{\rho},z^{\prime})}{\partial z^{\prime}}dk_{\rho}%
.\label{symmetricdetable}%
\end{equation}
Then
\[
\frac{\mathscr{D}_{nm}^{s}\widehat{\mathscr{D}}_{n^{\prime},m^{\prime}%
}^{s^{\prime}}u_{\ell\ell^{\prime}}^{\uparrow}(\boldsymbol{r}_{c}%
^{l},\boldsymbol{r}_{c})}{M_{nms}^{n^{\prime}m^{\prime}s^{\prime}}}%
=\frac{(-1)^{m+s+s^{\prime}}(m+m^{\prime})!}{s!(m-s)!s^{\prime}!(m^{\prime
}-s^{\prime})!}e^{\ri(m+m^{\prime}-2(s+s^{\prime}))\phi}\mathcal{S}%
_{n-m,n^{\prime}-m^{\prime}}^{m+m^{\prime},s+s^{\prime}}(\rho,z_{c}^{l}%
,z_{c}).
\]
We pre-compute integrals $\mathcal{S}_{nn^{\prime}}^{mm^{\prime}}%
(\rho,z,z^{\prime})$ on a 3D grid $\{\rho_{i},z_{j},z_{k}^{\prime}\}$ in the
domain of interest for all $n,m=0,1,\cdots,p$; $s=0,1,\cdots,2p$; $s^{\prime
}=0,1,\cdots,s$. Then, a polynomial interpolation is performed for the
computation of derivatives $\mathscr{D}_{nm}^{s}\widehat{\mathscr{D}}%
_{n^{\prime},m^{\prime}}^{s^{\prime}}u_{\ell\ell^{\prime}}^{\uparrow
}(\boldsymbol{r}_{c}^{l},\boldsymbol{r}_{c})$ in the translation operators.

The computation of Sommerfeld integrals similar to $\mathcal{S}_{nn^{\prime}%
}^{mm^{\prime}}(\rho,z,z^{\prime})$ is a standard problem in acoustic and
electromagnetic scattering and often handled by contour deformation. It is
typical to deform the integration contour by pushing it away from the real
line into the fourth quadrant of the complex $k_{\rho}$-plane to avoid branch
points and poles in the integrand. Here, we use a piece-wise smooth contour
which consists of two segments:
\begin{equation}
\Gamma_{1}:\;\;\{k_{\rho}=\ri t,-b\leq t\leq0\},\quad\Gamma_{2}%
:\;\;\{k_{\rho}=t-\ri b,\quad0\leq t<\infty\}.
\end{equation}
We truncate $\Gamma_{2}$ at a point $t_{max}>0$, where the integrand has
decayed to a user specified tolerance.
\begin{figure}[ht!]
	\centering
	\subfigure[$(1, 3, 3, 2)$]{\includegraphics[scale=0.25]{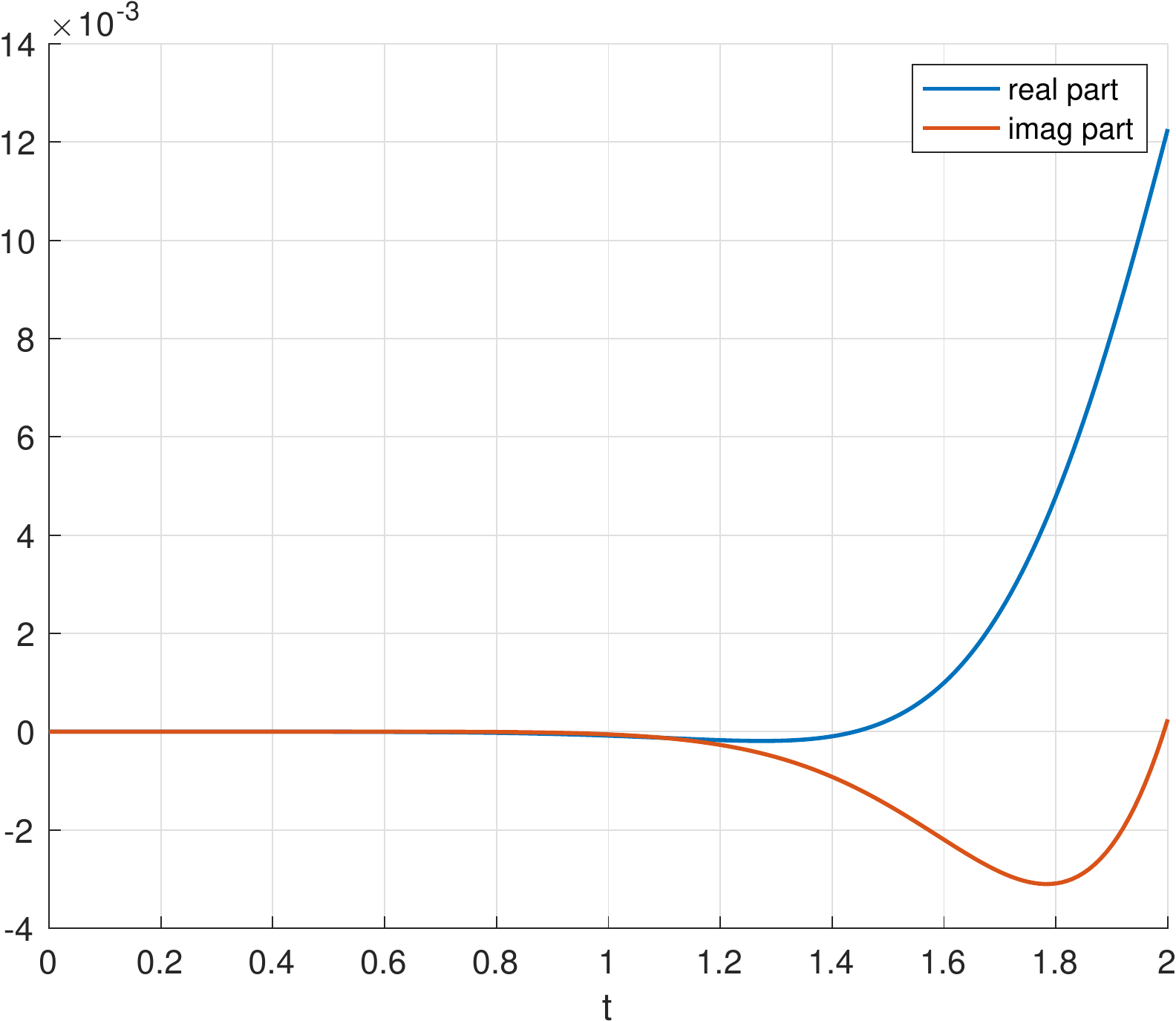}}\quad
	\subfigure[$(5, 6, 3, 2)$]{\includegraphics[scale=0.25]{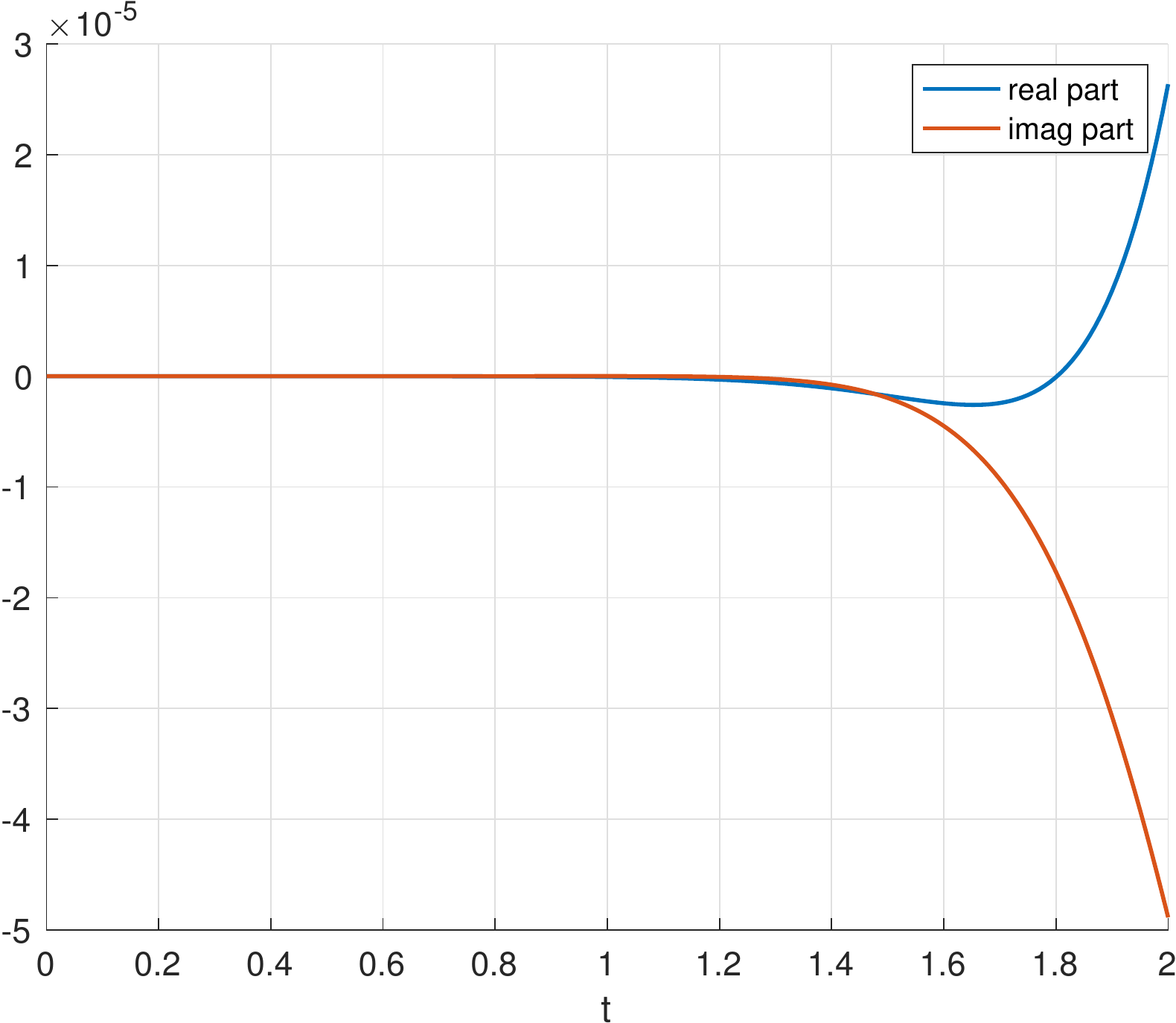}}\quad
	\subfigure[$(6, 8, 18, 11)$]{\includegraphics[scale=0.25]{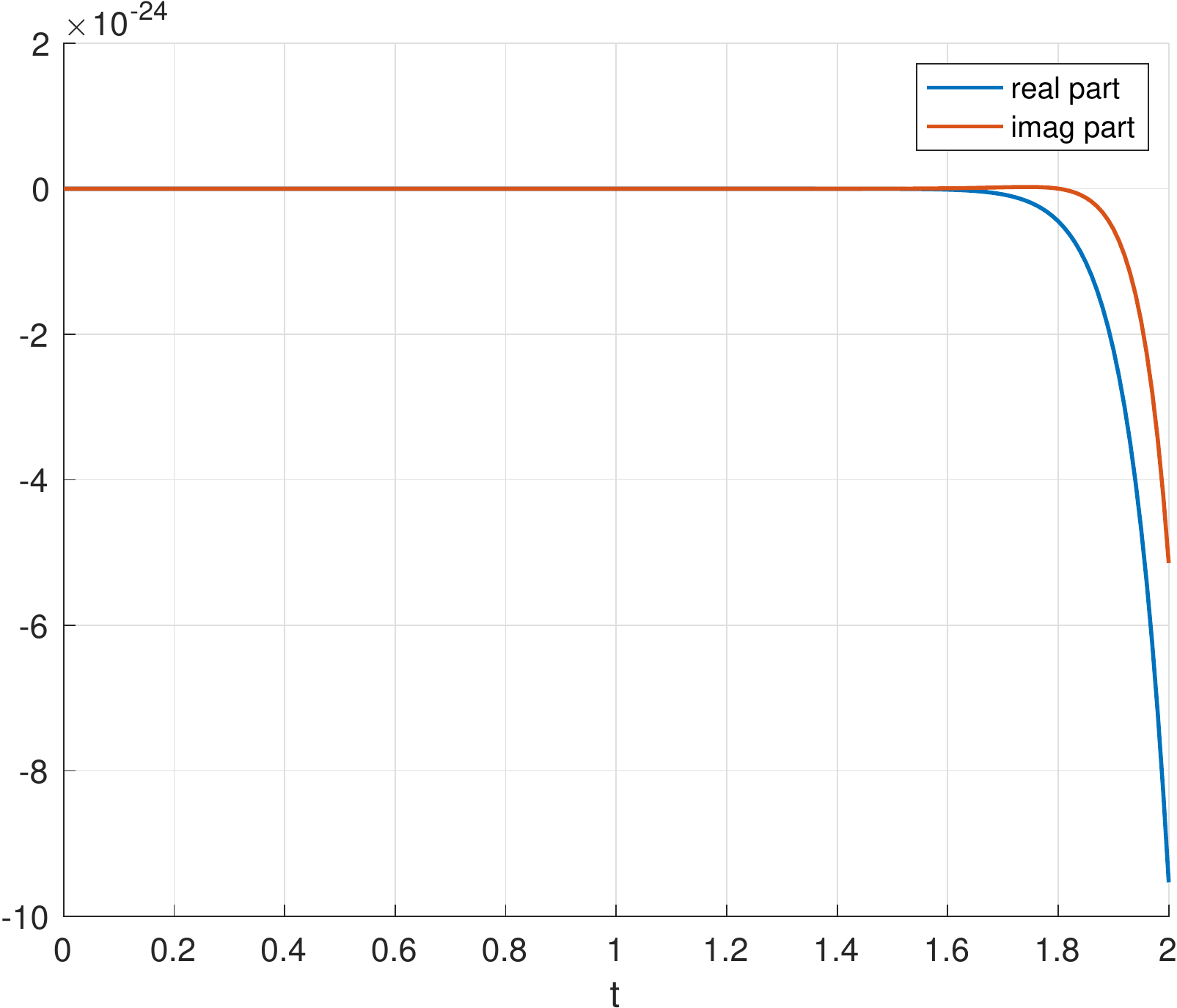}}
	\caption{Plots of integrand along the integration contour $\Gamma_{1}$ with
		$\ell=\ell^{\prime}=1$ and different $(n,n^{\prime},m,m^{\prime})$.}%
	\label{integrand1}%
\end{figure}
\begin{figure}[ht!]
	\centering
	\subfigure[$(1, 3, 3, 2)$]{\includegraphics[scale=0.25]{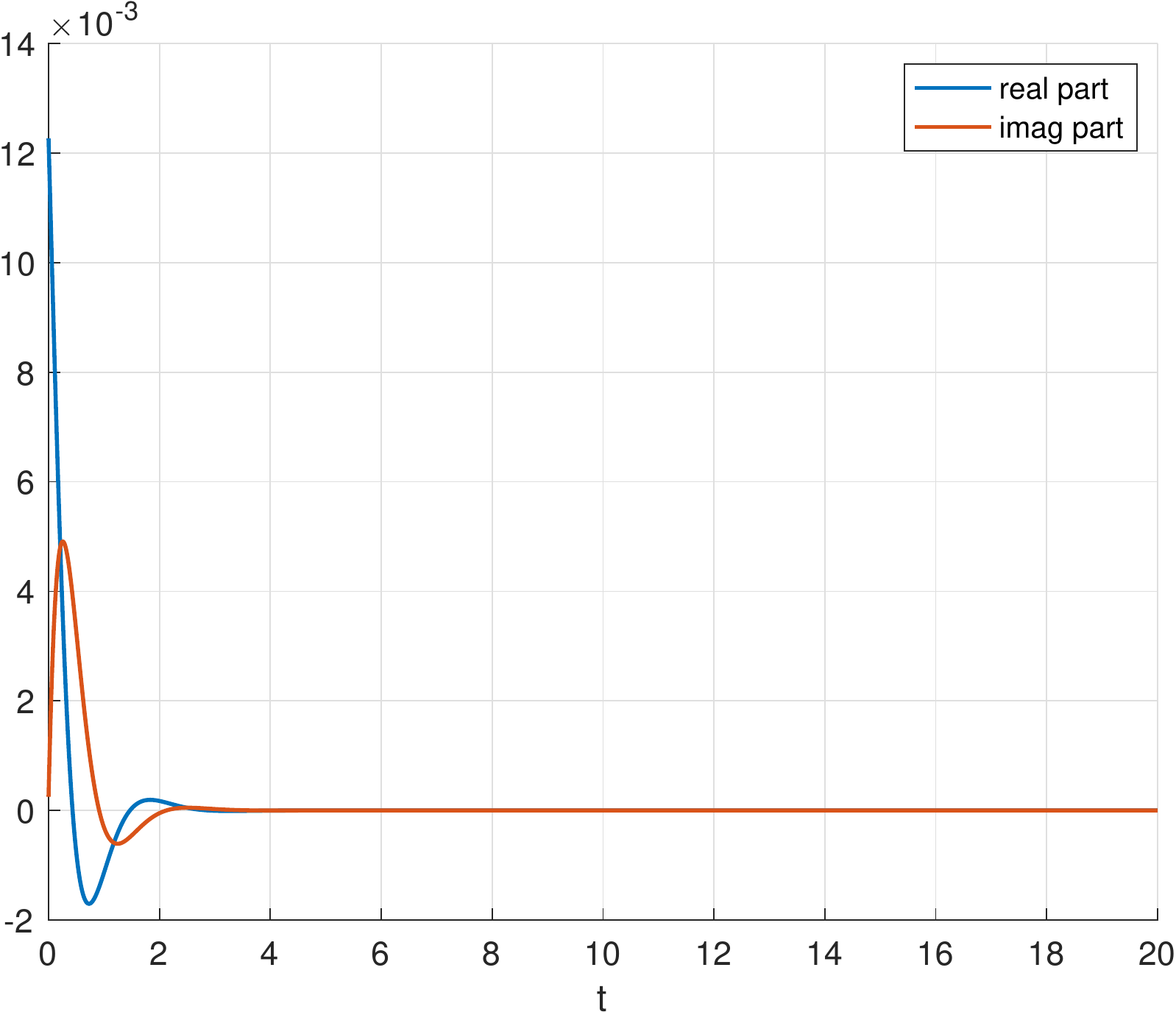}}\quad
	\subfigure[$(5, 6, 3, 2)$]{\includegraphics[scale=0.25]{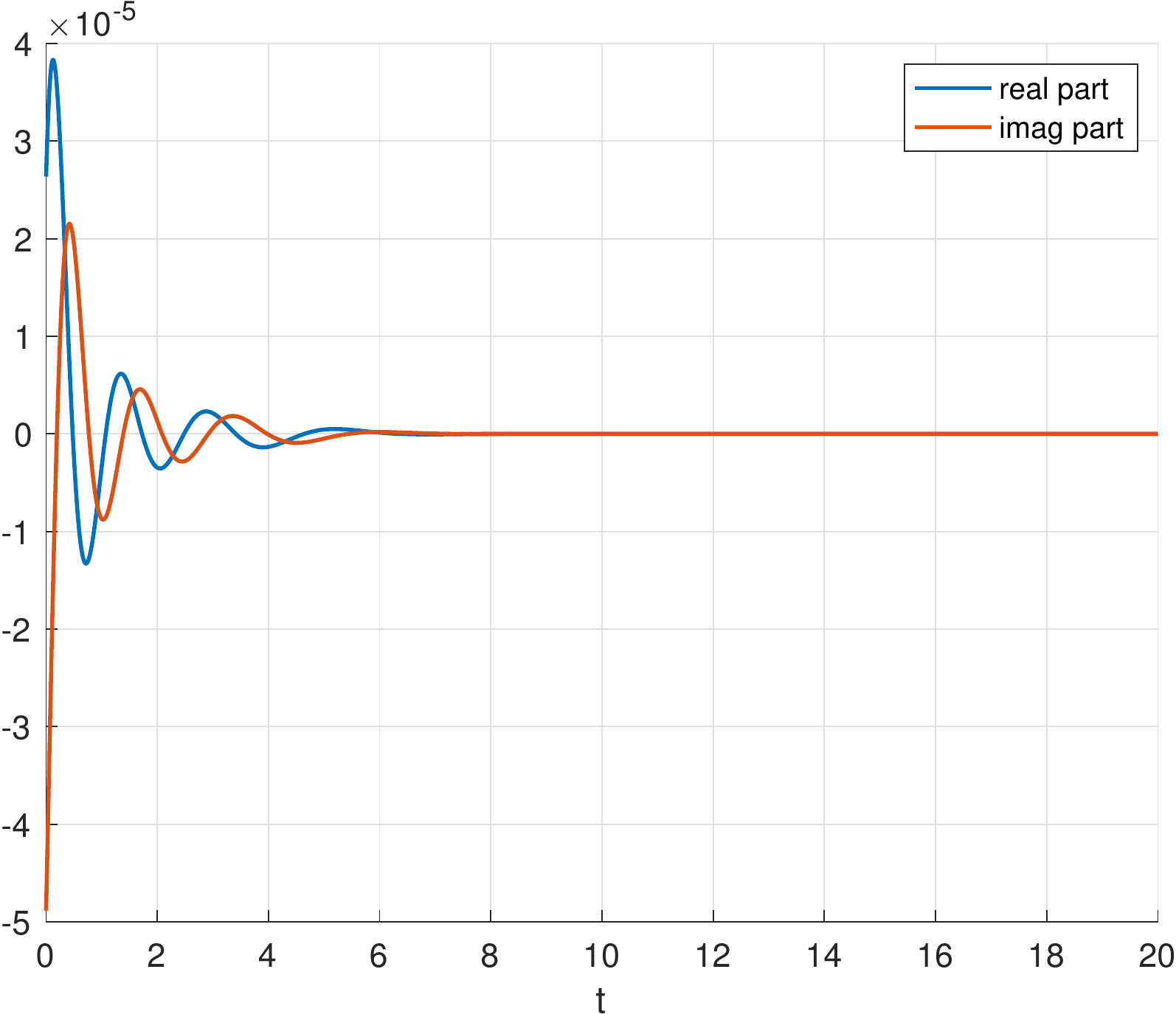}}\quad
	\subfigure[$(6, 8, 18, 11)$]{\includegraphics[scale=0.25]{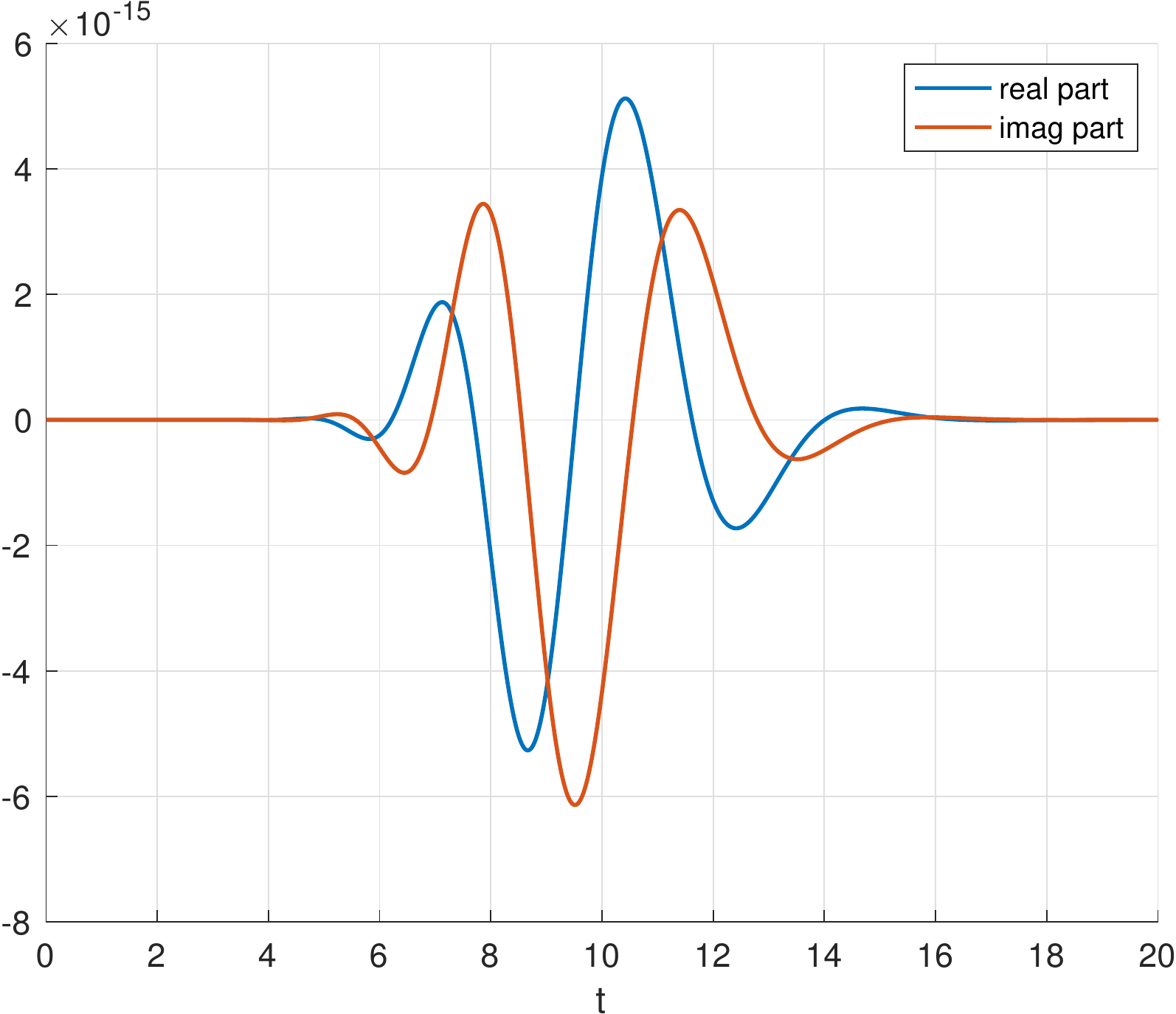}}
	\caption{Plots of integrand along the integration contour $\Gamma_{2}$ with
		$\ell=\ell^{\prime}=1$ and different $(n,n^{\prime},m,m^{\prime})$.}%
	\label{integrand2}%
\end{figure}
As an example, we plot the integrand in
\eqref{symmetricdetable} along $\Gamma_{1}$ and $\Gamma_{2}$ (see, Fig.
\ref{integrand1} and Fig. \ref{integrand2}). The three layers case with
$k_{0}=0.8,k_{1}=1.5,k_{2}=2.0,d=2.0,z=-0.3,z^{\prime}=-0.5$ and density given
in \eqref{densitythreelayer2} is used. We can see that the integrand has
exponential decay along $\Gamma_{2}$ as $t$ goes to infinity.

\begin{rem}
Similar to the first TE-FMM, the translation operators for center shift from
source boxes to their parents and from target boxes to their children are
exactly the same as in free space case which are given by
\eqref{me2mesymmetricte} and \eqref{le2lesymmetricte}.
\end{rem}

The algorithm using symmetric derivatives for general component
\eqref{generalsum} is as following:
\begin{algorithm}\label{algorithm3}
	\caption{TEFMM-II for general component \eqref{generalsum}}
	\begin{algorithmic}
		\State Generate an adaptive hierarchical tree structure and precompute tables.
		\State{\bf Upward pass:}
		\For{$l=H \to 0$}
		\For{all boxes $j$ on source tree level $l$ }
		\If{$j$ is a leaf node}
		\State{form the free-space TE using Eq. \eqref{mecoefficients}.}
		\Else
		\State form the free-space TE by merging children's expansions using the free-space center shift translation operator \eqref{me2mesymmetricte}.
		\EndIf
		\EndFor
		\EndFor
		\State{\bf Downward pass:}
		\For{$l=1 \to H$}
		\For{all boxes $j$ on target tree level $l$ }
		\State shift the TE of $j$'s parent to $j$ itself using the free-space center shift translation operator \eqref{le2lesymmetricte}.
		\State collect interaction list contribution using the source box to target box translation operator in Eq. \eqref{layerm2ltrans} with  precomputed tables of integrals \eqref{symmetricdetable}.
		\EndFor
		\EndFor
		\State {\bf Evaluate Local Expansions:}
		\For{each leaf node (childless box)}
		\State evaluate the local expansion at each particle location using \eqref{taylorexplocxclfree2}.
		\EndFor
		\State {\bf Local Direct Interactions:}
		\For{$i=1 \to N$ }
		\State compute Eq. \eqref{generalsum} of target particle $i$ in the neighboring boxes using precomputed table of $u_{\ell\ell'}^{\uparrow}(\bs r, \bs r')$.
		\EndFor
	\end{algorithmic}
\end{algorithm}

\section{Numerical results}

In this section, we present numerical results to demonstrate the performance
of two versions of TE-FMMs for acoustic wave scattering in layered media.
These algorithms are implemented based on an open-source FMM package DASHMM
\cite{debuhr2016dashmm}. The numerical simulations are performed on a
workstation with two Xeon E5-2699 v4 2.2 GHz processors (each has 22 cores)
and 500GB RAM using the gcc compiler version 6.3. Two and three layers media
are considered for the numerical tests. More specifically, interfaces are
placed at $z_{0}=0$ and $z_{0}=0,z_{1}=-2$ for two and three layer cases,
respectively. We first use an example with particles uniformly distributed
inside a cubic domain for accuracy and efficiency test. Then, more general
distributions of particles in irregular domains are tested.

\textbf{Example 1 (Cubic domains): }Set particles to be uniformly distributed
in cubes of size $1$ centered at $(0.5,0.5,1.0)$, $(0.5,0.5,-1.0)$ and
$(0.5,0.5,-3.0)$, respectively. Let $\widetilde{\Phi}_{\ell}%
(\boldsymbol{r}_{\ell i})$ be the approximated values of $\Phi_{\ell
}(\boldsymbol{r}_{\ell i})$ calculated by TE-FEM. For accuracy test, we put
$N=8000$ particles in each box and define $L^{2}$-error and maximum error as
\begin{equation}
Err_{2}^{\ell}:=\sqrt{\frac{\sum\limits_{i=1}^{N_{\ell}}|\Phi_{\ell
}(\boldsymbol{r}_{\ell i})-\widetilde{\Phi}_{\ell}(\boldsymbol{r}_{\ell
i})|^{2}}{\sum\limits_{i=1}^{N_{\ell}}|\Phi_{\ell}(\boldsymbol{r}_{\ell
i})|^{2}}},\qquad Err_{max}^{\ell}:=\max\limits_{1\leq i\leq{N_{\ell}}}%
\frac{|\Phi_{\ell}(\boldsymbol{r}_{\ell i})-\widetilde{\Phi}_{\ell
}(\boldsymbol{r}_{\ell i})|}{|\Phi_{\ell}(\boldsymbol{r}_{\ell i})|}.
\end{equation}
Convergence rates against $p$ are depicted in Figs. \ref{errorplot} and
\ref{errorplot2}. Comparisons between CPU time for the computation of free
space components $\Phi_{\ell}^{free}$ and scattering components in two and
three layers are presented in Tables \ref{Table:ex1two}-\ref{Table:ex1three}
and Tables \ref{Table:ex1two2}-\ref{Table:ex1three2}. We can see that the cost
for the computation of scattering components is about forty times of that for
the free space components. \begin{figure}[ptbh]
\center
\subfigure[two layers]{\includegraphics[scale=0.35]{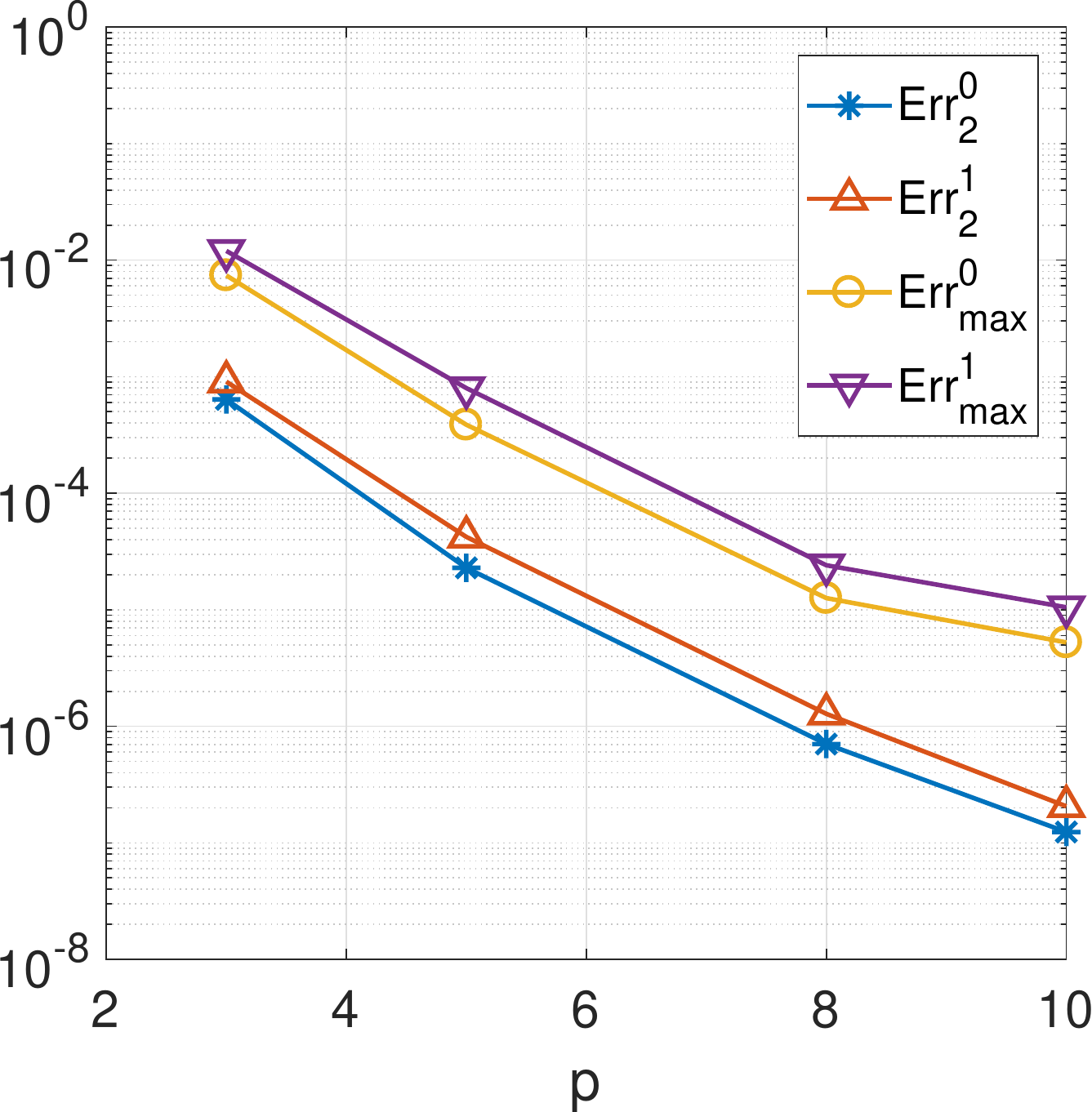}}\qquad
\qquad
\subfigure[three layers]{\includegraphics[scale=0.35]{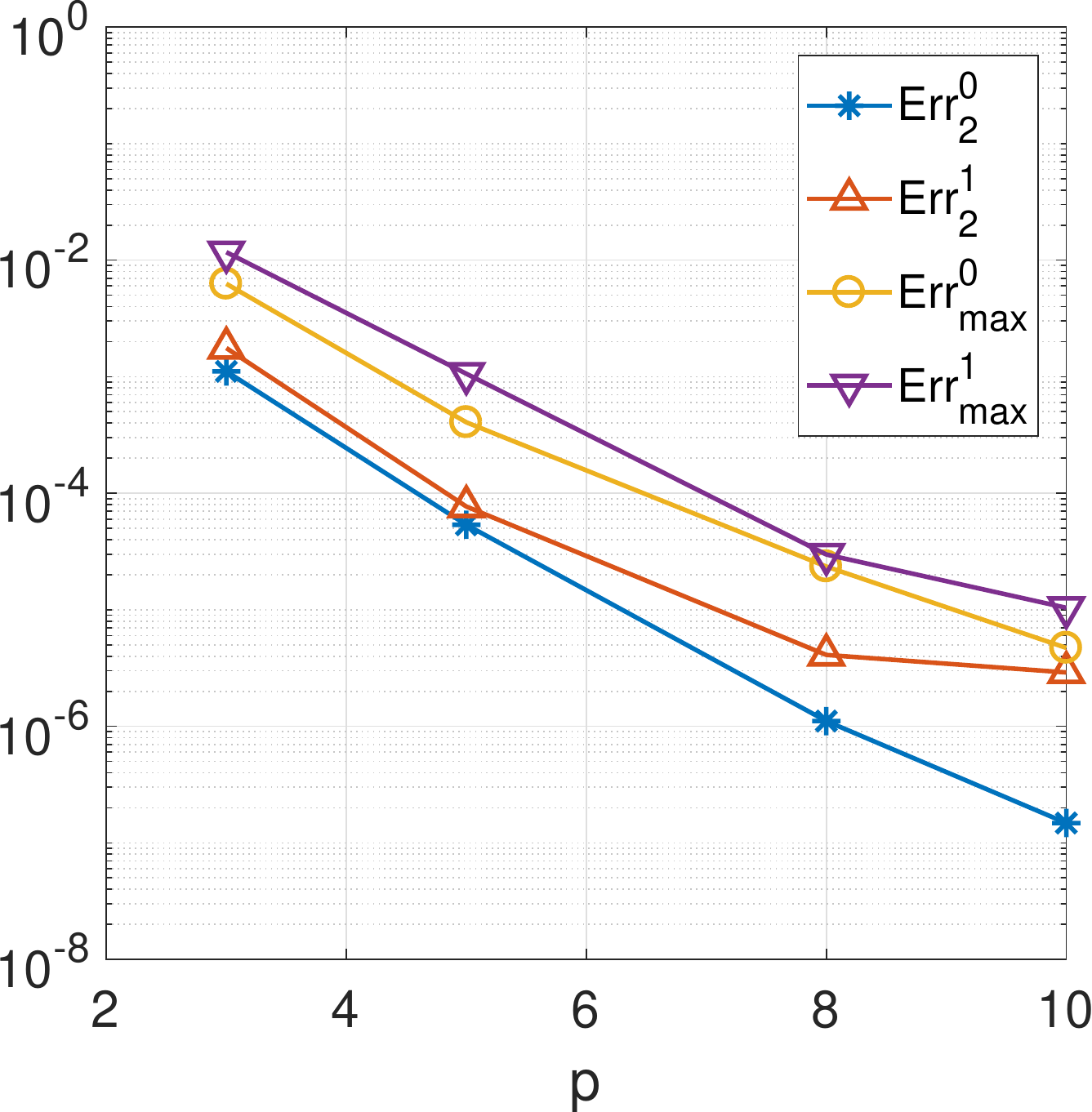}}
\caption{Convergence of TEFMM-I against truncation number $p$.}%
\label{errorplot}%
\end{figure}\begin{figure}[ptbhptbh]
\center
\subfigure[two layers]{\includegraphics[scale=0.35]{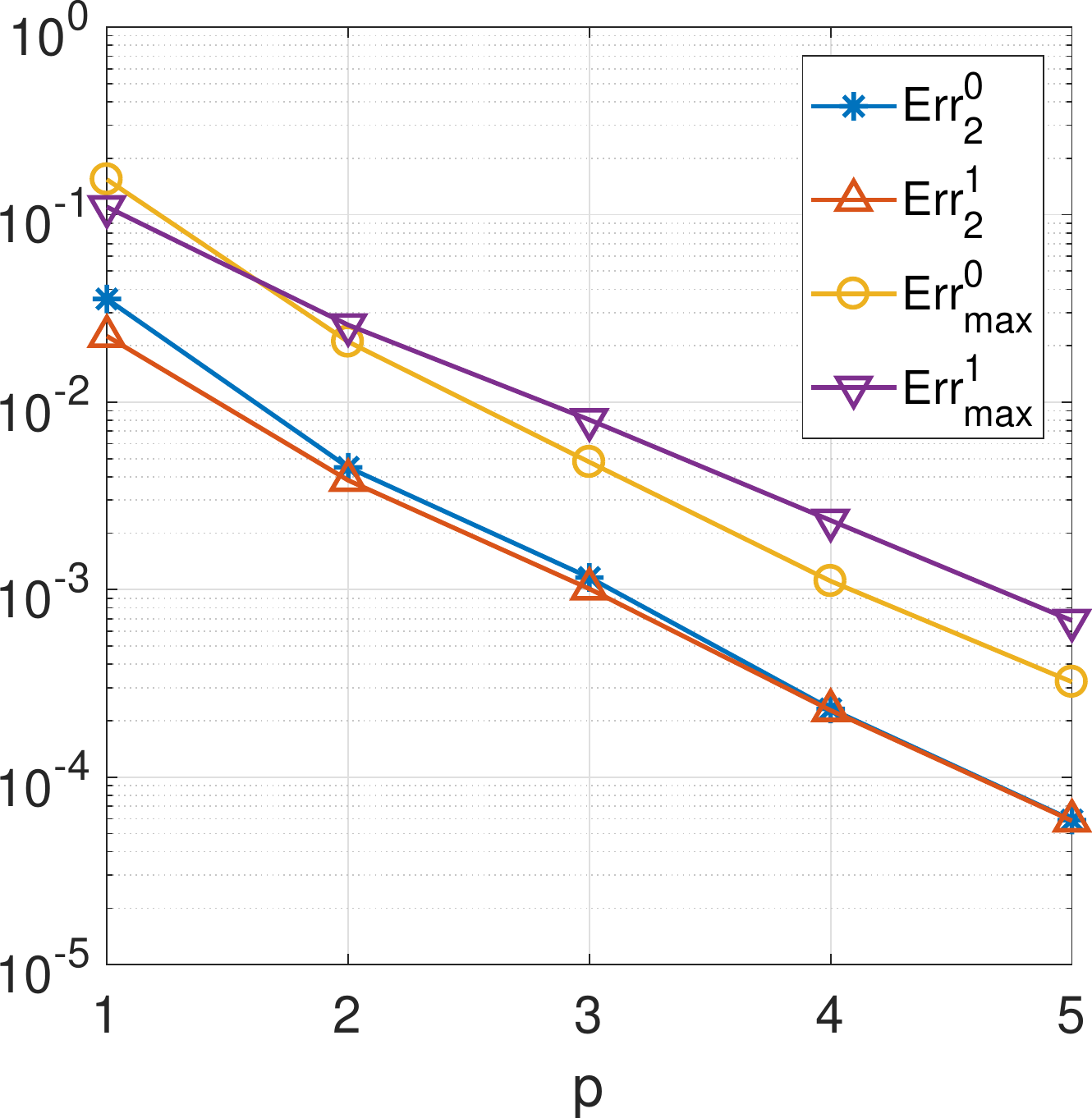}}\qquad
\qquad
\subfigure[three layers]{\includegraphics[scale=0.35]{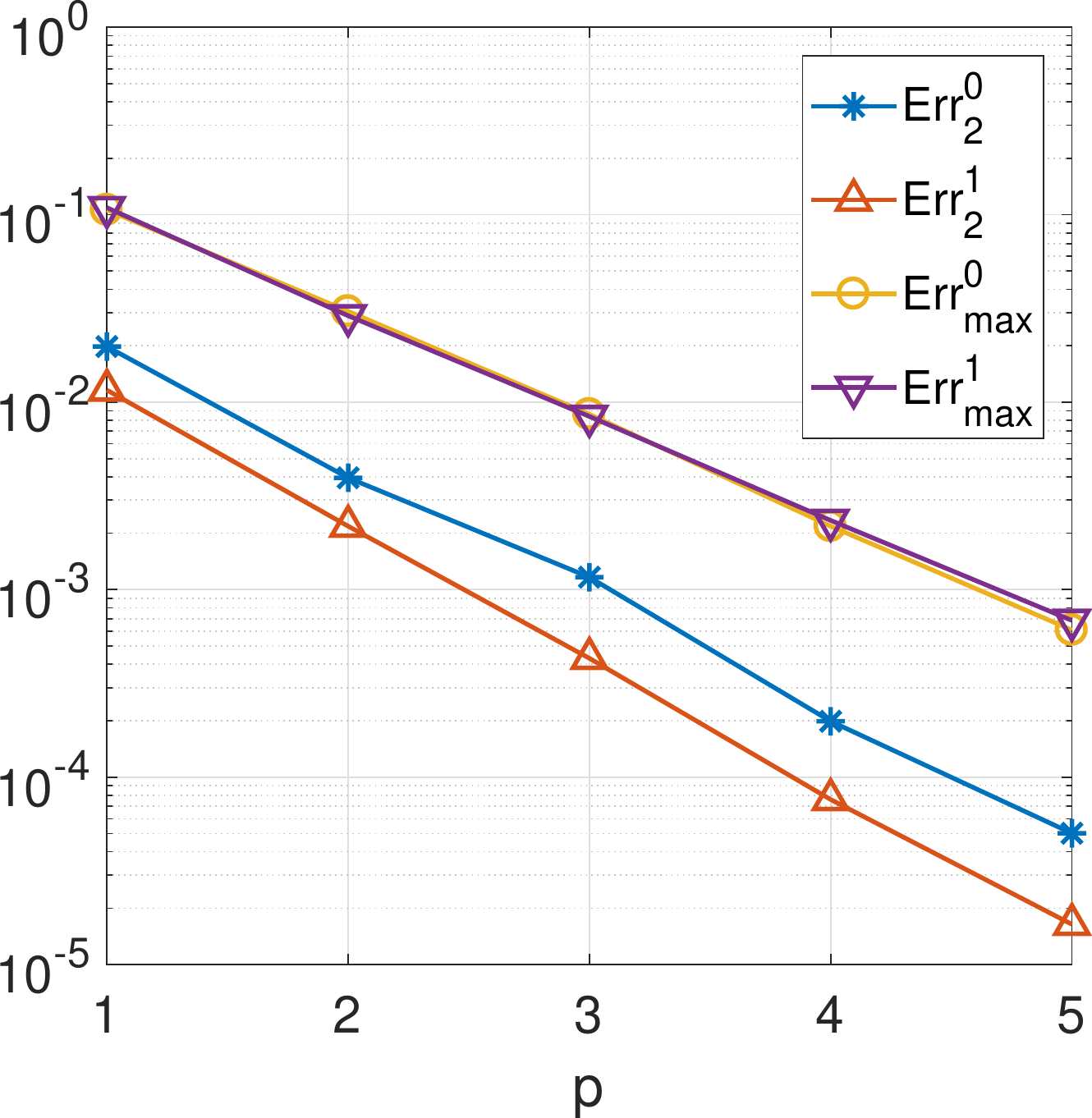}}
\caption{Convergence of TEFMM-II against truncation number $p$.}%
\label{errorplot2}%
\end{figure}\begin{table}[ptbhptbhptbh]
\centering {\small
\begin{tabular}
[c]{|c|c|c|c|c|c|}\hline
cores & $N$ & time for $\Phi_{0}^{free}$ & time for $\Phi_{00}^{\uparrow}%
+\Phi_{01}^{\uparrow}$ & time for $\Phi_{1}^{free}$ & time for $\Phi
_{10}^{\downarrow}+\Phi_{11}^{\downarrow}$\\\hline
\multirow{4}{*}{1} & 64000 & 4.11 & 120.11 & 4.00 & 135.57\\\cline{2-6}
& 216000 & 25.65 & 902.09 & 25.73 & 1005.43\\\cline{2-6}
& 512000 & 36.80 & 1120.73 & 36.58 & 1385.88\\\cline{2-6}
& 1000000 & 61.84 & 1422.70 & 63.07 & 1539.35\\\hline
\multirow{4}{*}{22} & 64000 & 0.25 & 7.08 & 0.23 & 7.85\\\cline{2-6}
& 216000 & 1.54 & 52.79 & 1.55 & 59.04\\\cline{2-6}
& 512000 & 2.21 & 65.88 & 2.18 & 73.65\\\cline{2-6}
& 1000000 & 3.75 & 81.39 & 3.65 & 90.25\\\hline
\multirow{4}{*}{44} & 64000 & 0.17 & 3.63 & 0.17 & 3.95\\\cline{2-6}
& 216000 & 1.10 & 26.60 & 1.09 & 29.86\\\cline{2-6}
& 512000 & 1.44 & 33.42 & 1.44 & 37.31\\\cline{2-6}
& 1000000 & 1.92 & 41.34 & 1.88 & 45.84\\\hline
\end{tabular}
}\caption{CPU time for two layers using TEFMM-I with $p=3$.}%
\label{Table:ex1two}%
\end{table}\begin{table}[ptbhptbhptbhptbh]
\centering {\small
\begin{tabular}
[c]{|c|c|c|c|c|c|}\hline
cores & $N$ & time for $\Phi_{0}^{free}$ & time for $\sum\limits_{\ell
^{\prime}=0}^{2}\Phi_{0\ell^{\prime}}^{\uparrow}$ & time for $\Phi_{1}^{free}$
& time for $\sum\limits_{\ell^{\prime}=0}^{2}\Phi_{1\ell^{\prime}}^{\uparrow}%
$\\\hline
\multirow{4}{*}{1} & 64000 & 3.94 & 109.61 & 3.97 & 134.37\\\cline{2-6}
& 216000 & 25.60 & 824.29 & 25.76 & 1016.34\\\cline{2-6}
& 512000 & 36.74 & 1034.80 & 36.51 & 1266.90\\\cline{2-6}
& 1000000 & 62.67 & 1286.97 & 60.39 & 1551.25\\\hline
\multirow{4}{*}{22} & 64000 & 0.25 & 6.47 & 0.24 & 7.90\\\cline{2-6}
& 216000 & 1.55 & 48.26 & 1.55 & 59.32\\\cline{2-6}
& 512000 & 2.21 & 60.36 & 2.19 & 73.81\\\cline{2-6}
& 1000000 & 3.75 & 75.05 & 3.65 & 90.67\\\hline
\multirow{4}{*}{44} & 64000 & 0.16 & 3.34 & 0.17 & 3.99\\\cline{2-6}
& 216000 & 1.09 & 24.38 & 1.09 & 30.06\\\cline{2-6}
& 512000 & 1.60 & 30.73 & 1.62 & 37.75\\\cline{2-6}
& 1000000 & 1.93 & 38.19 & 1.87 & 46.08\\\hline
\end{tabular}
}\caption{CPU time for three layers using TEFMM-I with $p=3$.}%
\label{Table:ex1three}%
\end{table}\begin{table}[ptbhptbhptbhptbhptbh]
\centering {\small
\begin{tabular}
[c]{|c|c|c|c|c|c|}\hline
cores & $N$ & time for $\Phi_{0}^{free}$ & time for $\Phi_{00}^{\uparrow}%
+\Phi_{01}^{\uparrow}$ & time for $\Phi_{1}^{free}$ & time for $\Phi
_{10}^{\downarrow}+\Phi_{11}^{\downarrow}$\\\hline
\multirow{4}{*}{1} & 64000 & 7.65 & 120.27 & 7.37 & 117.20\\\cline{2-6}
& 216000 & 60.76 & 610.60 & 59.30 & 663.51\\\cline{2-6}
& 512000 & 60.80 & 1071.95 & 60.06 & 1049.38\\\cline{2-6}
& 1000000 & 120.18 & 1153.79 & 118.06 & 1146.19\\\hline
\multirow{4}{*}{22} & 64000 & 0.38 & 9.48 & 0.38 & 9.64\\\cline{2-6}
& 216000 & 3.41 & 54.53 & 3.43 & 55.94\\\cline{2-6}
& 512000 & 3.47 & 74.21 & 3.43 & 74.84\\\cline{2-6}
& 1000000 & 6.49 & 79.62 & 6.37 & 82.53\\\hline
\multirow{4}{*}{44} & 64000 & 0.23 & 8.35 & 0.21 & 8.72\\\cline{2-6}
& 216000 & 1.74 & 47.62 & 1.76 & 47.09\\\cline{2-6}
& 512000 & 1.75 & 66.88 & 1.73 & 65.65\\\cline{2-6}
& 1000000 & 3.36 & 67.56 & 3.29 & 65.50\\\hline
\end{tabular}
}\caption{CPU time for two layers using TEFMM-II with $p=3$.}%
\label{Table:ex1two2}%
\end{table}\begin{table}[ptbhptbhptbhptbhptbhth]
\centering {\small
\begin{tabular}
[c]{|c|c|c|c|c|c|}\hline
cores & $N$ & time for $\Phi_{0}^{free}$ & time for $\sum\limits_{\ell
^{\prime}=0}^{2}\Phi_{0\ell^{\prime}}^{\uparrow}$ & time for $\Phi_{1}^{free}$
& time for $\sum\limits_{\ell^{\prime}=0}^{2}\Phi_{1\ell^{\prime}}^{\uparrow}%
$\\\hline
\multirow{4}{*}{1} & 64000 & 6.66 & 97.98 & 6.58 & 99.07\\\cline{2-6}
& 216000 & 60.42 & 651.24 & 61.14 & 657.92\\\cline{2-6}
& 512000 & 64.91 & 912.75 & 63.15 & 903.19\\\cline{2-6}
& 1000000 & 117.90 & 1101.38 & 116.92 & 1304.45\\\hline
\multirow{4}{*}{22} & 64000 & 0.38 & 9.64 & 0.38 & 9.69\\\cline{2-6}
& 216000 & 3.41 & 54.79 & 3.42 & 56.29\\\cline{2-6}
& 512000 & 3.47 & 75.98 & 3.41 & 75.91\\\cline{2-6}
& 1000000 & 6.49 & 81.98 & 6.37 & 84.57\\\hline
\multirow{4}{*}{44} & 64000 & 0.22 & 8.06 & 0.22 & 8.52\\\cline{2-6}
& 216000 & 1.77 & 47.69 & 1.80 & 46.63\\\cline{2-6}
& 512000 & 1.74 & 66.45 & 1.74 & 65.43\\\cline{2-6}
& 1000000 & 3.35 & 66.85 & 3.26 & 70.04\\\hline
\end{tabular}
}\caption{CPU time for three layers using TEFMM-II with $p=3$.}%
\label{Table:ex1three2}%
\end{table}

\textbf{Example 2 (Irregular domains): }In practical applications, objects of
irregular shape are often encountered. Here, we give examples with particles
located in irregular domains which are obtained by shifting the domain given
by
\begin{equation}
r=0.5-a+\frac{a}{8}(35\cos^{4}\theta-30\cos^{2}\theta+3),
\end{equation}
with $a=0.1,0.15$ to new centers $(0,0,1)$ and $(0,0,-1)$, respectively (see
Fig. \ref{expconfigure} (left) for an illustration). For the three layer case
test, we use particles in similar domains centered at $(0,0,1)$, $(0,0,-1)$
and $(0,0,-3)$ with $a=0.1,0.15,0.05$, respectively (see Fig.
\ref{expconfigure} (right)). All particles are generated by keeping the
uniform distributed particles in a larger cubic within corresponding irregular
domains. The CPU time for the computation of $\{\Phi_{\ell}%
(\boldsymbol{r}_{\ell i})\}_{i=0}^{N_{\ell}}$ and $\{\Phi_{\ell}%
^{free}(\boldsymbol{r}_{\ell i})\}_{i=0}^{N_{\ell}}$ are compared in Fig.
\ref{TEFMM1perform} and Fig. \ref{TEFMM2perform}. It shows that the new
algorithms have an $O(N)$ complexity.

\begin{figure}[ptbh]
\center  \includegraphics[scale=0.30]{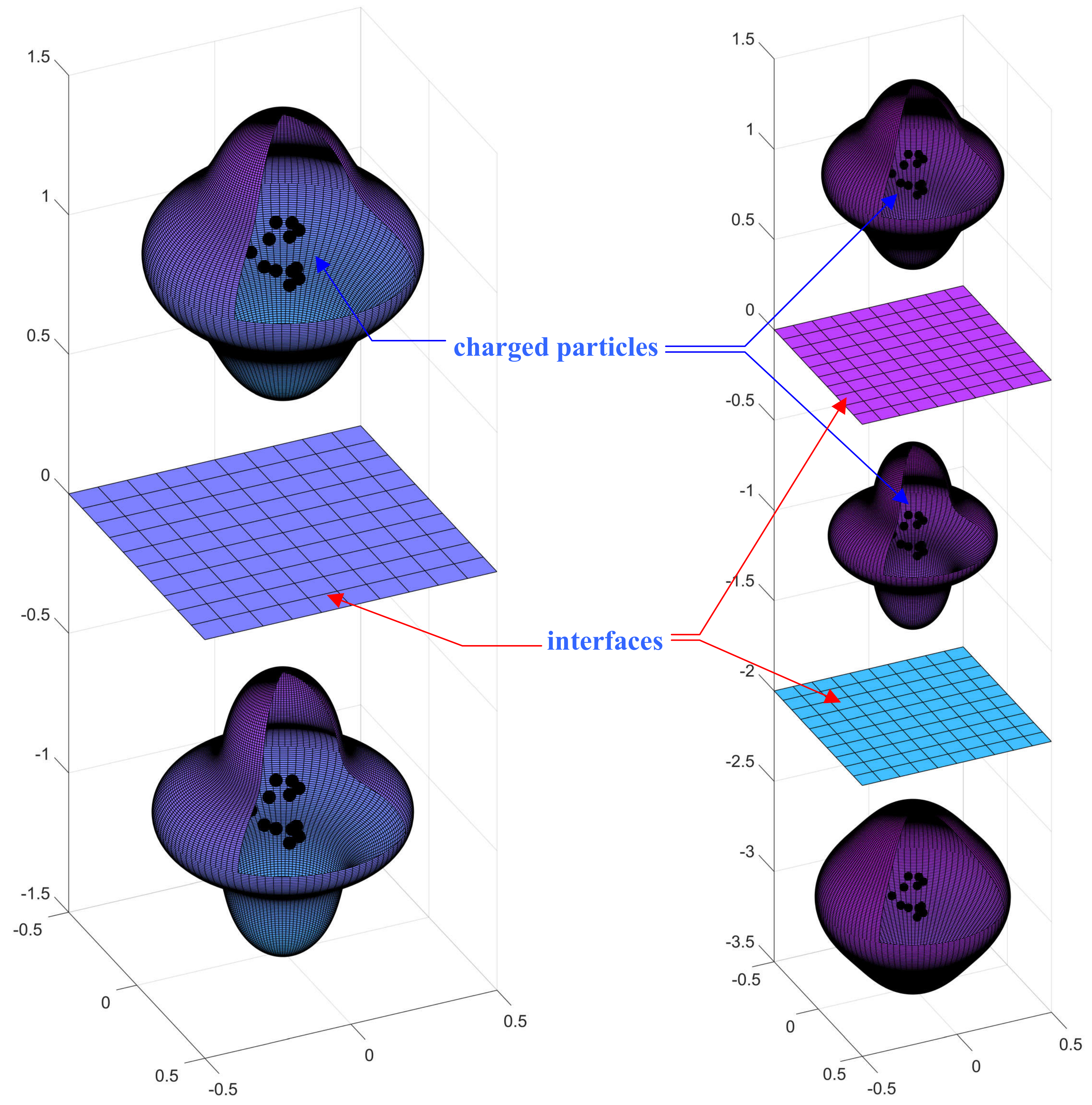}  \caption{Configuration
of two numerical examples.}%
\label{expconfigure}%
\end{figure}\begin{figure}[ptbhptbh]
\center
\subfigure[two layers]{\includegraphics[scale=0.35]{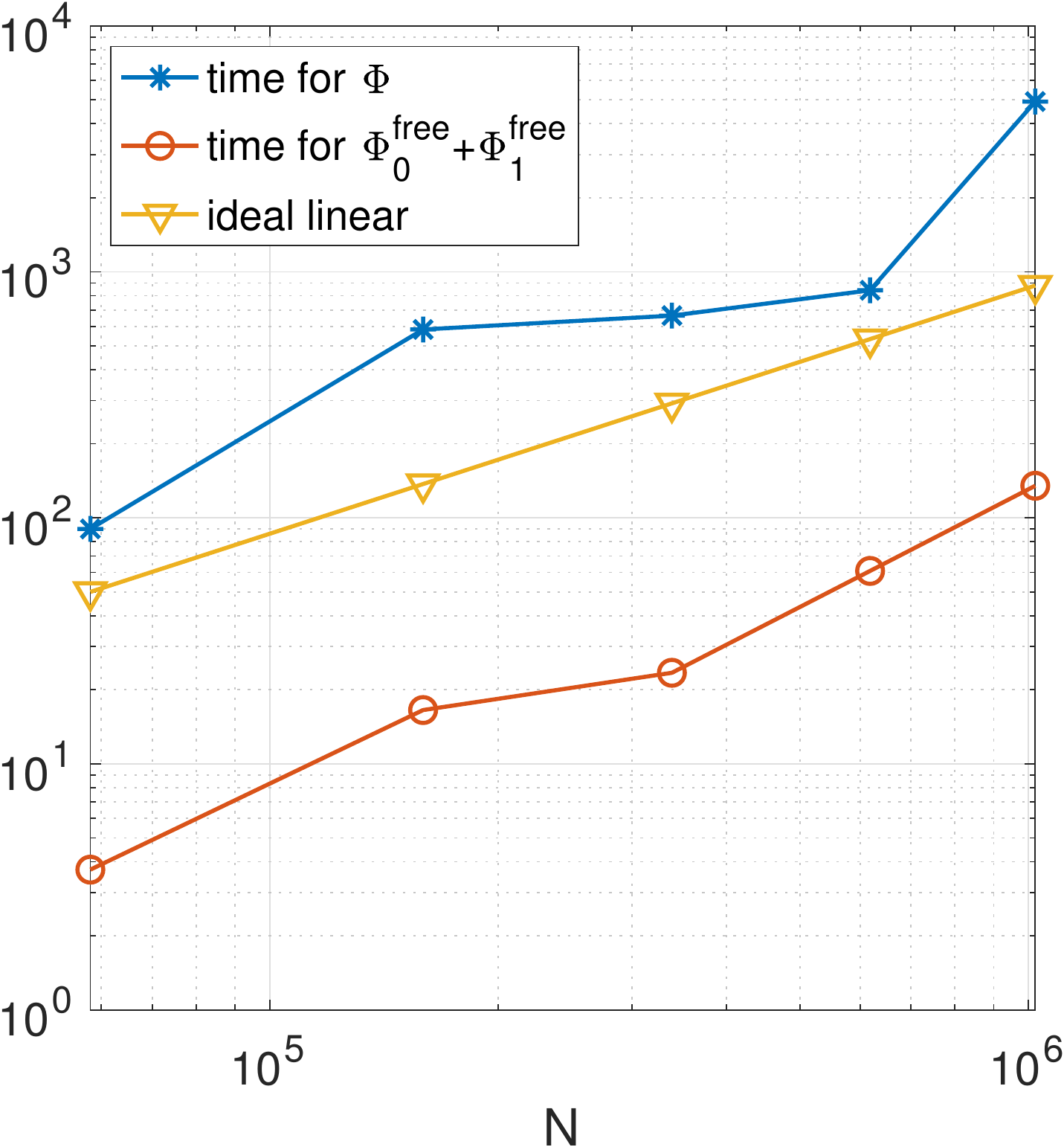}}\qquad
\qquad
\subfigure[three layers]{\includegraphics[scale=0.35]{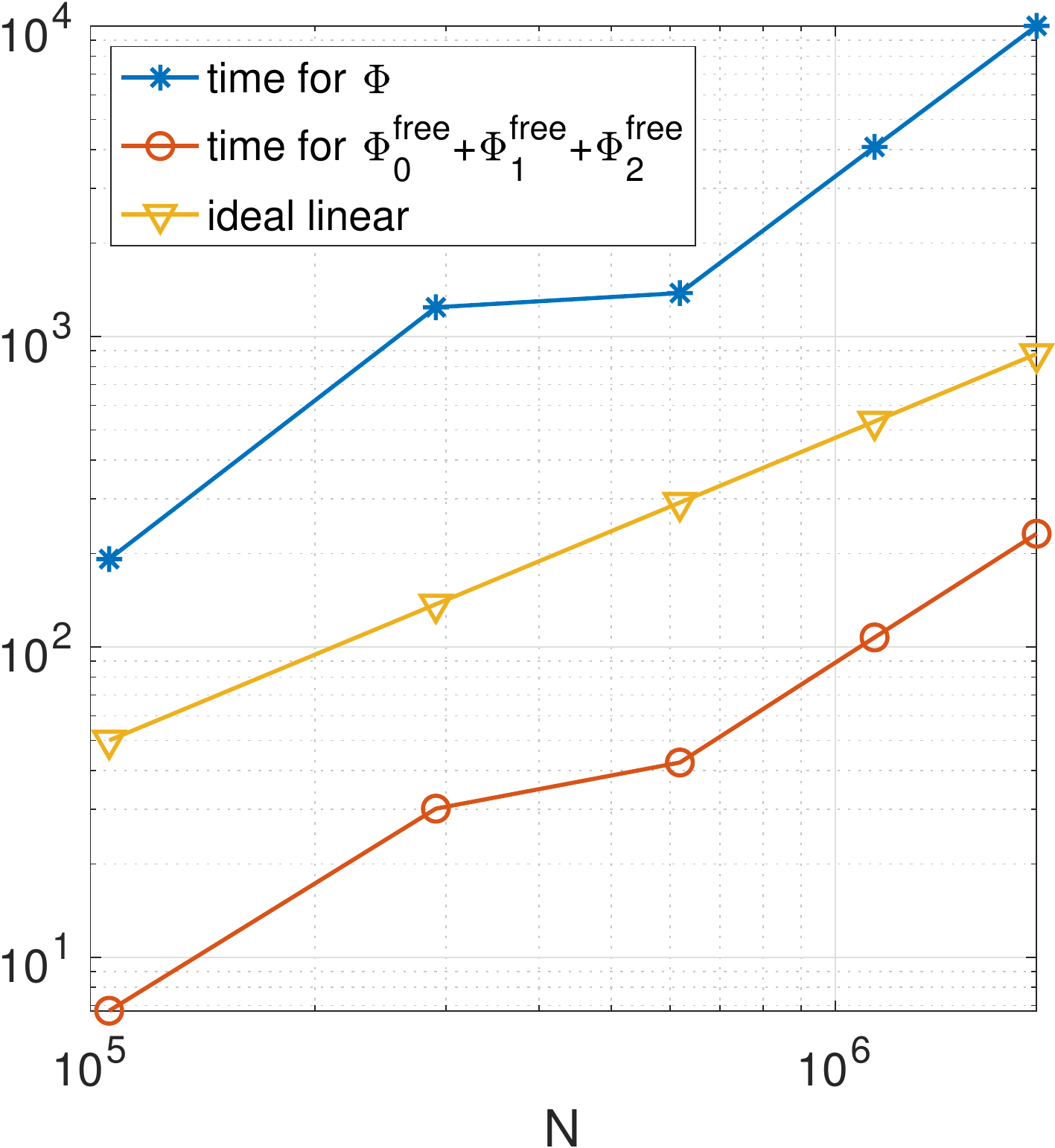}}
\caption{CPU time for TEFMM-I.}%
\label{TEFMM1perform}%
\end{figure}\begin{figure}[ptbhptbhptbh]
\center
\subfigure[two layers]{\includegraphics[scale=0.35]{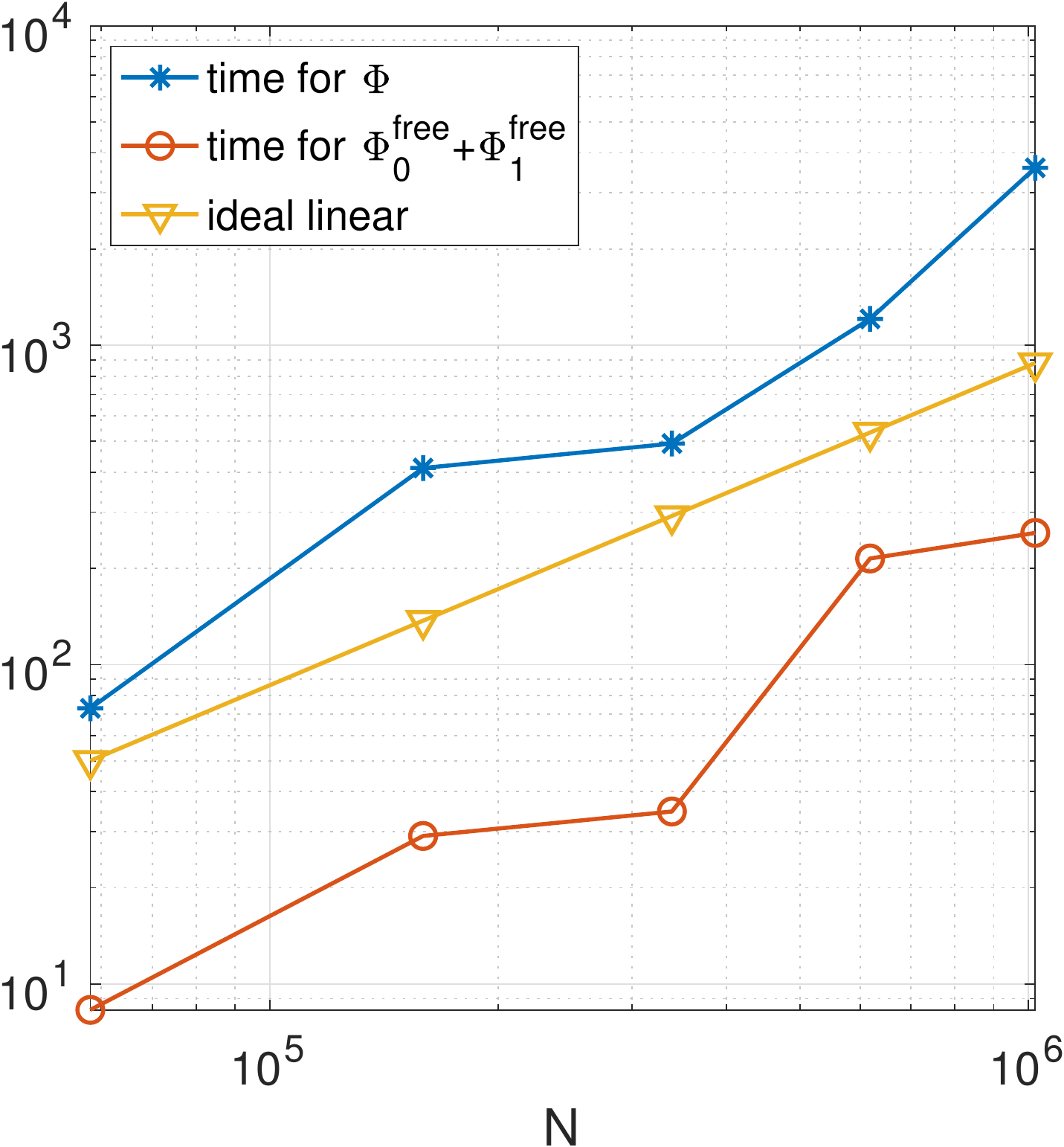}}\qquad
\qquad
\subfigure[three layers]{\includegraphics[scale=0.35]{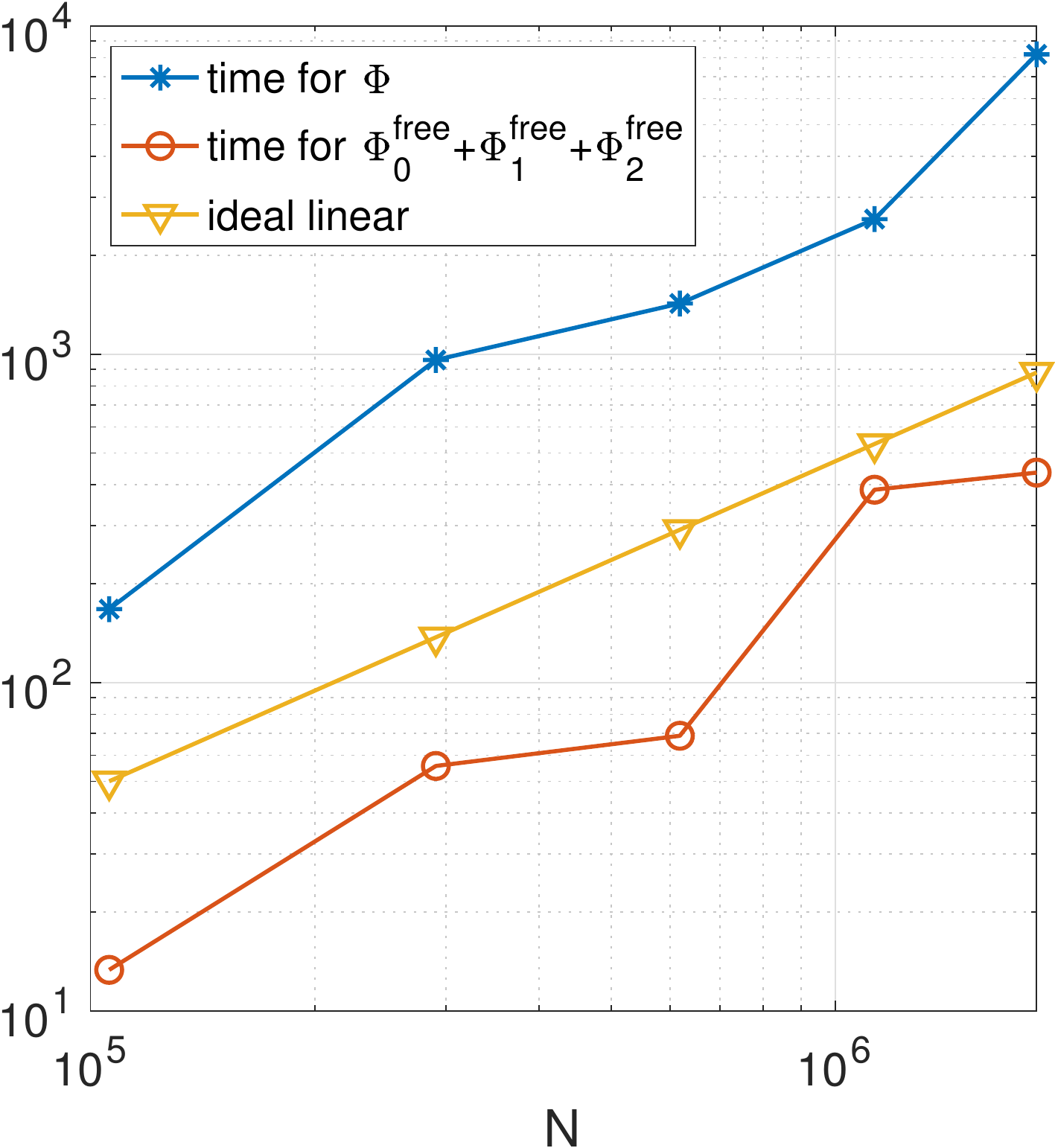}}
\caption{CPU time for TEFMM-II.}%
\label{TEFMM2perform}%
\end{figure}

\section{Conclusion}

In this paper, we have presented two Taylor-expansion based fast multipole
method for the efficient calculation of the discretized integral operator for
the Helmholtz equation in layered media. These methods use the Taylor
expansion of layered media Green's function for the low rank representation of
far field for acoustic wave scattering governed by the Helmholtz equation.
Comparing with the spherical harmonic multipole expansion in the traditional
FMM, the Taylor expansion requires $O(p^{3})$ terms for the low rank
representation of far field of layered Green's function in contrast to
$O(p^{2})$ for the spherical harmonic based multipole expansion FMM for the
free space Green's function. We addressed the main difficulty in developing
the TE-FMM for the layered media - the computation of up to $p$-th order
derivatives of the layered Green's function, which are given in terms of
oscillatory Sommerfeld integrals. We proposed two solutions to overcome this
difficulty. For the first TE-FMM\ based on non-symmetric derivatives, an
efficient algorithm was developed using discrete complex image, which has
shown to be very accurate and efficient for the low frequency Helmholtz
equation. Meanwhile, for the second TE-FMM\ based on symmetric derivatives,
pre-calculated tables are used for the translation operators.

Both versions of the TE-FMM have comparable accuracy, and  as our numerical examples show,  both have an $O(N)$ time complexity
similar to the FMM in the free space and they can provide fast solutions for
integral equations of Helmholtz equations in layered media with low to middle frequencies.
In comparison, the advantage of the first TE-FMM is the efficiency
of computing the translation operator with the complex image approximations. However, the complex image approximation is sensitive to the parameters and we still need a rigorous mathematical theory for the approach in finding the discrete images. On the other hand, the second TE-FMM could be used for higher order TE expansions, however, there is a need to pre-compute
large number of tables for the translation operators and it requires table storages.

For the future work, we will carry out error estimate of the TE-FMMs for the
layered media, which require an analysis of the Sommerfeld integral
representations of the derivatives and the complex image approximations.

\section{Appendix A}

\subsection{Two layers with sources in the top layer}

Let $L=1$ with source in the bottom layer at $\boldsymbol{r}^{\prime
}=(x^{\prime}, y^{\prime}, z^{\prime})$, i.e., $z^{\prime}>0$. Then, the
domain Green's function has representation
\begin{equation}
\label{greenspectraltwolayer3}%
\begin{cases}
\widehat u_{0}(k_{x},k_{y}, z)=A_{0}\cosh(\ri k_{0z}z)+B_{0}\sinh
(\ri k_{0z}z)+\frac{\ri e^{\ri(k_{0z}|z-z^{\prime}|-k_{x}x^{\prime}%
-k_{y}y^{\prime})}}{2k_{0z}}, & z>0,\\[7pt]%
\widehat u_{1}(k_{x},k_{y}, z)=A_{1}\cosh(\ri k_{1z}z)+B_{1}\sinh
(\ri k_{1z}z), & z<0,
\end{cases}
\end{equation}
or equivalently
\begin{equation}
\label{greenspectraltwolayer4}%
\begin{cases}
\widehat u_{0}(k_{x},k_{y}, z)=b_{0}e^{\ri k_{0z}z}+\frac{\ri e^{\ri(k_{0z}%
|z-z^{\prime}|-k_{x}x^{\prime}-k_{y}y^{\prime})}}{2k_{0z}}, & z>0,\\[7pt]%
\widehat u_{1}(k_{x},k_{y}, z)=a_{1}e^{-\ri k_{1z}z}, & z<0,
\end{cases}
\end{equation}
where
\begin{equation}%
\begin{split}
&  b_{0}=\frac{A_{0}+B_{0}}{2},\quad a_{1}=\frac{A_{1}-B_{1}}{2}.
\end{split}
\end{equation}
Proceeding the recursion \eqref{recursion} gives coefficients
\begin{equation}
\label{coefftwolayer3}%
\begin{cases}
\displaystyle A_{0}=B_{0}=\frac{(k_{0}k_{0z}-k_{1}k_{1z})e^{\ri k_{0z}%
z^{\prime}}}{2(k_{0}k_{0z}+k_{1}k_{1z})}\frac{\ri e^{-\ri(k_{x}x^{\prime
}+k_{y}y^{\prime})}}{k_{0z}},\\[10pt]%
\displaystyle A_{1}=-B_{1}=\frac{k_{0}k_{0z}e^{\ri k_{0z}z^{\prime}}}%
{k_{0}k_{0z}+k_{1}k_{1z}}\frac{\ri e^{-\ri(k_{x}x^{\prime}+k_{y}y^{\prime})}%
}{k_{0z}},
\end{cases}
\end{equation}
or alternatively
\begin{equation}
\label{coefftwolayer4}%
\begin{cases}
\displaystyle b_{0}=\frac{k_{0}k_{0z}-k_{1}k_{1z}}{2(k_{0}k_{0z}+k_{1}k_{1z}%
)}\frac{\ri e^{-\ri(k_{x}x^{\prime}+k_{y}y^{\prime})}e^{\ri k_{0z}z^{\prime}}%
}{ k_{0z}},\\[10pt]%
\displaystyle a_{1}=\frac{k_{0}k_{1z}}{k_{0}k_{0z}+k_{1}k_{1z}}\frac
{\ri e^{-\ri(k_{x}x^{\prime}+k_{y}y^{\prime})}e^{\ri k_{0z}z^{\prime}}}{
k_{1z}}.
\end{cases}
\end{equation}
Taking inverse Fourier transform on \eqref{greenspectraltwolayer4} with
coefficients given by \eqref{coefftwolayer4}, we have
\begin{equation}
\label{densitytwolayer1}\sigma_{00}^{\uparrow\uparrow}(k_{\rho})=\frac
{k_{0}k_{0z}-k_{1}k_{1z}}{k_{0}k_{0z}+k_{1}k_{1z}},\quad\sigma_{10}%
^{\downarrow\uparrow}(k_{\rho})=\frac{2k_{0}k_{1z}}{k_{0}k_{0z}+k_{1}k_{1z}}.
\end{equation}

\subsection{Two layers with sources in the bottom layer}

Let $L=1$ with source in the bottom layer at $\boldsymbol{r}^{\prime
}=(x^{\prime}, y^{\prime}, z^{\prime})$, i.e., $z^{\prime}<0$. The
domain Green's function has representation
\begin{equation}
\label{greenspectraltwolayer1}%
\begin{cases}
\widehat u_{0}(k_{x},k_{y}, z)=A_{0}\cosh(\ri k_{0z}z)+B_{0}\sinh
(\ri k_{0z}z), & z>0,\\[7pt]%
\widehat u_{1}(k_{x},k_{y}, z)=A_{1}\cosh(\ri k_{1z}z)+B_{1}\sinh
(\ri k_{1z}z)+\frac{\ri e^{\ri(k_{1z}|z-z^{\prime}|-k_{x}x^{\prime}%
-k_{y}y^{\prime})}}{2k_{1z}}, & z<0,
\end{cases}
\end{equation}
or equivalently
\begin{equation}
\label{greenspectraltwolayer2}%
\begin{cases}
\widehat u_{0}(k_{x},k_{y}, z)=b_{0}e^{\ri k_{0z}z}, & z>0,\\[7pt]%
\widehat u_{1}(k_{x},k_{y}, z)=a_{1}e^{-\ri k_{1z}z}+\frac{\ri e^{\ri(k_{1z}%
|z-z^{\prime}|-k_{x}x^{\prime}-k_{y}y^{\prime})}}{2k_{1z}}, & z<0,
\end{cases}
\end{equation}
where
\begin{equation}%
\begin{split}
&  b_{0}=\frac{A_{0}+B_{0}}{2},\quad a_{1}=\frac{A_{1}-B_{1}}{2}.
\end{split}
\end{equation}
Proceeding the recursion \eqref{recursion} gives coefficients
\begin{equation}
\label{coefftwolayer1}%
\begin{cases}
\displaystyle A_{0}=B_{0}=\frac{e^{-\ri k_{1z}z^{\prime}}k_{1}k_{1z}}%
{k_{0}k_{0z}+k_{1}k_{1z}}\frac{\ri e^{-\ri(k_{x}x^{\prime}+k_{y}y^{\prime})}}{
k_{1z}},\\[10pt]%
\displaystyle A_{1}=-B_{1}=\frac{k_{1}k_{1z}-k_{0}k_{0z}}{2(k_{0}k_{0z}%
+k_{1}k_{1z})}\frac{\ri e^{-\ri(k_{x}x^{\prime}+k_{y}y^{\prime})}%
e^{-\ri k_{1z}z^{\prime}}}{k_{1z}},
\end{cases}
\end{equation}
or alternatively
\begin{equation}
\label{coefftwolayer2}%
\begin{cases}
\displaystyle b_{0}=\frac{k_{1}k_{0z}}{k_{0}k_{0z}+k_{1}k_{1z}}\frac
{\ri e^{-\ri(k_{x}x^{\prime}+k_{y}y^{\prime})}e^{-\ri k_{1z}z^{\prime}}}{
k_{0z}},\\[10pt]%
\displaystyle a_{1}=\frac{k_{1}k_{1z}-k_{0}k_{0z}}{2(k_{0}k_{0z}+k_{1}k_{1z}%
)}\frac{\ri e^{-\ri(k_{x}x^{\prime}+k_{y}y^{\prime})}e^{-\ri k_{1z}z^{\prime}%
}}{ k_{1z}}.
\end{cases}
\end{equation}
Taking inverse Fourier transform on \eqref{greenspectraltwolayer2} with
coefficients given by \eqref{coefftwolayer2}, we have
\begin{equation}
\label{densitytwolayer2}\sigma_{01}^{\uparrow\downarrow}(k_{\rho}%
)=\frac{2k_{1}k_{0z}}{k_{0}k_{0z}+k_{1}k_{1z}},\quad\sigma_{11}^{\downarrow
\downarrow}(k_{\rho})=\frac{k_{1}k_{1z}-k_{0}k_{0z}}{k_{0}k_{0z}+k_{1}k_{1z}}.
\end{equation}

\subsection{Three layers with sources in the top layer}

Let $L=2$ with interfaces at $z=0$ and $z=-d<0$. Assume that the source is in
the first layer at $(x^{\prime}, y^{\prime}, z^{\prime})$, i.e., $z^{\prime
}>0$. The domain Green's function has representation
\begin{equation}
\label{greenspectral1}%
\begin{cases}
\displaystyle\widehat u_{0}(k_{x},k_{y}, z)=A_{0}\cosh(\ri k_{0z}z)+B_{0}%
\sinh(\ri k_{0z}z)+\frac{\ri e^{\ri(k_{0z}|z-z^{\prime}|-k_{x}x^{\prime}%
-k_{y}y^{\prime})}}{2k_{0z}}, & 0<z<z^{\prime},\\[7pt]%
\displaystyle\widehat u_{1}(k_{x},k_{y}, z)=A_{1}\cosh(\ri k_{1z}%
(z+d))+B_{1}\sinh(\ri k_{1z}(z+d)), & -d<z<0,\\[7pt]%
\displaystyle\widehat u_{2}(k_{x},k_{y}, z)=A_{2}\cosh(\ri k_{2z}z)+B_{2}%
\sinh(\ri k_{2z}z), & z<-d,
\end{cases}
\end{equation}
or equivalently
\begin{equation}
\label{greenspectral2}%
\begin{cases}
\displaystyle\widehat u_{0}(k_{x},k_{y}, z)=b_{0}e^{\ri k_{0z}z}%
+\frac{\ri e^{\ri(k_{0z}|z-z^{\prime}|-k_{x}x^{\prime}-k_{y}y^{\prime})}%
}{2k_{0z}}, & z>0,\\[7pt]%
\displaystyle\widehat u_{1}(k_{x},k_{y}, z)=a_{1}e^{-\ri k_{1z}(z+d)}%
+b_{1}e^{\ri k_{1z}(z+d)}, & -d<z<0,\\[7pt]%
\displaystyle\widehat u_{2}(k_{x},k_{y}, z)=a_{2}e^{-\ri k_{2z}z}, & z<-d,
\end{cases}
\end{equation}
where
\[
b_{0}=\frac{A_{0}+B_{0}}{2},\quad a_{1}=\frac{A_{1}-B_{1}}{2},\quad
b_{1}=\frac{A_{1}+B_{1}}{2},\quad a_{2}=\frac{A_{2}-B_{2}}{2}.
\]
Again proceeding the recursion \eqref{recursion} gives coefficients
\begin{equation}%
\begin{cases}
\displaystyle A_{0}=B_{0}=\frac{k_{0}k_{0z}\kappa_{11}+\ri k_{1}k_{1z}%
\kappa_{12}}{2(k_{0}k_{0z}\kappa_{11}-\ri k_{1}k_{1z}\kappa_{12})}%
\frac{\ri e^{-\ri(k_{x}x^{\prime}+k_{y}y^{\prime})} e^{\ri k_{0z} z^{\prime}}
}{ k_{0z}},\\[7pt]%
\displaystyle A_{1}=\frac{k_{0}k_{1}k_{0z} k_{1z} }{k_{0}k_{0z}\kappa
_{11}-\ri k_{1}k_{1z}\kappa_{12}}\frac{\ri e^{-\ri(k_{x}x^{\prime}%
+k_{y}y^{\prime})}e^{ik_{0z} z^{\prime}}}{ k_{0z}},\\[7pt]%
\displaystyle B_{1}=\frac{-k_{0}k_{2} k_{0z} k_{2z}}{k_{0}k_{0z}\kappa
_{11}-\ri k_{1}k_{1z}\kappa_{12}}\frac{\ri e^{-\ri(k_{x}x^{\prime}%
+k_{y}y^{\prime})}e^{ik_{0z} z^{\prime}}}{ k_{0z}},\\[7pt]%
\displaystyle A_{2}=-B_{2}=\frac{ k_{0} k_{1}k_{0z} k_{1z}e^{-\ri d k_{2z} }%
}{k_{0}k_{0z}\kappa_{11}-\ri k_{1}k_{1z}\kappa_{12}}\frac{\ri e^{-\ri(k_{x}%
x^{\prime}+k_{y}y^{\prime})}e^{\ri k_{0z} z^{\prime}}}{k_{0z}},
\end{cases}
\end{equation}
and
\[%
\begin{cases}
\displaystyle b_{0}=\frac{k_{0}k_{0z}\kappa_{11}+\ri k_{1}k_{1z}\kappa_{12}%
}{2(k_{0}k_{0z}\kappa_{11}-\ri k_{1}k_{1z}\kappa_{12})}\frac{\ri e^{-\ri(k_{x}%
x^{\prime}+k_{y}y^{\prime})} e^{\ri k_{0z} z^{\prime}}}{ k_{0z}},\\[7pt]%
\displaystyle a_{1}=\frac{k_{0}k_{1z}(k_{1} k_{1z}+ k_{2} k_{2z}) }%
{2(k_{0}k_{0z}\kappa_{11}-\ri k_{1}k_{1z}\kappa_{12})}\frac{\ri e^{-\ri(k_{x}%
x^{\prime}+k_{y}y^{\prime})}e^{\ri k_{0z} z^{\prime}}}{ k_{1z}},\\[7pt]%
\displaystyle b_{1}=\frac{k_{0}k_{1z}(k_{1} k_{1z}- k_{2} k_{2z}) }%
{2(k_{0}k_{0z}\kappa_{11}-\ri k_{1}k_{1z}\kappa_{12})}\frac{\ri e^{-\ri(k_{x}%
x^{\prime}+k_{y}y^{\prime})}e^{\ri k_{0z} z^{\prime}}}{ k_{1z}},\\[7pt]%
\displaystyle a_{2}=\frac{ k_{0} k_{1} k_{1z}k_{2z}e^{-\ri d k_{2z} }}%
{k_{0}k_{0z}\kappa_{11}-\ri k_{1}k_{1z}\kappa_{12}}\frac{\ri e^{-\ri(k_{x}%
x^{\prime}+k_{y}y^{\prime})}e^{\ri k_{0z} z^{\prime}}}{k_{2z}},
\end{cases}
\]
where
\begin{equation}
\label{kappa1112}%
\begin{split}
&  \kappa_{11}=\frac{k_{1}k_{1z}-k_{2}k_{2z}}{2}e^{\ri 2dk_{1z}}+\frac
{k_{1}k_{1z}+k_{2}k_{2z}}{2},\\
&  \kappa_{12}=\ri\Big(\frac{k_{2}k_{2z}-k_{1}k_{1z}}{2}e^{\ri 2dk_{1z}}%
+\frac{k_{1}k_{1z}+k_{2}k_{2z}}{2}\Big).
\end{split}
\end{equation}
Substuting into \eqref{greenspectral2} and applying inverse Fourier transform,
we have
\begin{equation}
\label{densitythreelayer1}%
\begin{cases}
\displaystyle\sigma_{00}^{\uparrow\uparrow}(k_{\rho})=\frac{k_{0}k_{0z}%
\kappa_{11}+\ri k_{1}k_{1z}\kappa_{12}}{k_{0}k_{0z}\kappa_{11}-\ri k_{1}%
k_{1z}\kappa_{12}},\\[8pt]%
\displaystyle\sigma_{10}^{\uparrow\uparrow}(k_{\rho})=\frac{k_{0}k_{1z}(k_{1}
k_{1z}- k_{2} k_{2z})e^{\ri dk_{1z}}}{k_{0}k_{0z}\kappa_{11}-\ri k_{1}%
k_{1z}\kappa_{12}},\\[8pt]%
\displaystyle\sigma_{10}^{\downarrow\uparrow}(k_{\rho})=\frac{k_{0}%
k_{1z}(k_{1} k_{1z}+ k_{2} k_{2z})e^{\ri dk_{1z}}}{k_{0}k_{0z}\kappa
_{11}-\ri k_{1}k_{1z}\kappa_{12}},\\[8pt]%
\displaystyle\sigma_{20}^{\downarrow\uparrow}(k_{\rho})=\frac{2k_{0} k_{1}
k_{1z}k_{2z}e^{\ri dk_{1z}} }{k_{0}k_{0z}\kappa_{11}-\ri k_{1}%
k_{1z}\kappa_{12}}.
\end{cases}
\end{equation}

\subsection{Three layers with sources in the middle layer}

Let $L=2$ with interfaces at $z=0$ and $z=-d<0$. Assume that the source is in
the middle layer at $(x^{\prime}, y^{\prime}, z^{\prime})$, i.e.,
$-d<z^{\prime}<0$. The domain Green's function has representation
\begin{equation}
\label{greenspectralmsour1}%
\begin{cases}
\displaystyle \widehat u_{0}=A_{0}\cosh(\ri k_{0z}z)+B_{0}\sinh(\ri k_{0z}%
z), & z>0,\\[7pt]%
\displaystyle \widehat u_{0}=A_{1}\cosh(\ri k_{1z}(z+d))+B_{1}\sinh
(\ri k_{1z}(z+d))+\frac{\ri e^{\ri(k_{1z}|z-z^{\prime}|-k_{x}x^{\prime}%
-k_{y}y^{\prime})}}{2k_{1z}}, & -d<z<z^{\prime},\\[7pt]%
\displaystyle \widehat u_{2}=A_{2}\cosh(\ri k_{2z}z)+B_{2}\sinh(\ri k_{2z}%
z), & z<-d,
\end{cases}
\end{equation}
or equivalently
\begin{equation}
\label{greenspectralmsour2}%
\begin{cases}
\displaystyle \widehat u_{0}=b_{0}e^{\ri k_{0z}z}, & z>0,\\[7pt]%
\displaystyle \widehat u_{1}=a_{1}e^{-\ri k_{1z}(z+d)}+b_{1}e^{\ri k_{1z}%
(z+d)}+\frac{\ri e^{\ri(k_{1z}|z-z^{\prime}|-k_{x}x^{\prime}-k_{y}y^{\prime}%
)}}{2k_{1z}}, & -d<z<0,\\[7pt]%
\displaystyle \widehat u_{2}=a_{2}e^{-\ri k_{2z}z}, & z<-d,
\end{cases}
\end{equation}
where
\[
b_{0}=\frac{A_{0}+B_{0}}{2},\quad a_{1}=\frac{A_{1}-B_{1}}{2},\quad
b_{1}=\frac{A_{1}+B_{1}}{2},\quad a_{2}=\frac{A_{2}-B_{2}}{2}.
\]
Again proceeding the recursion \eqref{recursion} gives coefficients
\begin{equation}%
\begin{cases}
\displaystyle A_{0}=B_{0}=\frac{k_{1}k_{1z}\kappa_{23}}{k_{0}k_{0z}\kappa
_{11}-\ri k_{1}k_{1z}\kappa_{12}}\frac{\ri e^{-\ri(k_{x}x^{\prime}%
+k_{y}y^{\prime})}}{ k_{1z}},\\[7pt]%
\displaystyle A_{1}=\frac{(e^{-\ri k_{1z}z^{\prime}}k_{1}k_{1z}(k_{1}%
k_{1z}-k_{0}k_{0z})+e^{\ri k_{1z}(d+z^{\prime})}\kappa_{21})e^{\ri dk_{1z}}%
}{k_{0}k_{0z}\kappa_{11}-\ri k_{1}k_{1z}\kappa_{12}}\frac{\ri e^{-\ri(k_{x}%
x^{\prime}+k_{y}y^{\prime})}}{k_{1z}},\\[7pt]%
\displaystyle B_{1}=\frac{(e^{-\ri k_{1z}z^{\prime}}k_{2}k_{2z}(k_{0}%
k_{0z}-k_{1}k_{1z})+e^{\ri k_{1z}(d+z^{\prime})}\kappa_{21}^{\prime
})e^{\ri dk_{1z}}}{k_{0}k_{0z}\kappa_{11}-\ri k_{1}k_{1z}\kappa_{12}}%
\frac{\ri e^{-\ri(k_{x}x^{\prime}+k_{y}y^{\prime})}}{k_{1z}},\\[7pt]%
\displaystyle A_{2}=-B_{2}=\frac{ k_{1} k_{1z}\kappa_{22}e^{\ri d(k_{1z}%
-k_{2z})}}{k_{0}k_{0z}\kappa_{11}-\ri k_{1}k_{1z}\kappa_{12}}\frac
{\ri e^{-\ri(k_{x}x^{\prime}+k_{y}y^{\prime})}}{k_{1z}},
\end{cases}
\end{equation}
where $\kappa_{11}, \kappa_{12}$ are defined in \eqref{kappa1112} and
\[%
\begin{split}
\kappa_{21}  &  =(k_{1}k_{1z}-k_{2}k_{2z})\Big(\frac{k_{1} k_{1z}-k_{0}
k_{0z}}{2}e^{\ri dk_{1z}}+\frac{k_{0} k_{0z}+k_{1} k_{1z}}{2}e^{-\ri d k_{1z}%
}\Big),\\
\kappa_{21}^{\prime}  &  =(k_{1}k_{1z}-k_{2}k_{2z})\Big(\frac{k_{0}
k_{0z}-k_{1} k_{1z}}{2}e^{\ri dk_{1z}}+\frac{k_{0} k_{0z}+k_{1} k_{1z}}%
{2}e^{-\ri d k_{1z}}\Big),\\
\kappa_{22}  &  =\frac{k_{1} k_{1z}+k_{0} k_{0z}}{2}e^{\ri k_{1z}z^{\prime}%
}+\frac{k_{1} k_{1z}-k_{0} k_{0z}}{2}e^{-\ri k_{1z}z^{\prime}},\\
\kappa_{23}  &  =\frac{k_{1}k_{1z}- k_{2}k_{2z}}{2}e^{\ri k_{1z}(2d+z^{\prime
})}+\frac{k_{1}k_{1z}+ k_{2}k_{2z}}{2}e^{-\ri k_{1z}z^{\prime}}.
\end{split}
\]
Noting that
\begin{equation}
\kappa_{21}+\kappa_{22}=(k_{0}k_{0z}+k_{1}k_{1z})e^{\ri k_{1z}z^{\prime}%
},\quad\kappa_{21}-\kappa_{22}=(k_{0}k_{0z}-k_{1}k_{1z})e^{-\ri k_{1z}%
z^{\prime}},
\end{equation}
we further have
\begin{equation}%
\begin{cases}
\displaystyle b_{0}=\frac{\ri e^{-\ri(k_{x}x^{\prime}+k_{y}y^{\prime})}}{
k_{0z}}\frac{k_{1}k_{0z}\kappa_{23}}{k_{0}k_{0z}\kappa_{11}-\ri k_{1}%
k_{1z}\kappa_{12}},\\[7pt]%
\displaystyle a_{1}=\frac{\ri e^{-\ri(k_{x}x^{\prime}+k_{y}y^{\prime})}}{
2k_{1z}}\frac{(k_{1}k_{1z}-k_{0}k_{0z})\kappa_{23}e^{\ri dk_{1z}}}{k_{0}%
k_{0z}\kappa_{11}-\ri k_{1}k_{1z}\kappa_{12}},\\[7pt]%
\displaystyle b_{1}=\frac{\ri e^{-\ri(k_{x}x^{\prime}+k_{y}y^{\prime})}%
}{2k_{1z}}\frac{(k_{1}k_{1z}-k_{2}k_{2z})\kappa_{22}e^{\ri dk_{1z}}}%
{k_{0}k_{0z}\kappa_{11}-\ri k_{1}k_{1z}\kappa_{12}},\\[7pt]%
\displaystyle a_{2}=\frac{\ri e^{-\ri(k_{x}x^{\prime}+k_{y}y^{\prime})}%
}{k_{2z}}\frac{ k_{1} k_{2z}\kappa_{22}e^{\ri dk_{1z}}e^{-\ri d k_{2z} }%
}{k_{0}k_{0z}\kappa_{11}-\ri k_{1}k_{1z}\kappa_{12}}.
\end{cases}
\end{equation}
Substuting into \eqref{greenspectralmsour2} and applying inverse Fourier
transform, we have
\begin{equation}
\label{densitythreelayer2}\hspace{-7pt}
\begin{cases}
\displaystyle\{\sigma_{01}^{\uparrow\uparrow}(k_{\rho}), \sigma_{01}%
^{\uparrow\downarrow}(k_{\rho})\}=\frac{k_{1}k_{0z}}{k_{0}k_{0z}\kappa
_{11}-\ri k_{1}k_{1z}\kappa_{12}}\big\{(k_{1}k_{1z}-k_{2}k_{2z})e^{\ri dk_{1z}%
}, k_{1}k_{1z}+k_{2}k_{2z}\},\\[8pt]%
\displaystyle\{\sigma_{11}^{\uparrow\uparrow}(k_{\rho}), \sigma_{11}%
^{\uparrow\downarrow}(k_{\rho})\}=\frac{k_{1}k_{1z}-k_{2}k_{2z}}{k_{0}%
k_{0z}\kappa_{11}-\ri k_{1}k_{1z}\kappa_{12}}\Big\{\frac{k_{1}k_{1z}%
+k_{0}k_{0z}}{2}, \frac{k_{1}k_{1z}-k_{0}k_{0z}}{2}e^{\ri dk_{1z}%
}\Big\},\\[8pt]%
\displaystyle\{\sigma_{11}^{\downarrow\uparrow}(k_{\rho}), \sigma
_{11}^{\downarrow\downarrow}(k_{\rho})\}=\frac{(k_{1}k_{1z}-k_{0}%
k_{0z})e^{\ri dk_{1z}}}{k_{0}k_{0z}\kappa_{11}-\ri k_{1}k_{1z}\kappa_{12}%
}\Big\{\frac{k_{1}k_{1z}-k_{2}k_{2z}}{2}e^{\ri dk_{1z}}, \frac{k_{1}%
k_{1z}+k_{2}k_{2z}}{2}\Big\},\\[8pt]%
\displaystyle\{\sigma_{21}^{\downarrow\uparrow}(k_{\rho}), \sigma
_{21}^{\downarrow\downarrow}(k_{\rho})\}=\frac{k_{1} k_{2z}e^{-\ri dk_{2z}}%
}{k_{0}k_{0z}\kappa_{11}-\ri k_{1}k_{1z}\kappa_{12}}\Big\{k_{1}k_{1z}%
+k_{0}k_{0z}, (k_{1}k_{1z}-k_{0}k_{0z})e^{\ri dk_{1z}}\Big\}.
\end{cases}
\end{equation}

\subsection{Three layers with sources in the bottom layer}

Let $L=2$ with interfaces at $z=0$ and $z=-d<0$. Assume that the source is in
the bottom layer at $(x^{\prime}, y^{\prime}, z^{\prime})$, i.e., $z^{\prime
}<-d$. The domain Green's function has representation
\begin{equation}
\label{greenspectralbsour1}%
\begin{cases}
\displaystyle \widehat u_{0}(k_{x},k_{y}, z)=A_{0}\cosh(\ri k_{0z}%
z)+B_{0}\sinh(\ri k_{0z}z), & z>0,\\[7pt]%
\displaystyle \widehat u_{1}(k_{x},k_{y}, z)=A_{1}\cosh(\ri k_{1z}%
(z+d))+B_{2}\sinh(\ri k_{1z}(z+d)),, & -d<z<0,\\[7pt]%
\displaystyle \widehat u_{2}(k_{x},k_{y}, z)=A_{2}\cosh(\ri k_{2z}%
z)+B_{2}\sinh(\ri k_{2z}z), & z<-d,
\end{cases}
\end{equation}
or equivalently
\begin{equation}
\label{greenspectralbsour2}%
\begin{cases}
\displaystyle \widehat u_{0}(k_{x},k_{y}, z)=b_{0}e^{\ri k_{0z}z}, &
z>0,\\[7pt]%
\displaystyle \widehat u_{1}(k_{x},k_{y}, z)=a_{1}e^{-\ri k_{1z}(z+d)}%
+b_{1}e^{\ri k_{1z}(z+d)}, & -d<z<0,\\[7pt]%
\displaystyle \widehat u_{2}(k_{x},k_{y}, z)=a_{2}e^{-\ri k_{2z}z}%
+\frac{\ri e^{\ri(k_{2z}|z-z^{\prime}|-k_{x}x^{\prime}-k_{y}y^{\prime})}%
}{2k_{2z}}, & z<-d,
\end{cases}
\end{equation}
where
\[
b_{0}=\frac{A_{0}+B_{0}}{2},\quad a_{1}=\frac{A_{1}-B_{1}}{2},\quad
b_{1}=\frac{A_{1}+B_{1}}{2},\quad a_{2}=\frac{A_{2}^{U}-B_{2}^{U}}{2}.
\]
then the solution can be calculated via \eqref{coefficientsolver}, i.e.,
\begin{equation}%
\begin{cases}
\displaystyle A_{0}=B_{0}=\frac{k_{1}k_{1z} k_{2}k_{2z}e^{\ri dk_{1z}}}%
{k_{2}k_{2z}\kappa_{31}-\ri k_{1}k_{1z}\kappa_{32}}\frac{\ri e^{-\ri(k_{x}%
x^{\prime}+k_{y}y^{\prime})}e^{-\ri k_{2z}(d+z^{\prime})}}{ k_{2z}},\\[7pt]%
\displaystyle A_{1}=\frac{k_{2}k_{2z}\kappa_{31}e^{-\ri k_{2z}(d+z^{\prime})}%
}{k_{2}k_{2z}\kappa_{31}-\ri k_{1}k_{1z}\kappa_{32}}\frac{\ri e^{-\ri(k_{x}%
x^{\prime}+k_{y}y^{\prime})}}{k_{2z}},\\[7pt]%
\displaystyle B_{1}=\frac{-\ri k_{2}k_{2z}\kappa_{32}e^{-\ri k_{2z}%
(d+z^{\prime})}}{k_{2}k_{2z}\kappa_{31}-\ri k_{1}k_{1z}\kappa_{32}}%
\frac{\ri e^{-\ri(k_{x}x^{\prime}+k_{y}y^{\prime})}}{k_{2z}},\\[7pt]%
\displaystyle A_{2}=-B_{2}=\frac{k_{2}k_{2z}\kappa_{31}+\ri k_{1}k_{1z}%
\kappa_{32}}{2 (k_{2}k_{2z}\kappa_{31}-\ri k_{1}k_{1z}\kappa_{32})}%
\frac{\ri e^{-\ri(k_{x}x^{\prime}+k_{y}y^{\prime})}e^{-\ri k_{2z}(d+z)}%
}{k_{2z}},
\end{cases}
\end{equation}
where
\begin{equation}%
\begin{split}
&  \kappa_{31}=\frac{k_{1} k_{1z} -k_{0} k_{0z}}{2}e^{\ri 2dk_{1z}}%
+\frac{k_{1} k_{1z} + k_{0} k_{0z}}{2},\\
&  \kappa_{32}=\ri\Big(\frac{k_{0} k_{0z} -k_{1} k_{1z}}{2}e^{\ri 2dk_{1z}}
+\frac{k_{0} k_{0z} +k_{1} k_{1z}}{2}\Big).
\end{split}
\end{equation}
Then
\begin{equation}%
\begin{cases}
\displaystyle b_{0}=\frac{k_{1}k_{1z} k_{2}k_{2z}e^{\ri dk_{1z}}}{k_{2}%
k_{2z}\kappa_{31}+\ri k_{1}k_{1z}\kappa_{32}}\frac{\ri e^{-\ri(k_{x}x^{\prime
}+k_{y}y^{\prime})}e^{-\ri k_{2z}(d+z^{\prime})}}{ k_{2z}},\\[7pt]%
\displaystyle a_{1}=\frac{\ri e^{-\ri(k_{x}x^{\prime}+k_{y}y^{\prime})}%
}{k_{2z}}\frac{k_{2}k_{2z}(k_{1}k_{1z}-k_{0}k_{0z})e^{\ri 2dk_{1z}%
}e^{-\ri k_{2z}(d+z^{\prime})}}{2(k_{2}k_{2z}\kappa_{31}+\ri k_{1}k_{1z}%
\kappa_{32})},\\[7pt]%
\displaystyle b_{1}=\frac{\ri e^{-\ri(k_{x}x^{\prime}+k_{y}y^{\prime})}}{
k_{2z}}\frac{k_{2}k_{2z}(k_{1}k_{1z}+k_{0}k_{0z})e^{-\ri k_{2z}(d+z^{\prime}%
)}}{2 (k_{2}k_{2z}\kappa_{31}+\ri k_{1}k_{1z}\kappa_{32})},\\[7pt]%
\displaystyle a_{2}=\frac{k_{2}k_{2z}\kappa_{31}+\ri k_{1}k_{1z}\kappa_{32}}{2
(k_{2}k_{2z}\kappa_{31}-\ri k_{1}k_{1z}\kappa_{32})}\frac{\ri e^{-\ri(k_{x}%
x^{\prime}+k_{y}y^{\prime})}e^{-\ri dk_{2z}}}{k_{2z}}.
\end{cases}
\end{equation}
Substuting into \eqref{greenspectralbsour2} and applying inverse Fourier
transform, we have
\begin{equation}
\label{densitythreelayer3}%
\begin{cases}
\displaystyle\sigma_{02}^{\uparrow\downarrow}(k_{\rho})=\frac{2k_{1}%
k_{1z}k_{2}k_{0z}e^{\ri dk_{1z}}}{k_{2}k_{2z}\kappa_{31}-\ri k_{1}k_{1z}%
\kappa_{32}},\\[8pt]%
\displaystyle\sigma_{12}^{\uparrow\downarrow}(k_{\rho})=\frac{k_{2}%
k_{1z}(k_{1}k_{1z}+k_{0}k_{0z})}{k_{2}k_{2z}\kappa_{31}-\ri k_{1}k_{1z}%
\kappa_{32}},\\[8pt]%
\displaystyle\sigma_{12}^{\downarrow\downarrow}(k_{\rho})=\frac{k_{2}%
k_{1z}(k_{1}k_{1z}-k_{0}k_{0z})e^{\ri 2dk_{1z}}}{k_{2}k_{2z}\kappa
_{31}-\ri k_{1}k_{1z}\kappa_{32}},\\[8pt]%
\displaystyle\sigma_{22}^{\downarrow\downarrow}(k_{\rho})=\frac{k_{2}%
k_{2z}\kappa_{31}+\ri k_{1}k_{1z}\kappa_{32}}{k_{2}%
k_{2z}\kappa_{31}-\ri k_{1}k_{1z}\kappa_{32}}.
\end{cases}
\end{equation}
It is worthy to point out that
\[
k_{2}k_{2z}\kappa_{31}-\ri k_{1}k_{1z}\kappa_{32}=k_{0}k_{0z}\kappa
_{11}-\ri k_{1}k_{1z}\kappa_{12}.
\]


\section*{Acknowledgement}

This work was supported by US Army Research Office (Grant No.
W911NF-17-1-0368) and US National Science Foundation (Grant No. DMS-1802143).
The authors thank Prof. Johannes Tausch for helpful discussions.

\bibliographystyle{plain}

\end{document}